%% file: paper.tex
\documentclass[11pt]{article}
\usepackage{fullpage}
\usepackage[hyphens]{url}
\usepackage{hyperref}
\hypersetup{
breaklinks=true,
    colorlinks=true, 
    linkcolor=black, 
    citecolor=black, 
    filecolor=black,
    urlcolor=black 
}
\usepackage[utf8]{inputenc} 
\usepackage[T1]{fontenc}    
\usepackage{booktabs}       
\usepackage{amsfonts}        
\usepackage{nicefrac}        
\usepackage{microtype}      
\usepackage{tcolorbox}
\usepackage{diagbox}
\usepackage{enumerate}
\usepackage[shortlabels]{enumitem}
\usepackage{algorithmic}
\usepackage{algorithm}
\usepackage{subfigure}
\usepackage{tabularx}
\usepackage{verbatim}
\usepackage{afterpage}
\usepackage{float}
\usepackage{makecell}  
\usepackage{soul} 
\usepackage{graphicx} 
\usepackage{caption}
\usepackage{amsmath}
\usepackage{mathtools}    
\mathtoolsset{showonlyrefs}  
\usepackage{amsthm}
\usepackage{amssymb}
\usepackage{tikz}
\usetikzlibrary{decorations.pathreplacing, positioning}
\usetikzlibrary{shapes, calc}
\usepackage{multirow}
\usepackage{enumerate}
\usepackage{color}
\usepackage{xcolor}

\newcommand{\citep}[1]{\cite{#1}}
\newcommand{\citet}[1]{\cite{#1}}

\allowdisplaybreaks[4]
\usepackage{mathrsfs}
\usepackage{bm,todonotes}

\usepackage{amsthm}
\newtheorem{theorem}{Theorem}[section]

\newtheorem{thm}{Theorem}[section]
\newtheorem{lem}{Lemma}[section]
\newtheorem{cor}{Corollary}[section]
\newtheorem{prop}{Proposition}[section]
\newtheorem{asmp}{Assumption}[section]
\newtheorem{defn}{Definition}[section]

\newtheorem{fact}{Fact}[section]

\newtheoremstyle{remarkstyle}
  {}                  
  {}                   
  {\normalfont}         
  {}                   
  {\itshape}           
  {.}                   
  { }                   
  {}                   
\theoremstyle{remarkstyle}
\newtheorem{rem}{Remark}[section]

\definecolor{wjs}{RGB}{200,0,50}

\hypersetup{
colorlinks=true,
filecolor=black,
citecolor=blue,
urlcolor=black,
}

\input{tex/math_commands}

\title{Optimal Detection for Language Watermarks with Pseudorandom Collision}

\newcommand\blfootnote[1]{ 
  \begingroup
  \renewcommand\thefootnote{} 
  \footnote{\hspace{-1.2em}#1} 
  \addtocounter{footnote}{-1}
  \endgroup
}

\author{
  {T.\ Tony Cai\textsuperscript{1}}\qquad
  {Xiang Li\textsuperscript{1}}  \qquad
  {Qi Long\textsuperscript{1}} \qquad
 {Weijie J. Su\textsuperscript{1}}  \qquad
  {Garrett G.\ Wen\textsuperscript{2}}\\[1ex]
  \textsuperscript{1}University of Pennsylvania\\
  \textsuperscript{2}Yale University\\[2ex]
}

\date{October 22, 2025}

\begin{document}

\maketitle

\begin{abstract}

Text watermarking plays a crucial role in ensuring the traceability and accountability of large language model (LLM) outputs and mitigating misuse. While promising, most existing methods assume perfect pseudorandomness. In practice, repetition in generated text induces collisions that create structured dependence, compromising Type~I error control and invalidating standard analyses.
 
 We introduce a statistical framework that captures this structure through a hierarchical two-layer partition. At its core is the concept of minimal units—the smallest groups treatable as independent across units while permitting dependence within. Using minimal units, we define a non-asymptotic efficiency measure and cast watermark detection as a minimax hypothesis testing problem.

Applied to Gumbel-max and inverse-transform watermarks, our framework produces closed-form optimal rules. It explains why discarding repeated statistics often improves performance and shows that within-unit dependence must be addressed unless degenerate. Both theory and experiments confirm improved detection power with rigorous Type I error control. These results provide the first principled foundation for watermark detection under imperfect pseudorandomness, offering both theoretical insight and practical guidance for reliable tracing of model outputs.

\blfootnote{Emails:\;\texttt{tcai@wharton.upenn.edu},\;\texttt{\{lx10077,qlong\}@upenn.edu},\;\texttt{suw@wharton.upenn.edu},\;\texttt{gang.wen@yale.edu}.}
\blfootnote{Author names are listed in alphabetical order.}

\end{abstract}

\input{intro}

\input{prelim}
\input{framework}
\input{results}

\input{experiment}

\input{proof}
\input{discuss}

\section*{Acknowledgments}
This work was supported in part by NIH grants U01CA274576, and R01EB036016, NSF grant DMS-2310679, a Meta Faculty Research Award, and Wharton AI for Business. The content is solely the responsibility of the authors and does not necessarily represent the official views of the NIH.

{
\bibliographystyle{abbrv}
\bibliography{bib/chatgpt,bib/privacy,bib/stat}
} 

\appendix
\input{tex/appendix}

\end{document}

%% file: tex/math_commands.tex
\usepackage{amsmath,amsfonts,bm}

\def\1{\bm{1}}

\def\eps{{\varepsilon}}

\def\bDelta{{\bar{\Delta}}}

\def\rd{{\textnormal{d}}}
\def\re{{\textnormal{e}}}

\def\ri{{\textnormal{i}}}

\def\ri{{\textnormal{i}}}

\def\vx{{\bm{x}}}
\def\vy{{\bm{y}}}

\DeclareMathAlphabet{\mathsfit}{\encodingdefault}{\sfdefault}{m}{sl}
\SetMathAlphabet{\mathsfit}{bold}{\encodingdefault}{\sfdefault}{bx}{n}

\def\sS{{\mathbb{S}}}

\renewcommand{\xi}{\zeta}

\def\sS{\boldsymbol{S}}

\def\0{{\bf 0}}
\def\1{{\bf 1}}

\def\AM{{\mathcal A}}

\def\DM{{\mathcal D}}

\def\EM{{\mathcal E}}

\def\IM{{\mathcal I}}
\def\HM{{\mathcal H}}

\def\KM{{\mathcal K}}

\def\NM{{\mathcal N}}

\def\PM{{\mathcal P}}
\def\QM{{\mathcal Q}}
\def\FPM{{\mathcal P}_{\Delta}}

\def\SM{{\mathcal S}}

\def\UM{{U}}
\def\VM{{\mathcal V}}
\def\WM{{\mathcal W}}

\def\RB{{\mathbb R}}

\DeclareMathOperator{\EB}{\mathbb{E}}

\def\PB{{\mathbb P}}

\def\bx{{\boldsymbol{x}}}
\def\by{{\boldsymbol{y}}}
\def\bh{{\boldsymbol{h}}}

\def\FPM{\PM_{\Delta}}
\def\type{\mathrm{type}}

\def\SPM{{\overline{\PM}_{\Delta}}}

\newcommand{\Var}{\mathrm{Var}}

\def\rd{{\mathrm{d}}}
\def\re{{\mathrm{e}}}

\def\Key{{\mathtt{Key}}}

\newcommand{\Pow}{\mathscr{P}}

\def\token{{w}}

\def\Voca{{\WM}}
\newcommand\bP{\bm{P}}
\newcommand\bQ{\bm{Q}}

\newcommand\bS{{\bm{S}}}

\def\SMmax{\SM^{\mathrm{gum}}}
\def\SMinv{\SM^{\mathrm{inv}}}

\def\Ydif{Y^{\mathrm{inv}}}
\def\Yars{Y^{\mathrm{gum}}}

\def\hars{h_{\mathrm{ars}}}
\def\hlog{h_{\mathrm{log}}}

\def\hneg{h_{\mathrm{neg}}}
\def\hlog{h_{\mathrm{log}}}

\def\hoptars{{h_{\mathrm{gum}, \Delta}}}

\def\hoptdif{{h_{\mathrm{dif}, \Delta}}}

\def\Perm{\mathrm{Perm}}

\DeclareMathOperator{\sign}{sign}

%% file: intro.tex
\section{Introduction}\label{sec:intro} 

Recent advances in generative artificial intelligence have profoundly transformed the creation and consumption of digital content. Systems capable of generating human-like text, images, and audio are now widely accessible, with large language models (LLMs) being particularly influential \citep{openai2023,liu2024deepseek}. The ability of LLMs to produce fluent text at scale enables powerful applications, from creative writing to automated code generation. However, this proliferation also precipitates pressing concerns over provenance and authenticity. In high-stakes domains such as education, journalism, and scientific research, the misattribution of AI-generated content can have severe consequences, including undermining academic integrity, eroding public trust, and compromising research reproducibility \citep{stokel2022ai, milano2023large, zellers2019defending, starbird2019disinformation, radford2023robust, das2024under}. This landscape highlights an urgent need for reliable methods to distinguish between human-written and machine-generated text.

While many detection methods rely on identifying linguistic artifacts, a more principled and statistical approach is LLM watermarking, which has seen internal implementation by OpenAI and Google DeepMind \citep{scott2023watermarking,dathathri2024scalable}. 
This technique embeds a verifiable statistical signal into the text generation process using pseudorandom variables derived from a secret cryptographic key \citep{kuditipudi2023robust, kirchenbauer2023reliability}. In effect, the key initializes a pseudorandom generator that governs how texts are generated, thereby creating a hidden statistical dependence between the generated text and the key. This dependence enables rigorous hypothesis testing for provable detection \citep{li2024statistical, li2024robust}. In a typical implementation, a provider deploys a watermarked LLM. A user, such as a student, interacts with the model to produce a text. A verifier, such as a teacher, who has been granted access to the cryptographic key, can then analyze the text to determine if it was generated by the watermarked model.

To formalize the watermarking mechanism, it is instructive to first recognize that LLMs sequentially generate a token in a probabilistical manner.\footnote{Here, a token represents a word, subword, or punctuation. For example, the sentence ``Hello, world!'' can be tokenized into four tokens: [``Hello'', ``,'', `` world'', ``!'']]. See \url{https://platform.openai.com/tokenizer} for examples.} To produce the $t$-th token, denoted by $\token_t$, the model first computes a next-token prediction (NTP) distribution $\bP_t$ over its vocabulary based on the preceding tokens $\token_{1:(t-1)} := \token_1 \cdots \token_{t-1}$. For a watermarked LLM, the sampling of $\token_t$ from the NTP distribution $\bP_t$ is governed by a pseudorandom variable $\xi_t$, which is typically generated by a cryptographic hash function $\AM$ that takes a private $\Key$ and the recent context window $\token_{(t-m):(t-1)}$ as input. While the resulting token $\token_t$ still marginally follows the original distribution $\bP_t$, its realization is now tied to $\xi_t$. Consequently, while the marginal distributions of the tokens may be indistinguishable from unwatermarked text, their joint distribution with the pseudorandom variables is not. Without a watermark, the tokens and pseudorandom variables are statistically independent, while with a watermark, they become dependent. This induced dependence is the statistical underpinning for detection, whereby a verifier reconstructs the sequence of pseudorandom variables $\xi_1, \dots, \xi_n$ and constructs a test statistic to capture their association with the observed text.

Two of the most commonly used watermarking schemes are the Gumbel-max watermark \citep{scott2023watermarking} and the inverse-transform watermark \citep{kuditipudi2023robust}.
Both, along with most existing watermarking schemes, are theoretically grounded in a fundamental assumption that the pseudorandom variables $\xi_t = \AM(\token_{(t-m):(t-1)}, \Key)$ are independent and identically distributed (i.i.d.) for $t = m+1, \dots, n$.
This assumption is justified when the context window $\token_{(t-m):(t-1)}$ is unique for every position $t$, since the cryptographic design of the hash function ensures that its outputs behave as independent uniform draws.\footnote{The hash function is sensitive to its inputs, that is, $\AM(\token_{(t-m):(t-1)}, \Key)$ is independent of $\AM(\token_{(t'-m):(t'-1)}, \Key)$ whenever the text windows differ, $\token_{(t-m):(t-1)} \neq \token_{(t'-m):(t'-1)}$, for $t \neq t'$, as $\Key$ is randomly selected.} In practice, however, language is inherently repetitive, particularly in specialized domains like programming and mathematical writing \citep{he2024empirical}. When a segment of text repeats such that $\token_{(t-m):(t-1)} = \token_{(t'-m):(t'-1)}$ for some $t \neq t'$, the deterministic nature of the hash function forces $\xi_t = \xi_{t'}$. This phenomenon, which is known as \textit{pseudorandom collision} \citep{wu2024distortion}, is surprisingly common. It typically becomes more frequent when the LLM generation is relatively deterministic (e.g., during code generation or list completion, where the entropy of the NTP distributions is low) or when the context window size $m$ is small (see the left panel in Figure \ref{fig:repetition}). Importantly, collisions are not merely implementation artifacts but an intrinsic feature of language. They cannot be eliminated entirely, as even human-written documents naturally contain repeated phrases (see Table \ref{tab:repetition} for examples from classic works of literature).

Unfortunately, pseudorandom collisions fundamentally violate the independence assumption that underpins recent statistical frameworks for watermark detection \citep{li2024statistical,li2024robust} and estimation \citep{li2025optimal}. Since these methods all rely on token-level independence of pivotal statistics, collisions make this assumption fail and render the guarantees unreliable.
Without this independence, not only are power analyses invalidated, but more critically, \textit{even} Type I error control is no longer guaranteed (see the right panel in Figure \ref{fig:repetition}). Among many challenges, one lies in the fact that collisions can occur anywhere within the text, leading to complex and unpredictable dependence structures.

While heuristic fixes have been proposed \citep{fernandez2023three, wu2024distortion, dathathri2024scalable}, a systematic statistical analysis is still largely absent. This presents a pressing statistical challenge to the reliable detection of LLM watermarks and calls into question both the framework and the optimality of detection rules derived under the i.i.d.\ assumption. Consequently, comparisons between different watermarking schemes that neglect pseudorandom collisions cannot be considered trustworthy. It thus leads to a central question: can we establish a new framework and design provably optimal detection rules in the presence of imperfect pseudorandomness?

\begin{figure}[!t]
\vspace{-0.2in}
\centering
\includegraphics[width=\textwidth]{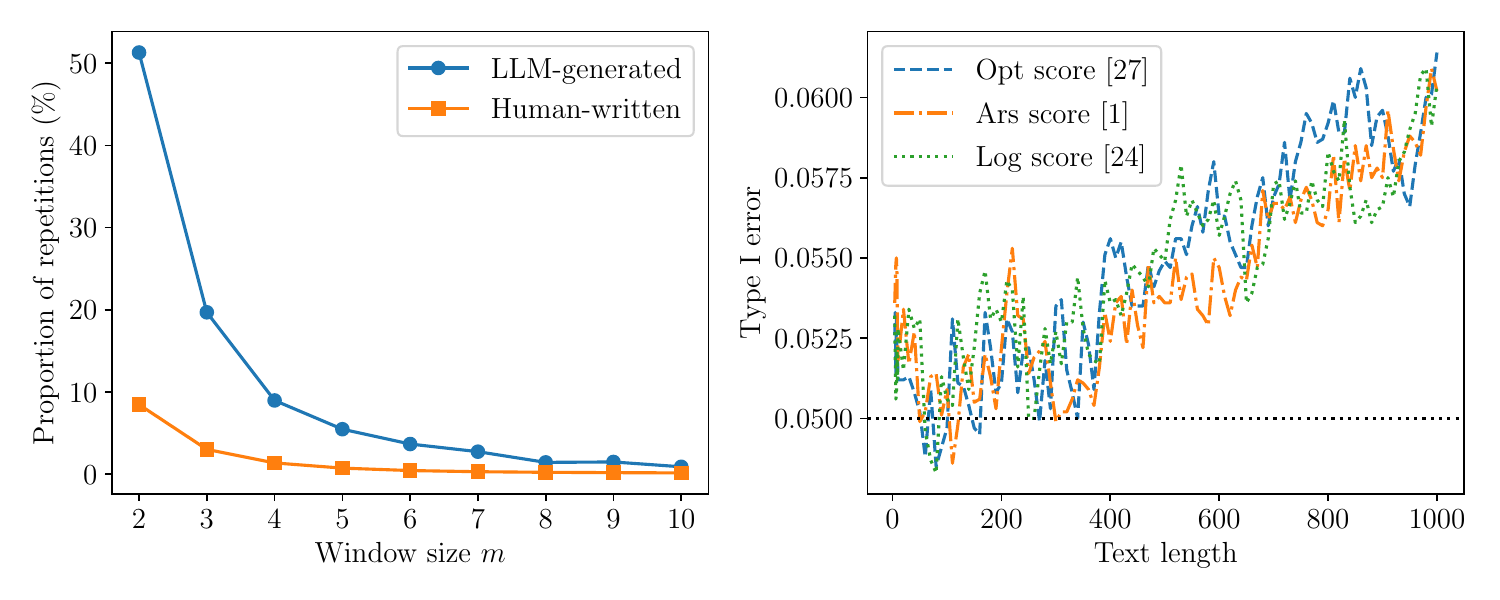}
\vspace{-0.3in}
\caption{\textbf{Left}: Fraction of repeated segments across different text window sizes $m$ for the OPT-1.3B model \citep{zhang2022opt} on the C4 news-like dataset \citep{raffel2020exploring}.
\textbf{Right}: Inflation of type I error when the null human-written data contains repetition, evaluated at significance level $\alpha = 0.05$.}
\label{fig:repetition}
\end{figure}

\sethlcolor{lightgray}	
\begin{table}[t!]
	\centering
	\resizebox{\textwidth}{!}{
		\begin{tabular}{p{0.2\linewidth} | p{0.75\linewidth}}
			\toprule
			\multicolumn{1}{c|}{\textbf{Context}}  & \multicolumn{1}{c}{\textbf{Repetitive Phrases}} \\
			\midrule
			\multicolumn{1}{c|}{
				\makecell[t]{Emphasis \\ from \textit{The Great Gatsby} \\ by F. Scott Fitzgerald}}
			& Gatsby turned sharply. ``Can't {repeat the past}?... \hl{Why of course you can!} \hl{Why of course you can!}'' He looked around him wildly, as if the past were about to rise before his eyes. \hfill (6 tokens) \\
			\hline
			\multicolumn{1}{c|}{
				\makecell[t]{Reassurance \\ from \textit{Moby-Dick} \\ by Herman Melville}}
			& And even in the whaleboat, in the stormiest gales, in the maddest tossing of the waves, the shouts of ``\hl{All's well!} \hl{All's well!}'' came to me across the water. \hfill (4 tokens) \\
			\hline
			\multicolumn{1}{c|}{
				\makecell[t]{Persuasion \\ from \textit{Julius Caesar} \\ by William Shakespeare}}
			& Antony, addressing the crowd after Caesar's death:  
			``He was my friend, faithful and just to me:  
			But \hl{Brutus says he was ambitious;}  
			\hl{And Brutus is an honourable man.}  
			\ldots  
			Yet \hl{Brutus says he was ambitious;}  
			\hl{And Brutus is an honourable man.}  
			\ldots'' \hfill (7 and 9 tokens) \\
			\hline
			\multicolumn{1}{c|}{
				\makecell[t]{Urging \\ from \textit{White Fang} \\ by Jack London}}
			& \hl{Fight!} \hl{Fight!} \hl{Fight!} That was it---the inexorable and eternal decree\ldots{} the urge of life, the tidal wave of life, surging upward, beating\hl{ in him}, pounding\hl{ in him}, driving him resistlessly on. \hfill (2 tokens) \\
			\bottomrule
	\end{tabular}}
	\vspace{-0.05in}
	\caption{Examples of natural repetition in literary works. Token counts are computed using the GPT-4o tokenizer (\url{https://platform.openai.com/tokenizer}).}
	\label{tab:repetition}
	\vspace{-0.1in}
\end{table}

\subsection{Our Contributions}

To address this challenge, we develop a new framework for watermark detection that explicitly accounts for pseudorandomness collisions. This framework still builds on the pivotal-statistic approach of \citet{li2024statistical}, but differs by carefully capturing how text repetition affects the joint distribution of the pivotal statistics $Y_{1:n}$.

When no text windows repeat, the pseudorandom variables $\zeta_{1:n}$ can be safely treated as i.i.d., and by the pivotal property, $Y_{1:n}$ is also i.i.d. With repetitions, however, some pseudorandom variables and pivotal statistics become identical: if two positions $t \neq t'$ share the same context window, then $\zeta_t = \zeta_{t'}$, in which case we also have $Y_t = Y_{t'}$ whenever $\token_t = \token_{t'}$. Such collisions at the $\zeta$-level and coincidences at the token level induce structured dependence in $Y_{1:n}$. To systematically capture this dependence, we introduce a hierarchical framework built on a two-level partition of pseudorandom variables and pivotal statistics. At the first level, $\zeta_{1:n}$ are grouped into blocks reflecting pseudorandom collisions, while, at the second, $Y_{1:n}$ are further divided into sub-blocks accounting for both pseudorandom collisions and token coincidences.

This two-level structure provides a refined basis for analysis. Within this framework, the detection problem reduces to testing distributional differences in $Y_{1:n}$ conditioned on the observed two-level partitions, with a formal formulation presented in \eqref{eq:structured_test}. This formulation serves as the foundation for developing provably optimal detection rules and sets the stage for our contributions below.

\begin{enumerate}
\item[]
\textbf{A hierarchical framework of LLM watermarks.} 
We propose a statistical framework for watermark detection that explicitly accounts for text repetition through the hierarchical two-layer partition. This partition captures the dependence among pivotal statistics and allows their joint distribution to be characterized without any information loss. Within this structure, we find that the pivotal statistics can be partitioned into disjoint subsets, which we call \textit{minimal units}, that are mutually independent across units though not independent within each unit. Taking minimal units as the basic analytic objects, we introduce a new non-asymptotic efficiency notion that quantifies least-favorable detection power when NTP distributions lie in a belief class, casting the search for optimal rules as a minimax problem. Finally, we develop a general non-i.i.d. large-deviation bound under verifiable conditions, which provides a tight characterization of this efficiency notion (see Remark \ref{rem:tight-lower-bound}). This framework is formally introduced in Section \ref{sec:framework}.

\item[]
\textbf{Application to the Gumbel-max watermark.}
We apply our framework to the Gumbel-max watermark in Section~\ref{sec:gumbel} and analyze the associated minimax problem of maximizing the efficiency notion. We find that a saddle-point pair---consisting of an optimal detection rule and the corresponding least-favorable distribution---does not always exist. When it does, we derive closed-form expressions; when it does not, we characterize the transition boundaries. Notably, the optimally derived rule reduces to discarding all repeated pivotal statistics in $Y_{1:n}$, a form that resonates with empirical heuristics proposed in \citet{fernandez2023three,wu2024distortion,dathathri2024scalable}. Our optimal rule rigorously controls Type I error and achieves detection power comparable to, and in some cases exceeding, existing methods, as shown in numerical experiments.

However, deriving these optimal rules is more challenging than in prior work \citep{li2024statistical}. While both frameworks maximize least-favorable detection power over a class of NTP distributions, theirs operates at the token level, whereas ours must operate on minimal units within the hierarchical partition. This shift renders the minimax problem highly non-convex, as it requires accounting for all NTP distributions within a unit rather than a single token-level distribution. To tackle this difficulty, we develop new analytical tools based on Schur-convexity and geometric arguments, which resolve the optimality issues in this non-convex setting and may be of independent interest.

\item[]
\textbf{Application to the inverse transform watermark.}
Finally, we apply our framework to the inverse transform watermark in Section~\ref{sec:inverse}. This case poses unique analytical challenges, as the joint distribution involves exponentially many terms and is intractable in finite form. We show that as the vocabulary size grows, the distribution converges to a simpler asymptotic limit, which makes the minimax problem tractable and yields a closed-form optimal detection rule. Our analysis further reveals that, while discarding repeated pivotal statistics remains harmless, optimal rules must still account for the dependence among statistics within each minimal unit, since they share the same pseudorandom variables. Numerical experiments corroborate these results, showing comparable detection power while maintaining rigorous Type I error control.
\end{enumerate}

\subsection{Related Work}

Since the introduction of text watermarking for LLMs \citep{kirchenbauer2023watermark,scott2023watermarking}, text repetition has been widely observed. Such repetition---often caused by relatively deterministic generation or small context windows \citep{fernandez2023three,kuditipudi2023robust}---induces pseudorandom collisions. Prior analysis frameworks \citep{li2024statistical,li2024robust,zhao2024permute} and downstream estimation tasks \citep{li2025optimal} overlook this issue by assuming perfect pseudorandomness, where all pivotal statistics are assumed to be i.i.d. In practice, collisions introduce strong dependencies, since repeated contexts force correlation or even identity among pivotal statistics. As a result, empirical Type~I error can be severely inflated, far beyond the nominal level \citep{fernandez2023three,wu2024distortion}, undermining the reliability of the watermark. While some studies note that mild repetition can occasionally improve power or robustness in goodness-of-fit tests \citep{he2024empirical}, this benefit comes at the cost of uncontrolled Type~I error, making repetition generally undesirable.
To address this issue, we develop a new formulation and analysis techniques that explicitly account for the dependence induced by pseudorandom collisions. As a consequence, our framework not only resolves this fundamental issue but also explains why a common empirical fix---discarding repeated pivotal statistics and applying detection rules only to the unique ones \citep{fernandez2023three,wu2024distortion,dathathri2024scalable}---is information-theoretically justified, as it matches the structure of the optimal detection rule.

From a statistical standpoint, the collision-induced dependence structure presents a novel challenge. Classical goodness-of-fit tests \citep{donoho2004higher,caiOptimalDetectionHeterogeneous2011,caiOptimalDetectionSparse2014} typically assume i.i.d.\ samples under both the null and alternative hypotheses, whereas our problem involves a non-i.i.d.\ setting where the dependence structure is captured by the hierarchical
two-layer partition. Unlike traditional cases (such as serial correlation in time series \citep{brockwell2002introduction,shumway2006time} or within-subject dependence in longitudinal data \citep{diggle2002analysis,fitzmaurice2012applied}) where dependence takes the form of partial correlation and each observation still contributes new information, our setting exhibits a more extreme structure: some pivotal statistics are exact duplicates due to collisions, while others are intricately linked through shared pseudorandom variables. These overlaps fall outside existing frameworks, and our work offers the first formal treatment of hypothesis testing under this collision-driven dependence. In pursuing optimal detection rules, our strategy connects to the classical literature on robust hypothesis testing \citep{huber1973minimax,veeravalli2002minimax,fauss2021minimax}, which also seeks detectors optimized against least-favorable distributions from a belief class. The key difference is that our setting is considerably more complex: saddle-point solutions may fail to exist, whereas in classical formulations they typically do, due to the simplicity of their model and problem setup.

%% file: prelim.tex
\section{Preliminaries}\label{sec:prelim}
 
\paragraph{Watermarking embedding and detection.}
At a high level, watermarking modifies text generation by coupling each token with a recoverable pseudorandom variable, often referred to as a random seed in \citet{dathathri2024scalable}. Concretely, rather than drawing the $t$-th token directly from the model’s next-token-prediction (NTP) distribution $\bP_t = (P_{t,w})_{w \in \Voca}$, the process first generates a pseudorandom variable $\zeta_t = \AM(\token_{(t-m):(t-1)}, \Key)$, where $\AM$ is a cryptographic hash function applied to the preceding context window $\token_{(t-m):(t-1)}$ together with a secret $\Key$. The token is then produced by a decoding function $\token_t = \SM(\bP_t, \zeta_t)$, which links $\bP_t$ and $\zeta_t$ in a deterministic way. The sequence $\zeta_{1:n} := \zeta_1 \ldots \zeta_n$ is typically modeled as i.i.d., a valid assumption only when every length-$m$ context prefix is unique \citep{barak2021book, schneier1996applied}. In this work, we focus on \emph{unbiased} decoders, which preserve the marginal distribution in the sense that $\PB_{\zeta}(\SM(\bP, \zeta) = w) = P_w$. In this way, watermarking does not degrade text quality.

To detect the watermark, a verifier reconstructs the sequence $\zeta_{1:n}$ and tests for the statistical dependence between each $\token_t$ and $\zeta_t$. This is formalized using a \emph{pivotal statistic} $Y_t = Y(w_t, \zeta_t)$ \citep{li2024statistical}. Under the null hypothesis $H_0$ (human-written text), $w_t$ and $\zeta_t$ are independent, by the pivotal property, $Y_t$ follows a fixed null distribution denoted by $\mu_0$, regardless of the distribution of $w_t$. Under the alternative $H_1$ (watermarked text), the induced dependence shifts its distribution to an alternative $\mu_{1,\bP_t}$, which depends on $\bP_t$ since in this case $Y_t$ takes the form $Y_t = Y(\SM(\bP_t, \zeta_t), \zeta_t)$. In this way, \citet{li2024statistical,li2024robust} formulate detection as the hypothesis testing problem:
\begin{equation}\label{eq:previous_test}
H_0: Y_t \sim \mu_0 \text{ i.i.d.}, \quad t = 1,\dots,n 
\qquad \text{vs.} \qquad 
H_1: Y_t \sim \mu_{1,\bP_t}, \quad t = 1,\dots,n.
\end{equation}
The standard detection approach, which aggregates scores $h(Y_t)$, relies on the i.i.d. property of the sequence $\{\zeta_t\}_{t=1}^n$. In practice, however, text repetition leads to hash collisions (that is, $\zeta_t = \zeta_{t'}$ for some $t \neq t'$), violating this core assumption. This breakdown of independence for the pivotal statistics $\{Y_t\}_{t=1}^n$ motivates the framework developed in this paper.

\paragraph{Gumbel-max watermark.}
The Gumbel-max watermark \citep{scott2023watermarking} is the most influential unbiased watermarking scheme and has seen widespread adoption in research \citep{openai2024understanding}. It builds on the classical Gumbel-max technique \citep{gumbel1948statistical, papandreou2011perturb}, which samples from a distribution $\bP = (P_w)_{w \in \Voca}$ by drawing $U_w \sim \mathrm{Unif}(0,1)$ independently for each $w \in \Voca$ and selecting
\[
\SMmax(\bP, \xi) := \arg\max_{w \in \Voca} \frac{\log U_w}{P_w}, \quad \text{where }\quad   \xi = (U_w)_{w \in \Voca}.
\]
This decoder is unbiased by construction \citep{li2024statistical}. The associated pivotal statistic is $Y_t = U_{t,\token_t}$, which is uniformly distributed on $(0,1)$ when the text is human-written (that is, $H_0$), but becomes stochastically larger under watermarking (that is, $H_1$) due to the watermark-induced alignment. Detection procedures exploit this shift by aggregating scores $\sum_{t=1}^n h(Y_t)$ and declaring watermarking when the sum exceeds a threshold. In practice, effective score functions are those whose expectations are larger under $H_1$ than under $H_0$. Common choices include $\hars(y) = -\log(1 - y)$ \citep{scott2023watermarking}, $\hlog(y) = \log y$ \citep{kuditipudi2023robust}, and the optimal $\hoptars$ from \citet{li2024statistical}, which depends on a user-specified parameter $\Delta \in (0,1)$.

\paragraph{Inverse transform watermark.}
 
An alternative unbiased scheme is the inverse transform watermark of \citet{kuditipudi2023robust}, which uses inverse transform sampling for unbiased token generation. To produce a token $\token$, the scheme first generates a random permutation of the vocabulary, denoted by $\pi$, together with a uniform draw $U \sim \mathrm{Unif}(0,1)$, and combines them as $\xi = (U, \pi)$. The token is then chosen via
\[
\SMinv(\bP, \xi) = \pi^{-1}(F^{-1}(U; \pi)), 
\quad \text{where } \quad 
F(x; \pi) = \sum_{\token' \in \Voca} P_{\token'} \cdot \mathbf{1}\{\pi(\token') \le x\},
\]
and $F^{-1}(u; \pi) = \min\{x : F(x; \pi) \ge u\}$ is the generalized inverse of $F(x; \pi)$.  

The corresponding pivotal statistic is $\Ydif_t = |\eta(\pi_t(w_t)) - U_t|$, with $\eta(w) = (w - 1)/(|\Voca| - 1)$ mapping token indices to $[0,1]$. Under human-written text ($H_0$), $\Ydif_t$ is approximately distributed as $|U - U'|$ for two independent $U, U' \sim \mathrm{Unif}(0,1)$, giving rise to a triangular distribution. Under watermarking ($H_1$), it concentrates near zero due to alignment. As in the Gumbel-max case, detection exploits this shift through score functions. Typical examples include $\hneg(y) = -y$ and the optimal $\hoptdif$ from \citet{li2024statistical}, also parameterized by a user-specified parameter $\Delta \in (0,1)$.

%% file: framework.tex
\section{A Statistical Framework under Pseudorandomness Collision}
\label{sec:framework}

This section introduces our statistical framework for watermark detection under pseudorandomness collisions. We begin in Section~\ref{sec:independence} with the two-layer partition structure that models the induced dependence, then in Section~\ref{sec:problem-formulation} formalize the detection problem, and finally in Section~\ref{sec:optimal} define an efficiency notion that enables a minimax characterization of optimal detection rules.

\subsection{Structural Dependence and Distribution Factorization}
\label{sec:independence}

Text repetition induces repeated pseudorandom variables and, in turn, repeated pivotal statistics. Specifically, under the hash rule $\zeta_t = \AM(w_{(t-m):(t-1)}, \Key)$, if two context windows satisfy $w_{(t-m):(t-1)} = w_{(t'-m):(t'-1)}$ for $t \ne t'$, then $\zeta_t = \zeta_{t'}$.
Moreover, if $w_t = w_{t'}$ as well, then by the definition $Y_t = Y(w_t, \zeta_t)$, it follows that $Y_t = Y_{t'}$.
We formalize this dependence structure via a two-level partition of the index set $\mathcal{I} = \{1, 2, \dots, n\}$.

\paragraph{Two-level partitions.}

The first partition focuses on pseudorandom variables.

\begin{defn}[$\zeta$-level partition]\label{def:zeta_partition}
The $\zeta$-level partition is defined as $\Pi_{\zeta} := \{\IM_k^{\zeta}\}_{k=1}^K = \{\IM_1^{\zeta}, \ldots, \IM_{K}^{\zeta}\}$, where each block $\IM_k^{\zeta} \subset \mathcal{I}$ satisfies:
\begin{itemize}
\item[(i)] All indices in $\IM_k^{\zeta}$ share the same pseudorandom variable: $\zeta_i = \zeta_j$ for all $i, j \in \IM_k^{\zeta}$, while distinct blocks correspond to distinct values: $\zeta_i \ne \zeta_j$ for $i \in \IM_k^{\zeta}, j \in \IM_{k'}^{\zeta}$ with $k \ne k'$. 
\item[(ii)] The blocks form a disjoint partition of $\mathcal{I}$: $\bigcup_{k=1}^{K} \IM_k^{\zeta} = \mathcal{I}$ and $\IM_k^{\zeta} \cap \IM_{k'}^{\zeta} = \varnothing$ for $k \ne k'$.
\end{itemize}
\end{defn}

Each $\zeta$-block is further refined based on whether the pivotal statistics coincide.

\begin{defn}[$Y$-level partition]\label{def:y_partition}
For each block $\IM_k^{\zeta}$, the corresponding $Y$-level partition is defined as $\Pi_{Y}^{(k)} =\{\IM_{k,l}^{Y}\}_{l=1}^{m_k} =\{\IM_{k,1}^{Y}, \ldots, \IM_{k,m_k}^{Y}\}$, where each sub-block $\IM_{k,l}^{Y} \subset \IM_k^{\zeta}$ satisfies:
\begin{itemize}
\item[(i)] All indices in $\IM_{k,l}^{Y}$ share the same pivotal statistic: $Y_i = Y_j$ for all $i, j \in \IM_{k,l}^{Y}$, while distinct sub-blocks correspond to distinct values: $Y_i \ne Y_j$ for $i \in \IM_{k,l}^{Y}, j \in \IM_{k,l'}^{Y}$ with $l \ne l'$.
\item[(ii)] The sub-blocks form a disjoint partition of $\IM_k^{\zeta}$: $\bigcup_{l=1}^{m_k} \IM_{k,l}^{Y} = \IM_k^{\zeta}$ and $\IM_{k,l}^{Y} \cap \IM_{k,l'}^{Y} = \varnothing$ for $l \ne l'$.
\end{itemize}
\end{defn}

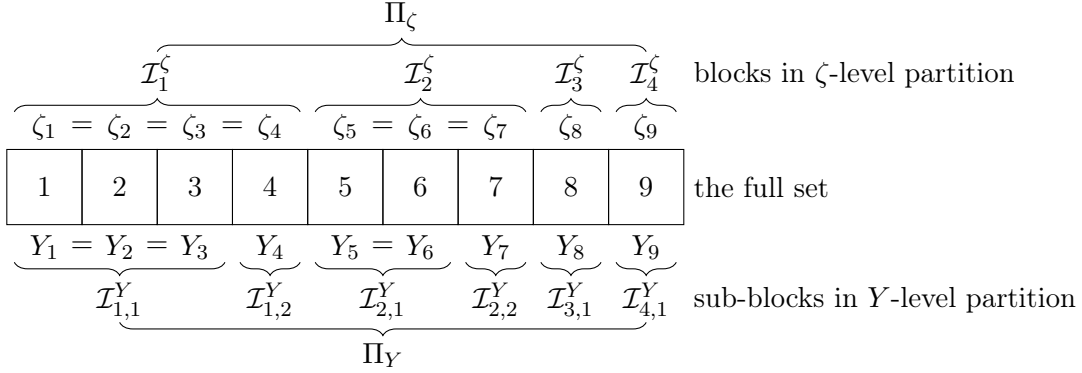
\begin{figure}[!t]
\centering
\begin{tikzpicture}[every node/.style={font=\normalsize}]
 
    \foreach \i in {1,...,9} {
        \draw (\i,0) rectangle ++(1,1);
        \node at (\i+0.5,0.5) {\i};
    }

    \foreach \i in {1,...,9} {
      \node at (\i + 0.5, 1.3) {$\zeta_{\i}$};
    }
    \node at (2,1.3) {$=$};
    \node at (3,1.3) {$=$};
    \node at (4,1.3) {$=$};
    \node at (6,1.3) {$=$};
    \node at (7,1.3) {$=$};

    \draw[decorate,decoration={brace,amplitude=5pt}] (1.1,1.5) -- (4.9,1.5) node[midway,above=4pt] {$\IM_1^{\zeta}$};
    \draw[decorate,decoration={brace,amplitude=5pt}] (5.1,1.5) -- (7.9,1.5) node[midway,above=4pt] {$\IM_2^{\zeta}$};
    \draw[decorate,decoration={brace,amplitude=5pt}] (8.1,1.5) -- (8.9,1.5) node[midway,above=4pt] {$\IM_3^{\zeta}$};
    \draw[decorate,decoration={brace,amplitude=5pt}] (9.1,1.5) -- (9.9,1.5) node[midway,above=4pt] {$\IM_4^{\zeta}$};
    \draw[decorate,decoration={brace,amplitude=5pt}] (3,2.3) -- (9.5,2.3) node[midway,above=4pt] {$\Pi_{\zeta}$};
 
    \foreach \i in {1,...,9} {
      \node at (\i + 0.5, -0.3) {$Y_{\i}$};
    }
    \node at (2,-.3) {$=$};
    \node at (3,-.3) {$=$};
    \node at (6,-.3) {$=$};

    \draw[decorate,decoration={brace,amplitude=5pt,mirror}] (1.1,-0.5) -- (3.9,-0.5) node[midway,below=4pt] {$\IM_{1,1}^Y$};
    \draw[decorate,decoration={brace,amplitude=5pt,mirror}] (4.1,-0.5) -- (4.9,-0.5) node[midway,below=4pt] {$\IM_{1,2}^Y$};
    \draw[decorate,decoration={brace,amplitude=5pt,mirror}] (5.1,-0.5) -- (6.9,-0.5) node[midway,below=4pt] {$\IM_{2,1}^Y$};
    \draw[decorate,decoration={brace,amplitude=5pt,mirror}] (7.1,-0.5) -- (7.9,-0.5) node[midway,below=4pt] {$\IM_{2,2}^Y$};
    \draw[decorate,decoration={brace,amplitude=5pt,mirror}] (8.1,-0.5) -- (8.9,-0.5) node[midway,below=4pt] {$\IM_{3,1}^Y$};
    \draw[decorate,decoration={brace,amplitude=5pt,mirror}] (9.1,-0.5) -- (9.9,-0.5) node[midway,below=4pt] {$\IM_{4,1}^Y$};

    \node[align=left, anchor=west] at (10,2) {blocks in $\zeta$-level partition};
    \node[align=left, anchor=west] at (10,0.5) {the full set};

    \draw[decorate,decoration={brace,amplitude=5pt,mirror}] (2.5,-1.3) -- (9.5,-1.3) node[midway,below=4pt] {$\Pi_Y$};
    \node[align=left, anchor=west] at (10,-1) {sub-blocks in $Y$-level partition};
\end{tikzpicture}
\vspace{-5pt}
\caption{Illustration of the two-level partition structure in a $9$-length sequence: the $\zeta$-level partition $\Pi_\zeta$ groups indices with the same pseudorandom variable, while the $Y$-level partition $\Pi_Y$ further groups them by shared pivotal statistic.}
\label{fig:two-layer-partition}
\end{figure}

An example of the two-layer partition is shown in Figure~\ref{fig:two-layer-partition}. While this structure captures the dependencies caused by repeated context windows, it also implies where conditional independence can still hold. In particular, pseudorandom variables associated with different blocks can be safely treated as independent (see Assumption~\ref{asmp:blocks-independence} for the formal statement). This independence follows from the input sensitivity of cryptographic hash functions: when the input contexts differ, the resulting pseudorandom outputs---$\AM(\token_{(t-m):(t-1)}, \Key)$ and $\AM(\token_{(t'-m):(t'-1)}, \Key)$---are statistically independent~\citep{wu2024distortion}.

\begin{asmp}[Independence across blocks]
\label{asmp:blocks-independence}
For $k \neq k'$ and any $i \in \IM_k^{\zeta}$ and $j \in \IM_{k'}^{\zeta}$,  $\zeta_i$ is statistically independent of $\zeta_j$, denoted as $\zeta_i \perp \zeta_j$.
\end{asmp}

\begin{cor}
Under Assumption \ref{asmp:blocks-independence}, for $k \neq k'$ and any indices $i \in \IM_k^{\zeta}$ and $j \in \IM_{k'}^{\zeta}$,  $Y_i$ is statistically independent of $Y_j$, denoted as $Y_i \perp Y_j$.
\end{cor}

In some cases, a finer level of independence holds between sub-blocks (see Assumption~\ref{asmp:subblocks-independence}). Recall that each $Y_t = Y(\zeta_t, w_t)$ is a deterministic function of both $\zeta_t$ and $w_t$. Since $\zeta_t$ is constant within each sub-block, this finer independence requires that the function $w \mapsto Y(\zeta_t, w)$ induces variability across tokens. Whether this holds depends on the specific structure of the decoder $\SM$ and the statistic $Y$, and does not hold universally. A notable case where it does is the Gumbel-max watermark, where $\zeta_t$ is a random vector with i.i.d.\ $\UM(0,1)$ entries and $Y_t$ selects the entry indexed by $w_t$, preserving independence across tokens even when $\zeta_t$ is shared.

\begin{asmp}[Independence across sub-blocks]
\label{asmp:subblocks-independence}
For any $i \in \IM_{k, l}^{Y}$ and $j \in \IM_{k', l'}^{Y}$, $Y_i \perp Y_j$, whenever either $k \neq k'$ or $k = k'$ but $l \neq l'$.
\end{asmp}

\paragraph{Factorization from structural independence.}
The pivotal statistics $Y_{1:n}$ are the basis for detection. A direct consequence of the above independence conditions is that the joint distribution of $Y_{1:n}$ factorizes across blocks---and in some cases, across sub-blocks---which simplifies both analysis and inference.

\begin{prop}[Distribution factorization]
\label{thm:data_gen}
Let $\Pi_{\zeta} = \{\IM_k^\zeta\}_{k=1}^K$ denote a $\zeta$-level partition.
Under Assumption \ref{asmp:blocks-independence}, the joint distribution of $(Y_t)_{t=1}^n$ factorizes as
\begin{align}
\PB( (Y_t)_{t=1}^n \mid \Pi_{\zeta} )
= \prod_{\VM \in \Pi_{\zeta}} \PB((Y_t)_{t \in \VM}| \Pi_{\zeta} )
\label{eq:thm_joint_1}
\end{align}
If Assumption~\ref{asmp:subblocks-independence} holds, let $\Pi_{Y}^{(k)} = \{\IM_{k,l}^{Y}\}_{l=1}^{m_k}$ be the $Y$-level refinement of $\IM_k^{\zeta}$, and define the full $Y$-level partition as $\Pi_Y := \{\IM_{k,l}^{Y}\}_{k, l}$. Then the joint distribution further factorizes as
\begin{align}
\PB( (Y_t)_{t=1}^n \mid \Pi_Y)
= \prod_{\VM \in \Pi_Y} \PB((Y_t)_{t \in \VM}|\Pi_{Y}).
\label{eq:thm_joint_2}
\end{align}
\end{prop}

\paragraph{Minimal units.}
Proposition~\ref{thm:data_gen} establishes that, conditioned on the observed repetition pattern (represented by the tuple $(\Pi_{\zeta}, \Pi_Y)$), the joint distribution of $(Y_t)_{t=1}^n$ factorizes into independent components.
We denote such a component by $\VM$, which corresponds either to a block like $\IM_1^\zeta, \ldots, \IM_K^\zeta$, where pseudorandom variables are shared (as in~\eqref{eq:thm_joint_1}), or to a sub-block like $\IM_{1,1}^Y, \ldots, \IM_{K,m_K}^Y$, where pivotal statistics coincide (as in~\eqref{eq:thm_joint_2}).
We refer to this element $\VM$ as a \emph{minimal unit}---the finest partition level at which this independence factorization holds.
We denote the set of all minimal units as $\Pi$, which can be either $\Pi_{\zeta}$ or $\Pi_Y$ depending on the structure. In the case of the Gumbel-max watermark, for instance, the minimal units are the sub-blocks. A key implication is that pivotal statistics from different minimal units are mutually independent, while those within the same unit might exhibit strong dependence due to pseudorandomness collisions. 

\subsection{Problem Formulation}
\label{sec:problem-formulation}

With the two-layer partition structure in place, we now formalize the hypothesis testing problem. Given data $Y_{1:n}$, where each $Y_t = Y(\token_t, \zeta_t)$ depends on the token $\token_t$ and its associated pseudorandom variable $\zeta_t$, we begin by identifying the repetition pattern and representing it through the two-layer partitions $\Pi_\zeta$ and $\Pi_Y$. The goal is to test:
\begin{equation}\label{eq:structured_test} 
H_0: Y_t~|~(\Pi_\zeta, \Pi_Y) \sim \mu_0, \quad t = 1,\dots,n  
\quad \text{vs.} \quad 
H_1: Y_t~|~(\Pi_\zeta, \Pi_Y) \sim \mu_{1,\bP_t}, \quad t = 1,\dots,n .\end{equation} 
The notation $Y_t~|~(\Pi_\zeta, \Pi_Y)$ indicates that the joint distribution of $(Y_t)_{t=1}^n$ follows the observed repetition structure: indices within the same block of $\Pi_\zeta$ share the same pseudorandom variable, and those within the same sub-block of $\Pi_Y$ take on the same pivotal statistic.

\begin{rem}[Comparison with previous work]
The main difference from prior work~\citep{li2024statistical} is that $(Y_t)_{t=1}^n$ are no longer independent under either $H_0$ or $H_1$.\footnote{For theoretical analysis, we assume that $\bP_{1:n}$ is fixed but unknown. This simplification preserves the difficulty of the problem, as $\bP_{1:n}$ are still not observed. Under this assumption, \citet{li2025optimal} shows that $(Y_t)_{t=1}^n$ are independent under both $H_0$ and $H_1$. See Section 3.1 of~\citet{li2025optimal} for a related discussion.} The dependence arises from the two-level partition $(\Pi_\zeta, \Pi_Y)$, which forces certain pseudorandom variables and pivotal statistics to be identical within groups. As a result, although each $Y_t$ still marginally follows $\mu_0$ under $H_0$ or $\mu_{1, \bP_t}$ under $H_1$ when conditioning on $(\Pi_\zeta, \Pi_Y)$, their joint distribution no longer factorizes across $t$ and instead follows the one described in Proposition~\ref{thm:data_gen}. In short, pseudorandom collisions induce dependence among pivotal statistics, motivating our new formulation in \eqref{eq:structured_test} and the minimal-unit technique to properly address it.
\end{rem}

At a high level, watermark detection under pseudorandomness collisions reduces to identifying distributional differences in $(Y_t)_{t=1}^n$, given the dependence structure specified by the two-layer partitions $(\Pi_{\zeta}, \Pi_Y)$. By the factorization established in Proposition~\ref{thm:data_gen}, it is both natural and sufficient to consider detection rules that assign score functions to each minimal unit and aggregate the resulting scores into a global test statistic.\footnote{The log-likelihood ratio test also falls into this class, though it is typically impractical as it depends on the inaccessible NTP distributions.}
Specifically, we propose and assign a score function $h_{\VM}$ to every minimal unit $\VM \in \Pi$, and write $Y_{\VM} := (Y_t)_{t \in \VM}$ for the vector of pivotal statistics in $\VM$.  
The detection rule then takes the form:
\begin{equation}
\label{eq:test_decision_rule}
T_n = 
\begin{cases}
1, & \text{if }S_n \ge \gamma_{n,\alpha}, \\
0, & \text{otherwise},
\end{cases}
\end{equation}
where the test statistic is defined as
\begin{equation}
\label{eq:S}
S_n 
= \sum_{\VM \in \Pi} h_{{\VM}}(Y_{\VM}),
\end{equation}
and $\gamma_{n,\alpha}$ is the $(1-\alpha)$ quantile of $S_n$ under $H_0$, ensuring Type I error control: $\PB_{0}(S_n \ge \gamma_{n,\alpha}) = \alpha$.
In practice, $\gamma_{n,\alpha}$ can be estimated via simulation, since the dependence structure of $Y_{1:n}$ is fully characterized by the partitions $(\Pi_\zeta, \Pi_Y)$, and each $Y_t$ marginally follows $\mu_0$ under the null.

\subsection{Detection Efficiency and Optimal Scores}\label{sec:typeII}
\label{sec:optimal}

The central goal of this paper is to solve the hypothesis testing problem~\eqref{eq:structured_test} optimally using detection rules of the form~\eqref{eq:test_decision_rule}. To this end, we require a criterion or efficiency notion to quantify the performance of a given score function.

We follow the spirit of the asymptotic efficiency notion introduced by~\citet{li2024robust}, which quantifies detection efficiency via the decay rate of the least favorable Type II error under a fixed Type I error level. Here, ``least favorable'' refers to the worst-case Type II error over a belief class $\PM$---a collection of plausible NTP distributions that the verifier assumes the true $\bP_t$ belongs to. This formulation reflects a practical constraint: the verifier does not have access to the true $\bP_t$ and must rely on prior knowledge or assumptions to evaluate efficiency.
However, this notion cannot be directly applied in our setting, as it relies on perfect pseudorandomness and thus assumes full independence among these $Y_t$'s. To address this, we introduce a new non-asymptotic notion of efficiency that explicitly incorporates the dependencies induced by the partition $\Pi$.

\begin{defn}[Non-asymptotic $\Pow$-efficiency]
\label{def:non_asymptotic_efficiency}
Let $S_n$ be a test statistic computed from $Y_{1:n}$ using a partition $\Pi$ with $N_n = |\Pi|$ minimal units. Let $\gamma_{n,\alpha}$ denote the critical value corresponding to a Type I error level $\alpha$. For a given family of belief classes $\Pow:=\{\PM_{\VM}\}_{\VM \in \Pi}$, the non-asymptotic $\Pow$-efficiency of the test based on the score functions $\bh = \{h_{\VM}\}_{\VM \in \Pi}$ is defined as
\begin{align}
\label{eq:non_asymptotic_efficiency}
R_{n,\Pow}(\bh) := -\frac{1}{N_n} \sup_{\bP_{\VM} \subseteq \PM_{\VM}, \forall \VM} \log \PB_{1, \bP_{\VM}}(S_n \le \gamma_{n,\alpha}),
\end{align}
 
where $\bP_{\VM} := (\bP_t)_{t \in \VM}$ collect the NTP distributions in the minimal unit $\VM$, $\PM_{\VM}$ is the belief class associated with $\VM$, and the supremum is taken over all collections where each $\bP_{\VM} \subseteq \PM_{\VM}$ for all $\VM \in \Pi$.
\end{defn}
\begin{rem}[Necessity of non-asymptotic efficiency]
Given the hash rule $\zeta_t = \AM(\token_{(t-m):(t-1)}, \Key)$, the number of distinct text windows $\token_{(t-m):(t-1)}$ is bounded by $|\Voca|^m$. Consequently, the total number of possible pseudorandom variables $\zeta_t$ is also bounded by $|\Voca|^m$. Since different minimal units must correspond to different pseudorandom variables, the number of minimal units satisfies $|\Pi| \leq |\Voca|^m$, which does not grow with the text length $n$. This boundedness necessitates a non-asymptotic efficiency notion, as $|\Pi|$ cannot diverge with $n$ when $|\Voca|$ and $m$ are fixed.
\end{rem}

There are two key differences between our efficiency notion and that of~\cite[Theorem 2.1]{li2024statistical}.  
First, $R_{n,\Pow}(\bh)$ is defined for finite $n$ and uses minimal units as the basic building blocks.  
In contrast, the earlier notion is defined at the token level and only in the asymptotic regime as $n \to \infty$.  
That special case corresponds to our framework when the partition is $\Pi = \{\{1\}, \{2\}, \dots, \{n\}\}$, that is, one token per unit.  
Second, our formulation allows different belief classes to be assigned to different minimal units, and each minimal unit can have its own score function.
This flexibility enables us to evaluate a broader range of detection rules and better reflect practical scenarios.  
By contrast, the efficiency notion of~\cite{li2024statistical} requires a single belief class and a single score function across all tokens, which is less expressive.

\begin{asmp}
\label{asmp:main}
 
We assume that
\begin{enumerate}[label=(\roman*)]
\item \textbf{(Independence structure)} 
\label{subasmp:indepedence} Either Assumption \ref{asmp:blocks-independence} or \ref{asmp:subblocks-independence} holds.

\item \textbf{(Bounded variance)} 
\label{subasmp:bounded-variance} Let $\bh=\{h_{\VM}\}_{\VM \in \Pi}$ be the score functions, with each assigned to a minimal unit. We assume that the variances of $\{h_{\VM}(Y_{\VM})\}_{\VM \in \Pi}$ are uniformly bounded under $H_0$.

\item \textbf{(Well posedness)} \label{subasmp:well-posedness} 
Let $B_{n,\Pow}(\bh)$ denote the non-asymptotic quantity defined in \eqref{eq:non_asymptotic_bound_B}.
There exists a minimizer of the infimum over $\theta$ that is bounded by a positive constant independent of both the partition $\Pi$ and $n$.

\end{enumerate}
\end{asmp}

We pose a mild Assumption~\ref{asmp:main} to simplify the efficiency notion $R_{n,\Pow}$. The first condition of independence structure reflects the repetition-induced partition and has been discussed in Section \ref{sec:independence}: although dependence may persist within a block, some independence still holds across different blocks or sub-blocks. The second condition of bounded variance rules out pathological score functions with unbounded variability, and is satisfied in practice since the score functions we study even admit finite MGFs. The last condition of well-posedness ensures that the minimization problem in $B_{n, \Pow}(\bh)$ has stable solutions: the minimizer over $\theta$ is uniformly bounded. Together, these assumptions require only mild regularity and do not limit the practical applicability of our framework.

\begin{thm}[Explicit lower bound for detection efficiency]
\label{thm:power}
Let $\Pow = \{\PM_{\VM}\}_{\VM \in \Pi}$ denote the family of belief classes, with one belief class assigned to each minimal unit.
Let $\phi_{\bP_{\VM}, h_{\VM}}(\theta)$ denote the moment generating function (MGF) under the alternative $H_1$ in~\eqref{eq:structured_test}, defined for any $\theta \ge 0$ as
\begin{align}
\label{eq:moment-generating-function}
\phi_{\bP_{\VM}, h_{\VM}}(\theta) := \EB_{1, \bP_{\VM}}[\exp(-\theta\, h_{\VM}(Y_{\VM}))].
\end{align}
Under Assumption~\ref{asmp:main}, the non-asymptotic $\Pow$-efficiency of the score functions $\bh=\{h_{\VM}\}_{\VM \in \Pi}$ is lower bounded by
\[
R_{n, \Pow}(\bh) \ge B_{n, \Pow}(\bh) - \omega_{N_n},
\]
where 
\begin{align}
\label{eq:non_asymptotic_bound_B}
B_{n, \Pow}(\bh) := - \inf_{\theta \ge 0} \frac{1}{N_n} \sum_{\VM \in \Pi} \left( 
\theta\, \EB_0[h_{{\VM}}(Y_{\VM})] + \sup_{\bP_{\VM} \subseteq \PM_\VM} \log \phi_{\bP_{\VM}, h_{{\VM}}}(\theta)
\right),
\end{align}
and $\omega_{N_n}$ is a deterministic function of $N_n$ satisfying $\omega_{N_n} \to 0$ as $N_n \to \infty$.
\end{thm}

\begin{rem}[Asymptotic tightness]
\label{rem:tight-lower-bound}
Under further regularity conditions, the lower bound \(B_{n, \Pow}(\bh)\) is asymptotically tight in the sense that $\left| R_{n, \Pow}(\bh) - B_{n, \Pow}(\bh) \right| \leq \omega_{N_n}$ for the same sequence $\omega_{N_n}$ introduced in Theorem \ref{thm:power}.
To prove this tightness, we develop a novel non-i.i.d. large-deviation bound.
See Theorem \ref{sec:proof-asymptotic-efficiency-tightness} in the Supplementary Material for more details.
\end{rem}

In Theorem \ref{thm:power}, we lower bound $R_{n, \Pow}(\bh)$ by a more explicit quantity $B_{n, \Pow}(\bh)$, using the classical Chernoff bound. Setting $\theta = 0$ further shows that $B_{n, \Pow}(\bh)$ is always non-negative.

\paragraph{Optimality via minimax optimization.}
The lower bound $B_{n, \Pow}(\bh)$ provides a tractable approximation to the efficiency notion $R_{n, \Pow}(\bh)$ and admits an explicit form suitable for analysis. In particular, identifying the optimal score functions reduces to solving the minimax optimization problem that $\max_{\bm{h}} B_{n, \Pow}(\bh)$.
Since the expression of $B_{n, \Pow}(\bh)$ decomposes over minimal units, the overall optimization problem naturally separates into independent subproblems. Viewing each scaled score function $\theta h_{\VM}$ as a reparameterization, finding the optimal collection $\bh = \{h_{{\VM}}\}_{\VM \in \Pi}$ reduces to solving the following minimax problem for each minimal unit $\VM$:
\begin{equation}
\label{eq:sub-optimization}
h_{{\VM}} = \arg\min_{h} \max_{\bP_{\VM} \subseteq \PM_{\VM}} L(h, \bP_{\VM}),
\quad \text{where} \quad
L(h, \bP_{\VM}) = \EB_0[h(Y_{\VM})] + \log \EB_{1, \bP_{\VM}}[\exp(-h(Y_{\VM}))].
\end{equation}
The key difference from the previous formulation in~\cite[Equation 14]{li2024statistical} is that we now optimize over all NTP distributions $\bP_{\VM}$ within each minimal unit $\VM$, which is necessary to capture the dependence induced by repetition in the two-level partition. In contrast, the previous setting corresponds to the non-repetition case where $|\VM| = 1$, which results in a significantly simpler minimax problem.
Following prior work, we adopt the $\Delta$-regular class as our belief set for simplicity:
\begin{align}
\label{eq:regular_classdef}
\FPM = \left\{ \bP : \max_{w} P_w \le 1 - \Delta \right\}.
\end{align}

%% file: results.tex
\section{Application to the Gumbel-max Watermark}
\label{sec:gumbel}

In this section, we apply our framework to the Gumbel-max watermarking scheme \cite{scott2023watermarking}. Recall that the Gumbel-max decoder can be equivalently written as
\begin{equation}
\label{eq:it-new}
\token_t = \SMmax(\bP_t, \xi_t) := \arg\max_{\token \in \Voca} \frac{\log U_{t, \token}}{P_{t, \token}},
\end{equation}
where $\{\xi_t\}_{t=1}^n = \{ (U_{t, \token})_{\token \in \Voca} \}_{t=1}^n$ denotes $n \times |\Voca|$ i.i.d.~replicates of standard uniform random variables $U(0,1)$. As shown in \eqref{eq:it-new}, the Gumbel-max trick ensures that the decoder samples exactly from the intended NTP distribution $\bP_t$.

The pivotal statistic in this setting is given by \( Y_t = \Yars(w_t, \zeta_t) = U_{t, w_t} \), namely the coordinate of $\zeta_t = (U_{t,\token})_{\token \in \Voca}$ corresponding to the chosen token $\token_t$.
This choice satisfies the refined Assumption~\ref{asmp:subblocks-independence}, implying that each minimal unit coincides with a sub-block.
Consequently, for a minimal unit $\VM = \{t_1, \dots, t_k\}$, all pivotal statistics $Y_{t_1}, \dots, Y_{t_k}$ collapse to the same value, so it is sufficient to consider only the unique representative, say $Y_{t_1}$.
Under the null $H_0$, this statistic still follows $\mathrm{Unif}(0,1)$, since repetition does not alter its marginal distribution.
We next derive its alternative distribution in the following lemma.
 
\begin{lem}
\label{lem:gumbel-alternative}
For the minimal unit $ \VM = \{t_1, \dots, t_k \} $, all pivotal statistics within the unit share the same value, that is, $Y_{t_1} = \cdots = Y_{t_k}$.
Let $ \bP_{\VM} = (\bP_{t})_{t \in \VM} $ denote their corresponding NTP distributions. Then, the alternative distribution of the shared pivotal statistic is given by
\begin{equation}
\label{eq:gumbel-alternative-CDF}
\PB_{1, \bP_{\VM}}(Y_{t_1} \le y \,|\, Y_{t_1} = \cdots = Y_{t_k})
=\frac{ \sum_{w\in \Voca} S_w y^{1/S_w}}{\sum_{w\in \Voca} S_w}
~\text{where}~
S_w =  \left(\sum_{w'\neq w} \max_{t \in \VM}\frac{P_{t,w'}}{P_{t,w}}+1\right)^{-1}.
\end{equation}
\end{lem}

The alternative distribution of $Y_{t_1}$ is considerably more complex, as it depends on the NTP distributions of all tokens $w_{t_1}, \ldots, w_{t_k}$ within the minimal unit $\VM$. This added dependence makes the analysis far more difficult than in \citet{li2024statistical}. In their case, with $|\VM|=1$, the alternative CDF $\bP \mapsto \PB_{1,\bP}(Y_{t_1} \le y)$ is convex for every $y \in [0,1]$, a property central to their analysis. By contrast, in our setting the mapping $(\bP_{t_1}, \ldots, \bP_{t_k}) \mapsto \PB_{1,\bP_{\VM}}(Y_{t_1} \le y \mid Y_{t_1} = \cdots = Y_{t_k})$ is highly non-convex, introducing a unique challenge that requires new analytical tools. We will show how we address this difficulty in Section~\ref{sec:proof-gumbel}.

To apply our framework, we evaluate the detection performance of score functions $\bh = \{h_{\VM}\}_{\VM \in \Pi}$ using the non-asymptotic $R_{n,\Pow}$-efficiency defined in Definition~\ref{def:non_asymptotic_efficiency}. Here, $\Pow$ assigns to each minimal unit a (potentially different) $\Delta$-regular class $\FPM$ for the prior belief, as introduced in~\eqref{eq:regular_classdef}. To identify the optimal score functions, we focus on saddle point solutions of the minimax problem \eqref{eq:sub-optimization}. For a minimal unit $\VM$, a pair $(h^\star, \bP_{\VM}^\star)$ is called a saddle point solution of the minimax problem $\min_h \max_{\bP_{\VM} \subseteq \FPM} L(h, \bP_{\VM})$ if and only if $L(h^\star, \bP_{\VM}) \le L(h^\star, \bP_{\VM}^\star) \le L(h, \bP_{\VM}^\star)$ holds for any score $h$ and $\bP_{\VM} \subseteq \FPM$, where $\bP_{\VM}^\star$ is the set of least-favorable NTP distributions and $h^\star$ is the corresponding optimal score function. We adopt this notion of optimality in line with the robust hypothesis testing literature \citep{huber1973minimax, veeravalli2002minimax, fauss2021minimax}, where saddle point solutions often provide both interpretability and explicit analytical forms. The following theorem specifies when such saddle point solutions exist and gives their explicit forms when they do.

\begin{thm}[Trichotomy of the saddle point solution]
\label{thm:gumbel}
Fix a sub-block $\VM$. There exist constants $0 < \Delta_1^\star \le \Delta_2^\star < \tfrac{1}{2}$, depending only on $\VM$, such that the minimax problem in \eqref{eq:sub-optimization} with belief class $\FPM$ admits a saddle point solution that falls into one of the following three regimes.
\begin{enumerate}[(i)]
\item \textbf{Low-regularity regime $(\Delta\in[0,\Delta_1^\star])$}  
A unique saddle point solution exists. The optimal score function is the weighted-log rule:
\begin{equation}\label{eq:weighted-log}
h_{{\VM}}^{\mathrm{gum}}(y)=
\frac{(|\VM|\wedge|\Voca|)\,\Delta}{(|\VM|\wedge|\Voca|-1)(1-\Delta)}\,
\log y.
\end{equation}

\item \textbf{Intermediate regime $(\Delta \in (\Delta_1^\star, \Delta_2^\star))$}
In this range, the minimax problem in \eqref{eq:sub-optimization} does not admit a saddle point solution.

\item \textbf{High-regularity regime $(\Delta\in[\Delta_2^\star,\tfrac12))$}
A unique saddle point solution exists. The optimal score function takes the least-favorable form:
\begin{equation}\label{eq:least-favorable}
h_{{\VM}}^{\mathrm{gum}}(y)=
\log \left(y^{\frac{\Delta}{1-\Delta}}+y^{\frac{1-\Delta}{\Delta}}\right).
\end{equation}
\end{enumerate}
\end{thm}
 
\begin{rem}[Beyond saddle point solutions]
\label{rem:beyong-saddle-point}
Saddle point solutions are a strong form of optimality, offering both interpretability and explicit analytical forms. If we relax this requirement and do not insist that the optimal score function be part of a saddle point pair, then a solution always exists in the intermediate regime. However, this solution is not associated with any least-favorable NTP distribution and does not admit a closed-form expression. A detailed discussion is provided in Supplementary~\ref{sec:intermediate-regime-analysis}.
\end{rem}

\paragraph{Discussion of the trichotomy.}
Theorem~\ref{thm:gumbel} reveals a trichotomy that reflects a transition in the existence and form of saddle point solutions as the regularity level $\Delta$ varies. In the low- and high-regularity regimes ($\Delta \notin (\Delta_1^\star, \Delta_2^\star)$), a saddle point solution exists and yields closed-form optimal score functions. Specifically, the expression in \eqref{eq:least-favorable} coincides with the least-favorable solution identified in \cite[Theorem 3.2]{li2024statistical}, which is designed to perform optimally against the least-favorable NTP distribution in $\FPM$. Meanwhile, the form in \eqref{eq:weighted-log} resembles a weighted-log score and arises in the low-regularity regime, where the alternative distribution remains close to the null. In contrast, the intermediate regime $\Delta \in (\Delta_1^\star, \Delta_2^\star)$ admits no saddle point solution, as the minimax problem is not convex–concave and the stability conditions required for a solution break down. A more detailed discussion is provided later.

\begin{figure}[!t]
    \vspace{-0.3in}
    \centering
    \includegraphics[width=\textwidth]{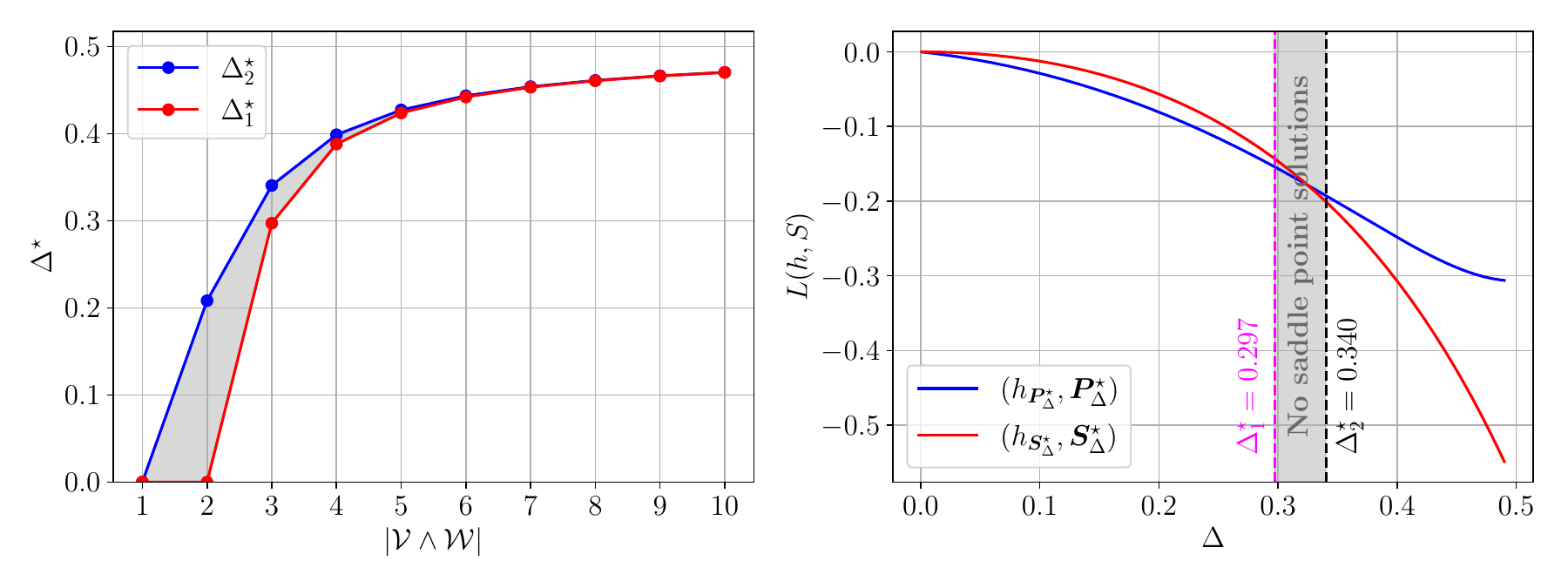}
    \vspace{-0.3in}
    \caption{\textbf{Left}: Transaction thresholds $\Delta_1^\star$ (in red) and $\Delta_2^\star$ (in blue) as functions of $|\VM| \wedge |\Voca|$. The gray region marks the intermediate regime, where no saddle point solution exists. \textbf{Right}: Illustration of why no saddle point solution exists when $|\VM| = 3 \leq |\Voca|$. 
    For $\Delta$ in the low- and high-regularity regimes, each optimal score ($h_{\bS_{\Delta}^\star}$ or $h_{\bP_{\Delta}^\star}$) corresponds to a specific distribution vector ($\bS_{\Delta}^\star$ or $\bP_{\Delta}^\star$). 
    In the intermediate regime, no distribution aligns with either score, so no saddle point solution arises.
    }
    \label{fig:deltavalue}
    \vspace{-0.1in}
    \end{figure}

\vspace{-0.5em}
\paragraph{Effects of $|\VM|$.}  
When $|\VM| = 1$, no repetition occurs, and our optimal score function reduces to the rule in \citet{li2025optimal}, which is exactly the least-favorable rule in \eqref{eq:least-favorable}. When $|\VM| \ge 2$, the role of $|\VM|$ becomes more nuanced. While the least-favorable rule in \eqref{eq:least-favorable} remains unaffected by $|\VM|$, the weighted-log rule in \eqref{eq:weighted-log} incorporates $|\VM|$ through the factor $|\VM| \wedge |\Voca|$, which weakly determines the effective regularity level assigned to each sub-block. In addition, $|\VM|$ influences the transition thresholds $\Delta_1^\star$ and $\Delta_2^\star$ that govern the trichotomy. As illustrated in the left panel of Figure~\ref{fig:deltavalue}, increasing $|\VM| \wedge |\Voca|$ increases both $\Delta_1^\star$ and $\Delta_2^\star$, thereby shrinking the intermediate regime that lacks a saddle point solution. This ``gray region'' eventually vanishes as the informativeness of each block grows.

\paragraph{Justification for discarding repeated pivotal statistics.}
Theorem~\ref{thm:gumbel} shows that for the Gumbel-max watermark, the optimal score function for each minimal unit depends only on its unique pivotal statistic. This makes sense since all pivotal statistics within a minimal unit $\VM$ take the same value, with only the size $|\VM|$ contributing limited additional information. A key implication is that practical heuristics \citep{fernandez2023three,wu2024distortion,dathathri2024scalable} that discard repeated pivotal statistics incur little information loss, as the optimal rule itself follows this principle. Furthermore, once repetitions are removed, the remaining pivotal statistics can be safely treated as i.i.d., which helps improve the alignment between empirical and theoretical Type I errors \citep{fernandez2023three}. Our analysis thus offers a theoretical justification for this widely used practice.

\paragraph{Practical suggestion.}
Since no saddle point solution exists when $\Delta \in (\Delta_1^\star, \Delta_2^\star)$, one may choose any preferred score function in practice. When $|\VM|=1$, the thresholds collapse to $\Delta_1^\star=\Delta_2^\star=0$, so the least-favorable rule in \eqref{eq:least-favorable} applies directly. Empirical evidence \cite[Figure 1]{li2024statistical} suggests that many practical scenarios fall into small-$\Delta$ regimes. Consequently, when $|\VM|\ge 2$, Theorem~\ref{thm:gumbel} often recommends the weighted-log rule.
A practical benefit of our framework over previous one \citep{li2024statistical} is its separation across minimal units, which allows different regularity levels to be assigned to different units. In our LLM experiments, we find that choosing $\Delta$ carefully—for example, setting $\Delta = 1- \max_{\token} P_\token$, where $P_\token$ is the underlying NTP distribution—often improves performance. In practice, however, $1- \max_{\token} P_\token$ is typically unknown and must be estimated from related models or tasks. Such estimation can introduce inaccuracies and, in turn, reduce detection efficiency.

\paragraph{Why the saddle point solution does not exist.}
We now briefly explain why no saddle point solution exists in the intermediate regime. To formalize this, we reparameterize the minimax problem in \eqref{eq:sub-optimization} as $\min_{h} \sup_{\bS \in \DM_{\Delta}} L(h, \bS)$, where $\bS$ denotes the reparameterized distribution vector and $\DM_{\Delta}$ its domain. If a saddle point solution existed, there would be a pair $(h^\star, \bS^\star)$ such that $L(h^\star, \bS) \le L(h^\star, \bS^\star) \le L(h, \bS^\star)$ holds for all $h$ and $\bS \in \DM_{\Delta}$, where $\bS^\star$ is the least-favorable distribution vector and $h^\star$ the corresponding optimal score function. Our analysis in Section~\ref{sec:proof-gumbel} establishes two key facts. First, $h^\star$ must be the log-likelihood ratio score associated with $\bS^\star$. Second, $\bS^\star$ must be either $\bS_{\Delta}^\star$ or $\bP_{\Delta}^\star$ (see Lemma~\ref{lem:extreme-points} for their closed forms). However, when $\Delta \in (\Delta_1^\star, \Delta_2^\star)$, $\bS_\Delta^\star$ fails to maximize the loss for its own log-likelihood ratio score, while $\bP_\Delta^\star$ fails for the same reason, so neither candidate consistently dominates the other. As a result, no saddle point solution exists in this regime. See the right panel of Figure~\ref{fig:deltavalue} for an illustration and Section \ref{sec:proof-gumbel} for a proof sketch.

\section{Application to the Inverse Transform Watermark}
\label{sec:inverse}

In this section, we apply the framework to the 
inverse transform watermark \citep{kuditipudi2023robust}.  
Recall that its decoder is defined as 
\begin{equation*}
\token_t = \SMinv(\bP_t, \xi_t) :=\pi_t^{-1}(F^{-1}(U_t; \pi_t)),
\end{equation*}
where the pseudorandom number $\zeta_t = (\pi_t, U_t)$ with $U_t \sim U(0, 1)$ and $\pi_t$ being sampled uniformly at random from all permutations on $\Voca$. Its pivotal statistic is defined as
\begin{equation*}
\Ydif_t = |U_t - \eta(\pi_t(\token_t))|, 
\quad \text{where}\quad  \eta(w) := \frac{w-1}{|\Voca|-1},
\end{equation*}
maps a discrete token index to the interval $[0,1]$ to enable direct comparison with $U_t \sim U(0,1)$.

The problem is inherently intricate, shown in prior work \citep{li2024statistical}, because the combinatorial structure introduced by the permutation $\pi_t$ significantly complicates the analysis. In our setting, this challenge is further intensified by the fact that we only have block-level independence rather than the stronger sub-block independence.
Indeed, under the pivotal function rule $\Ydif(w, \zeta) = |U - \eta(\pi(\token))|$, if $\zeta = (U, \pi)$ is shared within a block, then the pivotal statistics computed across different tokens $w$ in the block remain dependent. This violates the sub-block independence in Assumption~\ref{asmp:subblocks-independence}.
As a result, the minimal units are entire blocks for the inverse transform watermark, not sub-blocks as in the Gumbel-max watermark. These introduce two layers of complexity: the same combinatorial challenges from $\pi_t$, and the potentially arbitrary dependence within each block. Together, these make the analysis substantially more challenging than in the Gumbel-max case.

To address these challenges, we slightly modify the efficiency measure by adopting an asymptotic perspective in which the vocabulary size tends to infinity. This adjustment leads to a significantly simpler characterization of both the null and alternative distributions, as shown in Theorem~\ref{thm:inverse-asymptotic-distribution}. It also enables us to manage within-block dependence more effectively: in the asymptotic regime, the joint distribution of pivotal statistics within a block is governed by a set of independent latent variables. As a result, the within-block dependence structure becomes much more tractable, allowing for a straightforward derivation of the optimal score function. See Theorem~\ref{thm:inverse-optimal-rule} for details.

\subsection{Asymptotic Distributions}
\label{sec:inverse-asymptotic}

In the following, we focus our analysis on a minimal unit (or a block) $\IM_k^\zeta$ for some index $k$, which consists of $m_k$ sub-blocks denoted by $\{\IM_{k,\ell}^Y\}_{\ell = 1}^{m_k}$. By definition, we have $\IM_k^\zeta = \bigcup_{\ell=1}^{m_k} \IM_{k,\ell}^Y.$

Our results are asymptotic in nature and follow the convention in prior work \citep{li2024statistical}, which studies an asymptotic efficiency by letting the vocabulary size $|\Voca|$ tend to infinity. To enable this analysis, we introduce a comparable set of regularity conditions on the NTP distributions.

\begin{asmp}[Asymptotic NTP conditions]
\label{asmp:heavy-tokensa}
Let $P_{t,(i)}$ denote the $i$-th largest probability in the NTP distribution $\bP_t$. We assume that
\begin{itemize}
    \item[(i)] \textbf{Regular NTP distributions} There exists a universal constant $\delta > 0$ and a sequence $\{\Delta_t\}_{t \ge 1} \subseteq (0,1)$ such that for all $t \ge 1$,
    \begin{equation}
    \label{eq:PM-inv}
    \bP_t \in \overline{\PM}_{\Delta_t}, 
    \quad \text{where} \quad 
    \SPM := \left\{ \bP : \delta \le P_{(1)} \le 1 - \Delta, \quad P_{(2)} \le \eps_{|\Voca|} \right\},
    \end{equation}
    and $\eps_{|\Voca|}$ satisfies $\log |\Voca| \cdot \eps_{|\Voca|} \to 0$ as $|\Voca| \to \infty$.

    \item[(ii)] \textbf{Heavy repeated tokens} All tokens in non-singleton minimal units are \emph{heavy}, meaning that each token has the largest probability in its corresponding NTP distribution. That is, for any $t \in \IM_\ell^Y$ (for some $\ell$) and $m_k > 1$, we have $P_{t, w_\ell} = P_{t,(1)}$.
\end{itemize}
\end{asmp}

We briefly elaborate on Assumption~\ref{asmp:heavy-tokensa}.  
Condition (i) extends the $\Delta$-regular class defined in \eqref{eq:regular_classdef}, and a similar condition is adopted by \cite[Equation (24)]{li2024statistical}. As $|\Voca| \to \infty$, the second-largest probabilities $P_{t,(2)}$ vanish uniformly, implying that each $\bP_t$ becomes asymptotically concentrated on a single token. This assumption simplifies the theoretical analysis while remaining realistic; \cite[Figure 1]{li2024statistical} finds that practical NTP distributions are typically dominated by a single token.

Condition (ii) follows naturally from (i). Since $\bP_t$ asymptotically assigns non-negligible probability to a single token, that token is almost surely the one generated by the LLM, and thus must be the so-called heavy token. Importantly, this condition also aids the dependence analysis within a block: because the verifier lacks access to the NTP distributions during detection, and the same tokens in the same sub-block may come from distinct NTP distributions, assuming a heavy token allows us to use a single index $\Delta_t$ to represent $\bP_t$ in the asymptotic regime. This strategy---also employed by \citet{li2024statistical}---substantially simplifies the analysis while preserving essential asymptotic behavior. 

Since tokens are identical within each sub-block and distinct across different sub-blocks, we let $w_1, \ldots, w_{m_k}$ denote the unique tokens corresponding to each sub-block.
With Assumption~\ref{asmp:heavy-tokensa} in place, the following lemma establishes the asymptotic joint distribution of $(U, \eta(\pi(w_1)), \ldots, \eta(\pi(w_{m_k})))$ under both $H_0$ and $H_1$.

\begin{lem}[Asymptotic joint distribution of pseudorandomness and tokens]
\label{lem:asymptotic-joint-distributions-inverse}
Suppose Assumptions~\ref{asmp:blocks-independence} and \ref{asmp:heavy-tokensa} hold. Fix a minimal unit $\IM_k^\zeta$ from the partition $\Pi_\zeta$ (Definition~\ref{def:zeta_partition}) with $m_k$ sub-blocks $\{\IM^Y_{k,\ell}\}_{\ell=1}^{m_k}$. Define the block-wise regularity vector as
\begin{equation}
\label{eq:vector-Delta-bar}
\bm{\bDelta}_k := (\bDelta_{k,1}, \ldots, \bDelta_{k,m_k}), \quad 
\text{where} \quad \bDelta_{k,\ell} := \max_{t \in \IM^Y_{k,\ell}} \Delta_t.
\end{equation}
As $|\Voca| \to \infty$, the joint distribution of $(U, \eta(\pi(w_1)), \ldots, \eta(\pi(w_{m_k})))$ converges weakly as follows.
\begin{itemize}
    \item Under $H_0$, $(U, \eta(\pi(w_1)), \ldots, \eta(\pi(w_{m_k}))) \xrightarrow{d} (U, X_1, \ldots, X_{m_k})$, where $U, X_1, \ldots, X_{m_k}$ are i.i.d. $\mathrm{Unif}(0,1)$.
    
    \item Under $H_1$, if $P_{t,(1)} = 1 - \Delta_t$ for all $t \in \IM_{k}^\zeta$, $(U, \eta(\pi(w_1)), \ldots, \eta(\pi(w_{m_k}))) \xrightarrow{d} (U, X_1, \ldots, X_{m_k})$, where $X_1, \ldots, X_{m_k}$ are i.i.d. $\mathrm{Unif}(0,1)$, and $U$ is independent and uniformly distributed on
    \begin{equation}
    \label{eq:U-interval}
    \left[\max_{\ell \in [m_k]} \bDelta_{k,\ell} X_\ell,\quad \min_{\ell \in [m_k]} (1 - \bDelta_{k,\ell} + \bDelta_{k,\ell} X_\ell)\right], 
    \end{equation}
    conditioned on this interval being non-empty.
\end{itemize}
\end{lem}

Surprisingly, the asymptotic distributions of $(U, \eta(\pi(w_1)), \ldots, \eta(\pi(w_{m_k})))$ take simple forms under both $H_0$ and $H_1$. Under $H_0$, the pseudorandom variable $U$ is independent of the normalized token vector $(\eta(\pi(w_1)), \ldots, \eta(\pi(w_{m_k})))$, whose entries are all i.i.d. $\mathrm{Unif}(0,1)$. In contrast, under $H_1$, the pseudorandom value $U$ becomes dependent on the token vector due to the block structure specified by $\IM_k^\zeta = \{\IM_{k,\ell}^Y\}_{\ell=1}^{m_k}$. Specifically, $U$ is independently drawn from the interval in \eqref{eq:U-interval}, which itself depends on the token vector, provided the interval is non-empty. This conditional dependence reflects the watermark signal embedded in the generation process.

Recall that for each sub-block $\IM_{k,\ell}^Y$, the corresponding pivotal statistic is defined as $Y_{k,\ell} := |U - \eta(\pi(w_\ell))|$ for $\ell = 1, \ldots, m_k$. By applying a careful change-of-variable argument, we can then characterize the asymptotic joint distribution of the vector $\bm{Y}_k = (Y_{k,1}, \ldots, Y_{k,m_k})$ under both hypotheses, as stated in the following theorem.

\begin{theorem}[Asymptotic joint distribution of pivotal statistics]
\label{thm:inverse-asymptotic-distribution}
Under the same notions and assumptions of Lemma \ref{lem:asymptotic-joint-distributions-inverse}, let $\bm{Y}_k = (Y_{k,1}, \ldots, Y_{k,m_k})$ denote the vector of unique pivotal statistics within the block $\IM_k^\zeta$, where $Y_{k,\ell}$ represents the pivotal statistic within the sub-block $\IM^Y_{k,\ell}$.
Then, as $|\Voca| \to \infty$, the joint PDF of $\bm{Y}_k$ converges as follows.
\begin{itemize}
    \item Under $H_0$, the limiting null PDF is
\begin{align}\label{eq:joint-density-representationh0}
  f_{0}(\bm{y}) = \int_{0}^{1} 2^{|I_1(u)|} \1_{I_2(u) = \emptyset} \, \rd u,
\end{align}
where for a fixed vector $\bm{y} = (y_1, \ldots, y_{m_k})$ and $u \in [0,1]$,
\begin{align*}
  I_1(u) := \left\{ \ell \in [m_k] : 0 < y_\ell < \min(u, 1 - u) \right\}, \quad I_2(u) := \left\{ \ell \in [m_k] : y_\ell \ge \max(u, 1 - u) \right\}.
\end{align*}
\item Under $H_1$, the limiting alternative PDF is
\begin{align}\label{eq:joint-density-representationh1good}
  f_{\bm{\bDelta}_k}
  (\bm{y}) 
  = \frac{1}{I_{m_k}(\bm{\bDelta}_k)} \sum_{\bm{\sigma} \in \{-1,1\}^{m_k}} \left( B_{\bm{\sigma}}^{\bm{\bDelta}_k}(\bm{y}) - A_{\bm{\sigma}}^{\bm{\bDelta}_k}(\bm{y}) \right)_+,
\end{align}
where for each sign vector $\bm{\sigma} = (\sigma_1, \ldots, \sigma_{m_k}) \in \{-1,1\}^{m_k}$ and input $\bm{y} = (y_1, \ldots, y_{m_k})$,
\begin{align*}
  L_{\bm{\sigma}}(\bm{y}) &:= \max_{\ell \in [m_k]} (-\sigma_\ell y_\ell), & 
  U_{\bm{\sigma}}(\bm{y}) &:= \min_{\ell \in [m_k]} (1 - \sigma_\ell y_\ell), \\
  Y^+_{\bm{\sigma}}(\bm{y}) &:= \left( \max_{\ell: \sigma_{\ell}=1}\frac{\bDelta_{k,\ell} }{1 - \bDelta_{k,\ell}} \cdot y_\ell \right)_+, & 
  Y^-_{\bm{\sigma}}(\bm{y}) &:= \left( \max_{\ell: \sigma_{\ell}=-1}\frac{\bDelta_{k,\ell} }{1 - \bDelta_{k,\ell}} \cdot y_\ell \right)_+, \\
  A_{\bm{\sigma}}^{\bm{\bDelta}_k}(\bm{y}) &:= \max\left\{ L_{\bm{\sigma}}(\bm{y}),\ Y^+_{\bm{\sigma}}(\bm{y}) \right\}, & 
  B_{\bm{\sigma}}^{\bm{\bDelta}_k}(\bm{y}) &:= \min\left\{ U_{\bm{\sigma}}(\bm{y}),\ 1 - Y^-_{\bm{\sigma}}(\bm{y}) \right\},
\end{align*}
with $(x)_+ := \max(x, 0)$, and the normalization constant $I_{m_k}(\bm{\bDelta}_k)$ is given by
\[
I_{m_k}(\bm{\bDelta}_k) := \int_{[0,1]^{m_k}} \left( \min_{\ell \in [m_k]} \{1 - \bDelta_{k,\ell} + \bDelta_{k,\ell} x_\ell\} - \max_{\ell \in [m_k]} \{\bDelta_{k,\ell} x_\ell\} \right)_+ \rd x_1 \cdots \rd x_{m_k}.
\]
\end{itemize}
\end{theorem}

As a special case, when $m_k = 1$, the block $\IM_k^\zeta$ contains no repeated tokens. In this setting, the PDFs in Theorem~\ref{thm:inverse-asymptotic-distribution} simplify significantly and recover the previous non-repetitive results in \cite[Theorem 4.1]{li2024statistical}, as shown in the following corollary.

\begin{cor}[Case $m_k=1$]\label{cor:inverse-m=1}
Consider a minimal unit $\IM_k^\zeta$ consisting of a single sub-block. In this case, the parameter vector $\bm{\bDelta}_k$ reduces to a scalar $\Delta_{k,1}$.
The normalization constant from Theorem~\ref{thm:inverse-asymptotic-distribution} simplifies to $I_1(\bm{\bDelta}_k) = 1 - \Delta_{k,1}$. The asymptotic PDF for the single pivotal statistic $Y_{k,1}$ reduce to:
\begin{itemize}
    \item Under $H_0$, the PDF of $Y_{k,1}$ is a triangular distribution on $[0, 1]$:
\begin{align*}
    f_{0}(y_1) =2(1-y_1)\1_{0 \le y_1 \le 1}
\end{align*}
    \item Under $H_1$, the PDF of $Y_{k, 1}$ is a triangular distribution supported on $[0, 1 - \Delta_{k,1}]$:
\begin{align}\label{eq:joint-density-representationh1goodaand}
    f_{\Delta_{k,1}}(y_1) &= 
    \begin{cases}
        \dfrac{2}{1 - \Delta_{k,1}} - \dfrac{2 y_1}{(1 - \Delta_{k,1})^2}, & \text{if } 0 < y_1 < 1 - \Delta_{k,1}, \\[0.5em]
        0, & \text{otherwise}.
    \end{cases}
\end{align}
\end{itemize}
\end{cor}

\subsection{Optimal Score Function}
\label{sec:inverse-optimal}

As shown in Theorem~\ref{thm:inverse-asymptotic-distribution}, when the vocabulary size $|\Voca|$ tends to infinity, the joint distributions of the unique pivotal statistics within each minimal unit simplify significantly under both $H_0$ and $H_1$. To incorporate this effect in our framework, we replace the original class-dependent efficiency from Definition~\ref{def:non_asymptotic_efficiency} with its asymptotic counterpart, defined over the new class $\SPM$ introduced in~\eqref{eq:PM-inv}. Specifically, we consider the following asymptotic efficiency, denoted by $\bar{R}_{n, \Pow}$ and defined by
\begin{equation}
\label{eq:new-efficiency}
\bar{R}_{n, \Pow}(\bh):=
\liminf_{|\Voca|\to \infty} R_{n, \Pow}(\bh)
\ge \liminf_{|\Voca|\to \infty} B_{n, \Pow}(\bh) - \omega_{N_n},
\end{equation}
where $\Pow$ assigns the new class $\SPM$ (with potentially different values of $\Delta$) to the minimal units, and $\omega_{N_n}$ is a vanishing term tending to zero as the number of minimal units $N_n \to \infty$ (by Theorem~\ref{thm:power}).

To identify the optimal score functions, we adopt the same strategy as in the analysis of the Gumbel-max watermark. The quantity $\liminf\limits_{|\Voca|\to \infty} B_{n, \Pow}(\bh)$, introduced in~\eqref{eq:new-efficiency} and defined in~\eqref{eq:non_asymptotic_bound_B}, retains its additive structure across minimal units. The main difference is that the null and alternative distributions are now replaced by their asymptotic limits, as established in Theorem~\ref{thm:inverse-asymptotic-distribution}. Thus, the problem reduces to optimizing the score function for each minimal unit individually. If we assign $\overline{\PM}_{\Delta_{\VM}}$ to a minimal unit $\VM$, we obtain the following minimax optimization problem, which parallels the structure of~\eqref{eq:sub-optimization},
\begin{equation}
\label{eq:new-sub-optimization}
h_{{\VM}} = \arg\min_{h}  \sup_{\Delta_{\VM} \le \bm{\bDelta} \le 1-\delta} L'(h, \bm{\bDelta}),
\quad \text{where} \quad
L'(h, \bm{\bDelta}) = \EB_{f_0}[h(\bm{Y}_k)] + \log \EB_{f_{\bm{\bDelta}}}[\exp(-h(\bm{Y}_k))],
\end{equation}
where $f_0$ and $f_{\bm{\bDelta}}$ denote the asymptotic PDFs of the vector of pivotal statistics $\bm{Y}_k = (Y_{k,1}, \ldots, Y_{k,m_k})$ under the null and alternative, respectively, as given in Theorem~\ref{thm:inverse-asymptotic-distribution}. Here, $\bm{\bDelta} \ge \Delta_{\VM}$ means that every entry of $\bm{\bDelta}$ is at least $\Delta_{\VM}$, and the notation $\bm{\bDelta} \le 1-\delta$ is defined analogously.

We then characterize the optimal score functions that maximize $\bar{R}_{n, \Pow}$-efficiency (up to the infinitesimal error $\omega_{N_n}$), as stated in the following theorem.

\begin{theorem}\label{thm:inverse-optimal-rule}
Suppose Assumptions~\ref{asmp:blocks-independence}, \ref{asmp:main}~\ref{subasmp:bounded-variance}, and \ref{asmp:heavy-tokensa} hold.
Fix a block $\VM = \IM_k^\zeta$ consisting of $m_k$ minimal units, and assume that $\Pow$ assigns the class $\overline{\PM}_{\Delta_{\VM}}$ with $\Delta_{\VM} \in (0, 1)$ to $\VM$.
Define
\begin{align}\label{eq:optimal-detection-rule-def}
  h_{{\VM}}^{\mathrm{inv}}(\bm{y}) 
  := \log \frac{f_{\bm{\bDelta}_{\VM}}(\bm{y})}{f_{0}(\bm{y})}
  \quad \text{with} \quad \bm{y} = (y_1,\ldots,y_{m_k}), \quad
  \bm{\bDelta}_{\VM} = (\Delta_{\VM},\ldots, \Delta_{\VM}) \in \RB^{m_k},
\end{align}
where $f_{0}$ and $f_{\bm{\bDelta}_{\VM}}$ denote the asymptotic null and alternative PDFs, respectively, as given in Theorem~\ref{thm:inverse-asymptotic-distribution}.
The score functions $\{h_{{\VM}}^{\mathrm{inv}}\}_{\VM \in \Pi}$ maximizes the $\bar{R}_{n, \Pow}$-efficiency defined in~\eqref{eq:new-efficiency}, in the sense that
\[
\lim_{M \to \infty} \bar{R}_{n, \Pow}\left(\{[h_{{\VM}}^{\mathrm{inv}}]_{[-M, M]}\}_{\VM \in \Pi}\right) = \infty,
\]
where $[\cdot]_{[-M, M]}$ denotes the clipping operator onto the interval $[-M, M]$.
\end{theorem}

As shown in Theorem~\ref{thm:inverse-optimal-rule}, the optimal score function takes the form of a log-likelihood ratio between the asymptotic null and alternative PDFs. This result generalizes the previous result of~\cite[Theorem 4.2]{li2024statistical}, which corresponds to the special case $m_k = 1$, where each block consists of only a single sub-block.
Notably, the efficiency at the rule $h_{{\VM}}^{\mathrm{inv}}$ diverges to infinity. This arises because the null and alternative PDFs $f_0$ and $f_{\bm{\bDelta}}$ differ on their supports, causing the KL divergence (which is essentially the optimal efficiency) to diverge.
In practice, although the asymptotic regime $|\Voca| \to \infty$ only holds approximately, the score function $h_{{\VM}}^{\mathrm{inv}}$ still performs well---particularly when the regularity level $\Delta_{\VM}$ is adaptively selected.

%% file: experiment.tex
\section{Experiments}
\label{sec:exp}
This section highlights the effectiveness of our framework through synthetic and real-data experiments and shows the practical utility of our proposed methods under pseudorandom collision. All the experiment codes are at \url{https://github.com/lx10077/WatermarkCollision}.

\subsection{Synthetic Studies}
\label{sec:simulation}

\paragraph{Experimental setup.}
We deliberately introduce repetition to evaluate Type~I and Type~II errors under pseudorandom collisions.
We set the vocabulary size to $|\Voca|=10^3$. At each step $t$, with probability $0.9$, a new token is generated according to the considered watermarking scheme. Specifically, we first sample $\Delta_t \sim \mathrm{Unif}(10^{-3}, \Delta_{\max})$ for a prespecified $\Delta_{\max} \in (0,1)$, and then independently construct an NTP distribution $\bP_t$ satisfying $\max_{\token \in \Voca} \bP_{t,\token} = 1 - \Delta_t$. The NTP distribution interpolates between a Zipf law \citep{zipfHumanBehaviorPrinciple2016} and the uniform distribution,\footnote{See Algorithm~1 in the appendix of \citet{li2025optimal} for details.} with $\Delta_{\max}$ controlling its degree of randomness or entropy.
 
With the remaining probability $0.1$, we introduce repetition through two independent mechanisms. With probability $0.05$, we insert a segment sampled from a growing pool of previously used segments, which is updated whenever a new segment is generated or observed. With another $0.05$, we copy a contiguous block from the generated prefix: draw a length $L \in \{1, \ldots, L_{\max}\}$ with $L_{\max} = 5$, select a valid start uniformly, and replicate the block as the next output. Since the repeat decision is independent of the mechanism, this yields a decoupled corruption setup, enabling direct comparisons of Type~I and Type~II errors across score functions. The simulation results for $\Delta_{\max} = 0.7$ are shown in Figure~\ref{fig:simulation}, while further implementation details and additional results for other values of $\Delta_{\max}$ are provided in Supplementary~\ref{append:simulation}.

\begin{figure}[!t]
\centering
\vspace{-0.1in}
\includegraphics[width=\linewidth]{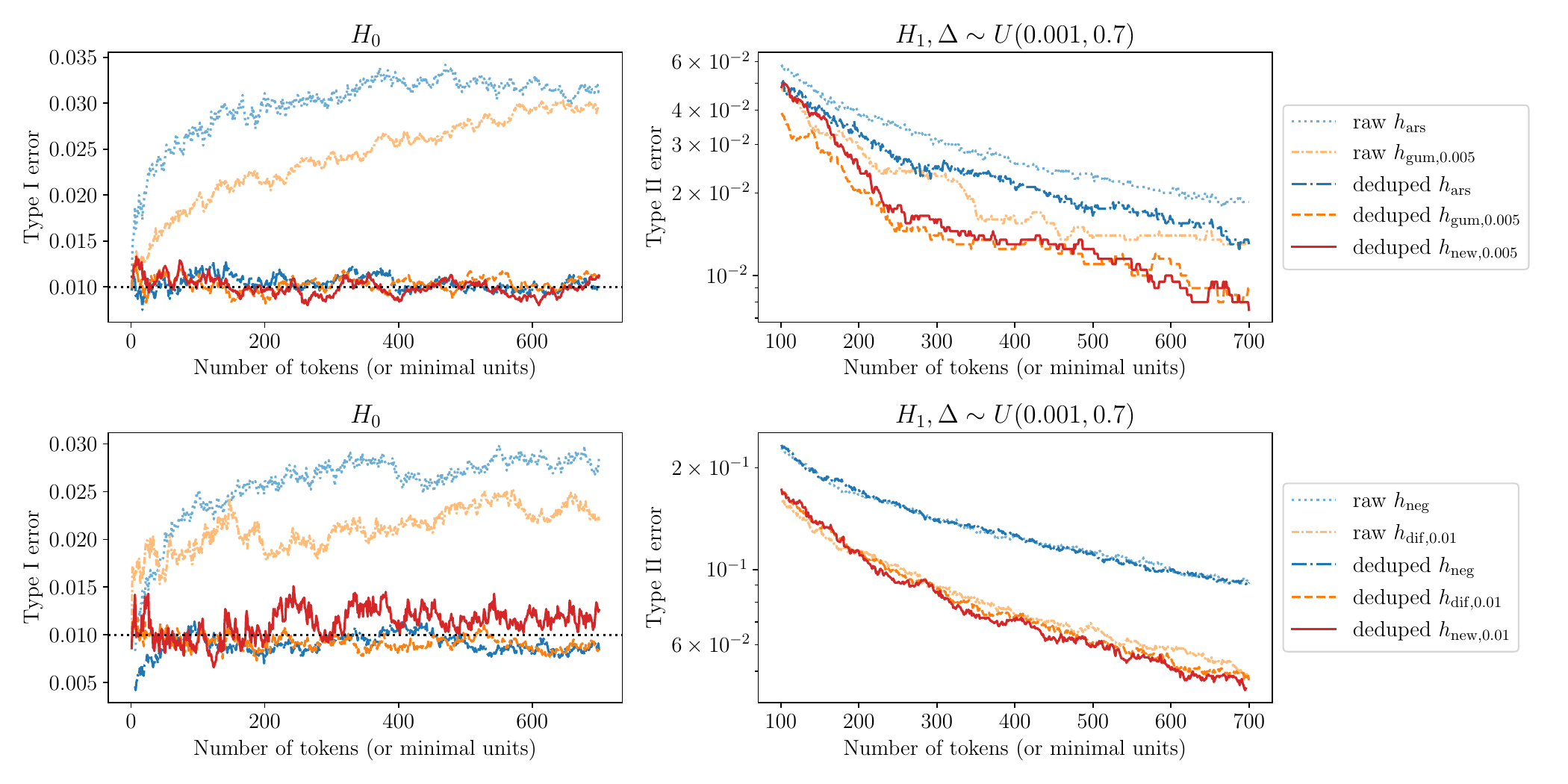}
\caption{Type~I errors (left) and Type~II errors (right, log scale) on synthetic datasets for the Gumbel-max watermark (top) and the inverse-transform watermark (bottom). 
Our new detection rules are denoted by $h_{\mathrm{new},\Delta}$. Here, ``raw'' or ``deduped'' indicates that the detection rule is applied to raw or unique pivotal statistics.}
\vspace{-0.1in}
\label{fig:simulation}
\end{figure}
\vspace{-0.1in}

\paragraph{Type I error.}
From the first column of Figure~\ref{fig:simulation}, existing rules, when directly applied to raw data, fail to control Type~I error: at $\alpha=0.01$, their empirical errors (light curves) hover around $0.03$, well above the nominal level. This inflation arises because repeated pivotal statistics are double-counted as independent evidence, inflating the effective sample size and understating variance. In our setup, about 15\%–20\% of the data is repeated on average. We then evaluate the same detection rules, together with our proposed rule in Theorem~\ref{thm:gumbel} and the new rule in Theorem~\ref{thm:inverse-optimal-rule}, after removing repetitions while regenerating until the sequence length is maintained. In contrast, these methods (darker curves) control Type~I error well, with only natural random fluctuation. These results show that treating minimal units as the basic unit is effective for controlling Type~I errors.

\paragraph{Type II error.}
From the second column of Figure~\ref{fig:simulation}, we find that detection rules, when applied to deduplicated data, achieve Type~II errors that are comparable to, and sometimes smaller than, those on raw data. For example, for the Gumbel-max watermark, $\hars$ performs slightly better once repetitions are removed. Both of our proposed detection rules also perform on par with existing state-of-the-art methods: for Gumbel-max, $h_{\mathrm{new},0.005}$ behaves similarly to $\hars$, while for the inverse-transform watermark, $h_{\mathrm{new},0.01}$ matches the performance of $h_{\mathrm{dif},0.01}$ from \citet{li2024statistical}. These findings indicate that modest levels of repetition do not substantially degrade Type II errors or the detection power, as the watermark signal embedded in unique data is already strong. Another reason is that we set a uniform $\Delta$ across all minimal units, which may limit potential gains from explicitly modeling repetition.

\begin{figure}[!tb]
\centering
\vspace{-0.1in}
\includegraphics[width=\linewidth]{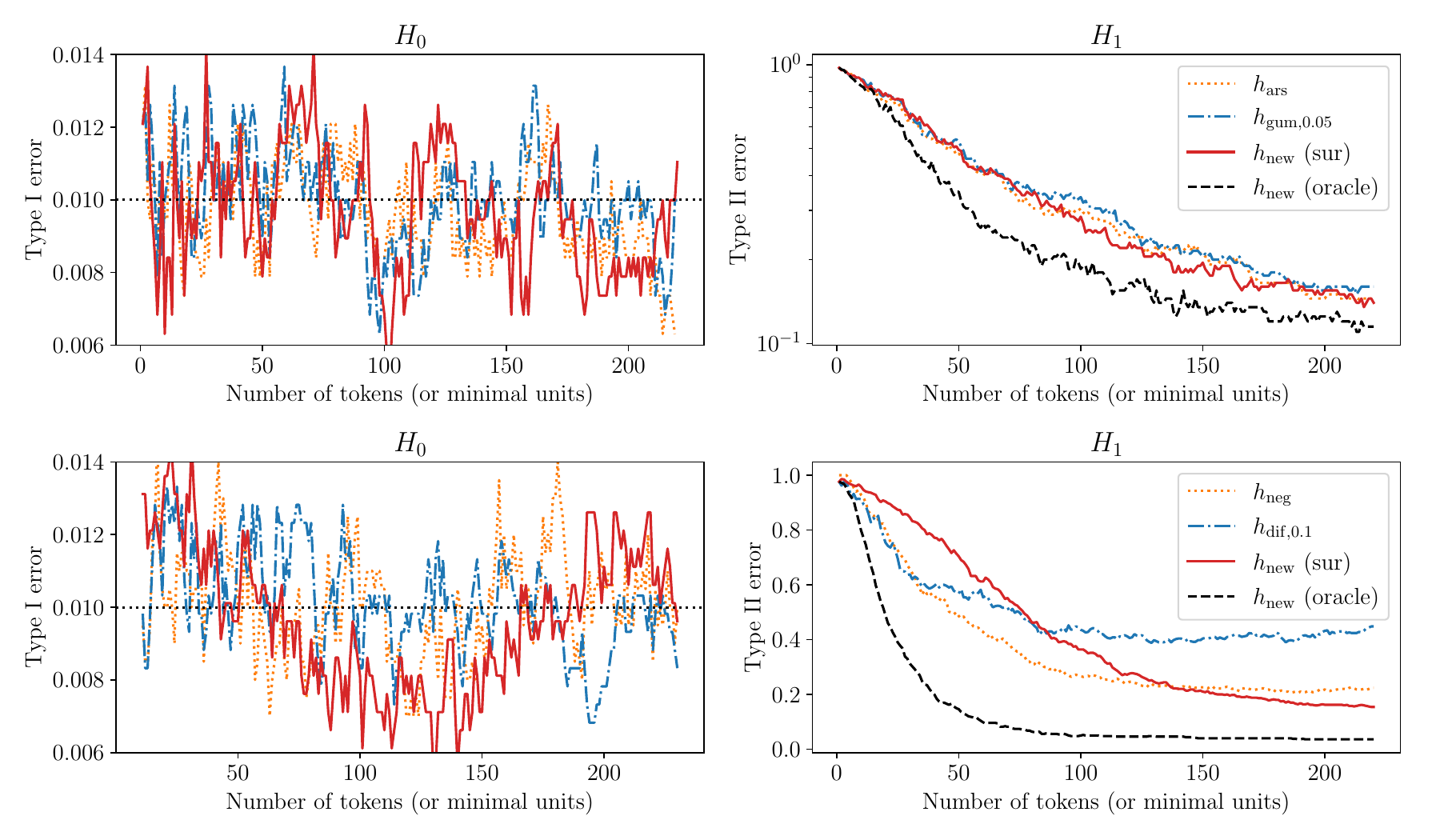}
\caption{Type~I errors (left) and Type~II errors (right, log scale) on C4 dataset for the Gumbel-max watermark (top) and the inverse-transform watermark (bottom). 
Our new detection rules are denoted by $h_{\mathrm{new}}$. Here, ``sur'' and ``oracle'' indicate that the $\Delta$-values are approximated or computed using the ground truth.}
\vspace{-0.1in}
\label{fig:real}
\end{figure}

\subsection{Real-World Examples}
Next, we conduct an empirical analysis of the detection performance of different watermark detection methods on text sequences generated by the language model, OPT-1.3B \citep{zhang2022opt}. We evaluate Type I errors using 2000 human-written samples from the C4 news-like dataset \citep{raffel2020exploring}. To assess Type II errors, we randomly sample prompts from the same dataset, feed them to the model, and let it generate continuations. To ensure a fair evaluation based on unique pivotal statistics, we continue generating until each generated sentence contains at least 300 unique pivotal statistics (or minimal units). This approach guarantees a sufficient number of valid statistics, regardless of the total sequence length. The remaining experimental setup follows \citep{li2024statistical} and is detailed in Supplementary Material~\ref{append:LLM} for completeness.

The empirical Type I (left) and Type II errors (right) are presented in Figure \ref{fig:real}.
The score functions $h_{\mathrm{gum},0.05}$ and $h_{\mathrm{dif},0.1}$ are the two methods proposed in \citep{li2024statistical}, while $h_{\mathrm{ars}}$ and $h_{\mathrm{neg}}$ serve as baseline scores introduced in their original works \citep{scott2023watermarking, kuditipudi2023robust}.
Across most scenarios, all detection methods maintain Type I errors between 0.006 and 0.014, closely aligning with the nominal 0.01 level.
This result is consistent with expectations, as the deduplicated pivotal statistics can be regarded as i.i.d., allowing conventional detection methods to remain effective --- a phenomenon also observed in \citep{fernandez2023three, wu2024distortion}.
To further demonstrate the advantage of our new framework, we consider two approaches for computing the $\Delta$-values for each minimal unit.
The first, denoted as ``oracle,'' uses the ground-truth NTP distributions to compute the regularity level $\Delta_{\VM} = 1 - \max_{t \in \VM}\max_{\token} P_{t,\token}$ for the minimal unit $\VM$.
Since the ground-truth NTP distributions are inaccessible in practice, we introduce a practical surrogate: for a given text, we feed it directly into the detection model (OPT-1.3B in our setup) and autoregressively estimate the NTP distributions.
Although this surrogate approximation omits the preceding context and initial prompt, it still yields a reasonably accurate estimate of $\Delta$. See the red solid curve for $h_{\mathrm{new}}$ ``sur''.
Remarkably, even with this rough approximation, our proposed methods consistently outperform previous state-of-the-art approaches.
Furthermore, when oracle $\Delta$-values are available, our methods demonstrate a clear and substantial advantage, underscoring the effectiveness of our framework in adaptively selecting $\Delta$ for each minimal unit.

%% file: proof.tex
\section{Proof of Main Results}

In this section, we provide proof sketches for Theorems~\ref{thm:gumbel} and~\ref{thm:inverse-optimal-rule}, with the proofs of technical lemmas deferred to the Supplementary Material.

\subsection{Proof of Theorem \ref{thm:gumbel}}
\label{sec:proof-gumbel}

Fix a minimal unit $\VM = \{t_1, \dots, t_k\}$ and, without loss of generality, let $Y_{\VM} := Y_{t_1}$. To facilitate analysis, we reparameterize the alternative distribution $F_{\bS}$ of $Y_{\VM}$ in terms of a vector $\bS$ rather than the original NTP distributions $\bP_{\VM}$. This step is motivated by the fact that, as shown in Lemma~\ref{lem:gumbel-alternative}, the mapping $\bP_{\VM} \mapsto F_{\bS}$ is non-convex, making direct optimization over $\bP_{\VM}$ intractable. In contrast, the mapping $\bS \mapsto F_{\bS}$ yields a much simpler structure that is more amenable to analysis.
Specifically, under this parameterization, the alternative distribution takes the form $F_{\bS}(y) = \sum_{w}\frac{S_w y^{1/S_w}}{\sum_{w'}S_{w'}}$, where each $S_w$ is a nonlinear transformation of $\bP_{\VM}$ (see~\eqref{eq:gumbel-alternative-CDF}).

With this reparameterization, we revisit the minimax problem in~\eqref{eq:sub-optimization}, which now is expressed as
\begin{equation}
\label{eq:opt-gumbel-reformulation}
L(h, \bS) 
= \EB_0[h(Y_{t_1})] + \log \EB_{F_\bS}[\re^{-h(Y_{t_1})}]
= \int h(y) F_0(\rd y) + \log \int \re^{-h(y)} F_{\bS}(\rd y),
\end{equation}
where $F_0$ and $F_{\bS}$ denote the null and alternative distributions of $Y_{t_1}$, respectively. 

Let $\DM_{\Delta}$ denote the set of all feasible $\bS$ vectors induced by $\bP_{\VM} \subseteq \FPM$. Our goal is reduced to identify a saddle point pair $(h^\star, \bS^\star)$ that solves the following minimax problem:
\begin{align}
\label{eq:target-minimax}
\min_h \max_{\bS \in \DM_{\Delta}} L(h, \bS) 
= \int h(y) F_0(\rd y) + \log \int \re^{-h(y)} F_{\bS}(\rd y).
\end{align}
 
The function $L(h, \bS)$ is convex in the score function $h$ for a fixed $\bS$, but generally not concave or convex in $\bS$ when $h$ is fixed and $|\VM| > 1$. As a result,  this renders standard minimax tools unusable, and so even the existence of a solution is not guaranteed. We begin by characterizing when a saddle point solution exists in Lemma \ref{lem:saddle}. 

\begin{lem}[Necessity of optimal score functions]
\label{lem:saddle} 
Let $h_\bS = \log (\rd F_\bS/ \rd y)$ denote the loglikelihood ratio with respect to the alternative distribution $F_\bS$.
The saddle point pair $(h^{\star}, \bS^{\star})$ solves the minimax problem \eqref{eq:target-minimax} if and only if there exists a vector $\bS^\star \in \DM_{\Delta}$ such that $h^{\star} = h_{\bS^{\star}}$ and
\begin{equation}
\label{eq:sup-over-S}
\max_{\bS \in \DM_{\Delta}} L(h_{\bS^{\star}}, \bS) =   L(h_{\bS^{\star}}, \bS^\star).
\end{equation}
The optimal objective value is $ -\mathrm{KL}(F_0\|F_{\bS^\star})$ where $F_0 = \UM(0, 1)$ for the Gumbel-max watermark.
\end{lem}

Lemma~\ref{lem:saddle} implies that any optimal score function corresponding to a saddle point pair must be of the log-likelihood ratio form $h_{\bS}$ for some $\bS \in \DM_\Delta$, and that such functions are always non-decreasing from Lemma \ref{lem:gumbel-alternative}. Hence, it suffices to restrict our attention to non-decreasing $h$. A similar approach is used in \citet{li2024statistical}, but while their feasible domain $\FPM$ is straightforward, our domain $\DM_\Delta$ is substantially more complex, as shown in Lemma~\ref{lem:domain}.

\begin{lem}[Properties of the domain $\DM_{\Delta}$]
\label{lem:domain}
$\DM_{\Delta}$ is a permutation-invariant set.\footnote{It means that for any permutation \( \pi \in \mathrm{Perm}(\Voca) \), the permuted vector \( \pi(\bS) := (S_{\pi(w)})_{w \in \Voca} \) also belongs to \( \DM_{\Delta} \).}
For any $\bS = (S_w)_{w \in \Voca}$ in $\DM_{\Delta}$, it follows that \textnormal{(i)} $0 \le S_w \le 1-\Delta$ for any $w$, \textnormal{(ii)} $\sum_{w} S_w \le 1$, \textnormal{(iii)} $\frac{\max_{w} S_w}{1-\Delta} \le 1 - \frac{1-\sum_{w} S_w }{|\VM| \wedge |\Voca|}$, and \textnormal{(iv)} $\bS_{\Delta}^{\star} := (\frac{1-\Delta}{1+\frac{\Delta}{|\VM|\wedge |\Voca|-1}},0,\cdots,0) \in \DM_{\Delta}$.
\end{lem}

Now, solving the minimax problem \eqref{eq:target-minimax} reduces to identifying a feasible vector $\bS^\star$ that satisfies condition \eqref{eq:sup-over-S}. This requires understanding both \textnormal{(i)} the geometry of the feasible domain $ \DM_{\Delta} $ and \textnormal{(ii)} how to achieve the maximum in the mapping $ \bS \mapsto \EB_{F_{\bS}}[\re^{-h(Y_{t_1})}]=\int \re^{-h(y)} F_\bS(\rd y)$ for any fixed $ y \in [0,1] $ and a given function $ h $.
The first issue is addressed in detail in Lemma~\ref{lem:domain}, while the second issue can be approached by noting that the mapping is \emph{Schur-convex} in $\bS$. In principle, the maximum of a Schur-convex function over a permutation-invariant domain typically occurs at its boundary.  
Hence, both Lemmas \ref{lem:domain} and \ref{lem:schur-convex} assist in solving the inner maximization problem in \eqref{eq:target-minimax}.

\begin{defn}[Schur-convexity]
\label{def:schur}
A function $F$ is Schur-convex if it is isotonic and preserves order. Specifically, if $\bx$ is majorized by $\by$, denoted by, $\bx \le_{m} \by$, then it must satisfy $F(\bx) \le F(\by)$.
For two vectors $\bx, \by \in \RB^d$, $\bx \le_{m} \by$ if and only if \textnormal{(i)} $\sum_{i=1}^k y_{(i)} \ge \sum_{i=1}^k x_{(i)}$ for all $k=1, 2, \ldots, d$ with $y_{(1)} \ge \ldots \ge y_{(d)}$ and $x_{(1)} \ge \ldots \ge x_{(d)}$ the ordered entries and \textnormal{(ii)} $\sum_{i=1}^d x_i = \sum_{i=1}^d y_i$.
\end{defn}

\begin{lem}[Schur-convexity]
\label{lem:schur-convex} 
For any non-decresing function $h$, the map $\bS \mapsto \int \re^{-h(y)} F_\bS(\rd y)$ is Schur-convex in $\bS$.
\end{lem}

\begin{lem}[Reduced domain]
\label{lem:maximum}
Let $\HM_{\Delta} =\{\bS: \frac{1-\Delta}{1+\frac{\Delta}{|\VM|\wedge |\Voca|-1}} 
\le \sum_{w}S_w \}$ be a half-space.
For any non-decreasing function $h$, 
\[
\max_{\bS \in \DM_{\Delta}} \int \re^{-h(y)} F_\bS(\rd y)
 = \max_{\bS \in \DM_{\Delta} \cap \HM_{\Delta}}\int \re^{-h(y)} F_\bS(\rd y).
\]
\end{lem}

Moreover, Lemma \ref{lem:maximum} shows that the maximum of the objective over the domain \( \DM_{\Delta} \) always lies within the half-space \( \HM_{\Delta} \). Consequently, any points in \( \DM_{\Delta} \setminus \HM_{\Delta} \) are suboptimal and can be safely excluded from consideration. 
This reduction allows us to focus on the reduced domain \( \DM_{\Delta} \cap \HM_{\Delta} \). To proceed with the proof, we aim to characterize its convex envelope, which provides a tractable outer approximation while preserving all potential maximizers of the objective.

\begin{lem}[Convex envelope of $\DM_{\Delta} \cap \HM_{\Delta}$]
\label{lem:extreme-points}
We define new sets $\KM_{\Delta}$ and $\EM_{\Delta}$ by
\begin{gather*}
\KM_{\Delta} = \left\{\bS: 
\forall w,0 \le S_w \le 1-\Delta,~\frac{1-\Delta}{1+\frac{\Delta}{|\VM|\wedge |\Voca|-1}} 
\le \sum_{w}S_w \le 1
~\text{and}~
\frac{\max_{w} S_w}{1-\Delta} \le 1 - \frac{1-\sum_{w} S_w }{|\VM|\wedge |\Voca|} \right\},\\
\EM_{\Delta} := \{ \pi(\bP_{\Delta}^{\star}), \pi(\bS_{\Delta}^{\star}),~\forall \pi \in \mathrm{Perm}(\Voca) \},
\end{gather*}
where $\bP_{\Delta}^{\star} := \left(1-\Delta,\Delta,0,\cdots,0\right)$ and $\bS_{\Delta}^{\star} := (\frac{1-\Delta}{1+\frac{\Delta}{|\VM|\wedge |\Voca|-1}},0,\cdots,0)$.
With $\Delta \in (0, 0.5]$, then
\begin{enumerate}  
    \item $\KM_{\Delta}$ is a convex polyhedron with extreme points given by  $\EM_{\Delta} $, that is, $\KM_{\Delta} = \mathrm{conv}(\EM_{\Delta}) $.
    \item $\EM_{\Delta} \subseteq \DM_{\Delta} \cap \HM_{\Delta}$.
    \item $\KM_{\Delta}$ is the convex envelop of $\DM_{\Delta} \cap \HM_{\Delta}$, that is, $\KM_{\Delta} = \mathrm{conv}(\DM_{\Delta} \cap \HM_{\Delta})$.
\end{enumerate}
\end{lem}
By Lemma~\ref{lem:extreme-points}, the convex envelope of \( \DM_{\Delta} \cap \HM_{\Delta} \) is characterized by \( \KM_{\Delta} \), which forms a convex polyhedron whose extreme points are explicitly known. This structure enables the inner maximization to be reduced to a binary comparison, leveraging permutation invariance and Schur-convexity. In particular, Lemma~\ref{lem:maximum-CDF-polyhedron} shows that maximizing over this relaxed domain \( \KM_{\Delta} \) is straightforward.

\begin{lem}[Maximum over polyhedron]
\label{lem:maximum-CDF-polyhedron}
Let the points \( \bP_{\Delta}^{\star} \), \( \bS_{\Delta}^{\star} \), and the set \( \KM_{\Delta} \) be as defined in Lemma \ref{lem:extreme-points}. When $\Delta \in (0, 0.5)$, it follows that for any non-decreasing function $h$, 
\[
\max_{\bS \in \KM_{\Delta}}\int \re^{-h(y)} F_\bS(\rd y)
= \max\left\{ \int \re^{-h(y)} F_{\bP_{\Delta}^{\star}}(\rd y), \int \re^{-h(y)} F_{\bS_{\Delta}^{\star}}(\rd y)\right\}.
\]
\end{lem}

With all supporting lemmas established, we are ready to prove Theorem \ref{thm:gumbel}.

\begin{proof}[Proof of Theorem \ref{thm:gumbel}]

By Lemma \ref{lem:saddle}, solving the minimax problem \eqref{eq:target-minimax} by a saddle point reduces to obtaining a feasible solution $\bS \in \DM_{\Delta}$ that satisfies the optimality condition \eqref{eq:sup-over-S}. 
By Lemmas \ref{lem:maximum}, \ref{lem:extreme-points}, and \ref{lem:maximum-CDF-polyhedron}, for any non-decreasing function $h$, 
 \[
\max_{\bS \in \DM_{\Delta}}\int \re^{-h(y)} F_\bS(\rd y) = \max\left\{ \int \re^{-h(y)} F_{\bP_{\Delta}^{\star}}(\rd y), \int \re^{-h(y)} F_{\bS_{\Delta}^{\star}}(\rd y)\right\}.
\]
This implies that the maximum is achieved by either $\bS_{\Delta}^\star$ or $\bP_{\Delta}^\star$. According to Lemma \ref{lem:saddle}, if an optimal score function exists, it must be either $h_{\bS_{\Delta}^\star}$ or $h_{\bP_{\Delta}^\star}$. Therefore, we verify whether either pair---$(h_{\bS_{\Delta}^{\star}}, \bS_{\Delta}^{\star})$ or $(h_{\bP_{\Delta}^{\star}}, \bP_{\Delta}^{\star})$---solves the minimax problem \eqref{eq:target-minimax}.

\begin{itemize}
\item If $h_{\bS_{\Delta}^\star}$ is the optimal score function, then it must satisfy $L(h_{\bS_{\Delta}^\star}, \bS_{\Delta}^\star)  \ge L(h_{\bS_{\Delta}^\star}, \bP_{\Delta}^\star)$.
  This condition is equivalent to the inequality
  \begin{equation}\label{eq:if-and-only-if1}
  1 \ge \int \re^{-h_{\bS_{\Delta}^\star}(y)} F_{\bP_{\Delta}^\star}(\rd y) = \int \frac{\rd F_{\bP_{\Delta}^\star}}{\rd F_{\bS_{\Delta}^\star}} \rd F_0,
  \end{equation}
  which leads to an algebraic constraint. By numerically solving this condition, we identify the first valid parameter range: $\Delta \in [0, \Delta_1^\star)$.

\item If $h_{\bP_{\Delta}^\star}$ is the optimal score function, it must satisfy $ L(h_{\bP_{\Delta}^\star}, \bP_{\Delta}^\star) \ge L(h_{\bP_{\Delta}^\star}, \bS_{\Delta}^\star).$
  This is equivalent to the inequality
  \begin{equation}\label{eq:if-and-only-if2}
  1 \ge \int \re^{-h_{\bP_{\Delta}^\star}(y)} F_{\bS_{\Delta}^\star}(\rd y) = \int \frac{\rd F_{\bS_{\Delta}^\star}}{\rd F_{\bP_{\Delta}^\star}} \rd F_0.
  \end{equation}
  Numerically solving this condition yields the second valid range: $\Delta \in (\Delta_2^\star, 0.5]$.

\item 
We always have $\Delta_1^{\star} \le \Delta_2^{\star}$ because the Chebyshev inequality ensures that the sum of the right-hand sides of both \eqref{eq:if-and-only-if1} and \eqref{eq:if-and-only-if2} is at least $2$. This implies that the intervals $[0, \Delta_1^\star)$ and $(\Delta_2^\star, 0.5]$ are disjoint, and hence $\Delta_1^{\star} \le \Delta_2^{\star}$. By Lemma~\ref{lem:saddle}, no optimal score function exists when $\Delta \in (\Delta_1^\star, \Delta_2^\star)$. The gray region in Figure~\ref{fig:deltavalue} highlights where this breakdown occurs.
\end{itemize}
\end{proof}

\subsection{Proof of Theorem~\ref{thm:inverse-optimal-rule}}
\label{sec:inverse-optimal-rule-proof}

Recall that $\Pow = \{\overline{\PM}_{\Delta_{\VM}}\}_{\VM \in \Pi}$.
For the score functions $\bh=\{h_{\VM}\}_{\VM \in \Pi}$, it follows that
\begin{align}
\bar{R}_{n, \Pow}(\bh)
&\overset{(a)}{\ge} \liminf_{|\Voca|\to \infty} B_{n, \Pow}(\bh) - \omega_{N_n}, \nonumber \\
&\overset{(b)}{\ge} - \inf_{\theta \ge 0} \limsup_{|\Voca| \to \infty} \frac{1}{N_n} \sum_{\VM \in \Pi} \left( 
\theta\, \EB_0[h_{\VM}(Y_{\VM})] + \sup_{\bP_{\VM} \subseteq \overline{\PM}_{\Delta_{\VM}}} \log \phi_{\bP_{\VM}, h_{\VM}}(\theta)
\right)- \omega_{N_n} \nonumber\\
&\overset{(c)}{\ge} - \limsup_{|\Voca| \to \infty} \frac{1}{N_n} \sum_{\VM \in \Pi} \left(\EB_0[h_{\VM}(Y_{\VM})] + \sup_{\bP_{\VM} \subseteq \overline{\PM}_{\Delta_{\VM}}} \log \phi_{\bP_{\VM}, h_{\VM}}(1)
\right)- \omega_{N_n} \nonumber \\
&\overset{(d)}{=} -  \frac{1}{N_n} \sum_{\VM \in \Pi} \limsup_{|\Voca| \to \infty} \left(\EB_0[h_{\VM}(Y_{\VM})] + \sup_{\bP_{\VM} \subseteq \overline{\PM}_{\Delta_{\VM}}} \log \phi_{\bP_{\VM}, h_{\VM}}(1)
\right)- \omega_{N_n}, \label{eq:Inverse-R-lower-bound}
\end{align}
where (a) applies \eqref{eq:new-efficiency}, (b) uses the expression in \eqref{eq:non_asymptotic_bound_B} with the MGF $\phi_{\bP_{\VM}, h_{\VM}}$ defined in \eqref{eq:moment-generating-function}, (c) follows by setting $\theta = 1$, and (d) exchanges the order of summation and $\limsup$ since the number of minimal units $|\Pi|$ is finite and independent of the vocabulary size $|\Voca|$.

The last lower bound~\eqref{eq:Inverse-R-lower-bound} separates over the scores of each sub-block, so it suffices to consider each subproblem individually. 
Lemma~\ref{lem:inverse-optimal-rule-lower-bound} shows that, as $|\Voca| \to \infty$, the objective function for each subproblem simplifies exactly to \eqref{eq:new-sub-optimization}. 
Its proof essentially exchanges the order of $\limsup$ and $\sup$, and then applies the weak convergence result in Theorem~\ref{thm:inverse-asymptotic-distribution}.

\begin{lem}[Simplified limits]
\label{lem:inverse-optimal-rule-lower-bound}
For a minimal unit $\VM = \IM_k^\zeta$ containing $m_k$ sub-blocks, we represent its associated pivotal statistics $Y_{\VM}$ as the vector $\bm{Y}_k = (Y_{k,1}, \ldots, Y_{k,m_k})$, where each component corresponds to a distinct sub-block. Under Assumptions~\ref{asmp:blocks-independence} and \ref{asmp:heavy-tokensa}, for any Lipschitz continuous function $h : \RB^{m_k} \to \RB$,
\[
\limsup_{|\Voca| \to \infty} \left( \EB_0[h(\bm{Y}_k)] + \sup_{\bP_{\VM} \subseteq \overline{\PM}_{\Delta_{\VM}}} \log \EB_{1, \bP_{\VM}}[\exp(-h(\bm{Y}_k))] \right)  
= \sup_{\Delta_{\VM} \le \bm{\bDelta}' \le 1-\delta} L'(h, \bm{\bDelta}'),
\]
where $\bm{\bDelta}' = (\Delta_{1}', \ldots, \Delta_{m_k}')$ is a regularity-level vector and $L'$ is given in~\eqref{eq:new-sub-optimization}.
\end{lem}

\begin{lem}
\label{lem:inverse-optimal-rule-upper-bound-true}
Let $h_{\VM}^{\mathrm{inv}}(\bm{y}) = \log \frac{f_{\bm{\bDelta}_{\VM}}(\bm{y})}{f_{0} (\bm{y})}$ be defined as in~\eqref{eq:optimal-detection-rule-def}. For any $\Delta_{\VM} \in (0, 1)$, it follows that
\[
\lim_{M \to \infty}\sup_{\Delta_{\VM} \le \bm{\bDelta}' \le 1-\delta} L'([h_{\VM}^{\mathrm{inv}}]_{[-M, M]}, \bm{\bDelta}') =  -\infty,
\]
where $[\cdot]_{[-M, M]}$ denotes the clipping operator onto the interval $[-M, M]$.
\end{lem}
Finally, Theorem~\ref{thm:inverse-optimal-rule} is obtained by combining the lower bound~\eqref{eq:Inverse-R-lower-bound} with Lemmas~\ref{lem:inverse-optimal-rule-lower-bound} and~\ref{lem:inverse-optimal-rule-upper-bound-true}:
\[
\lim_{M \to \infty} \bar{R}_{n, \Pow}(\{[h_{\VM}^{\mathrm{inv}}]_{[-M, M]}\}_{\VM \in \Pi}) = \infty.
\]

%% file: discuss.tex
\section{Discussion}
\label{sec:discussion}

In this paper, we study how to optimally perform watermark detection under pseudorandomness collisions, a phenomenon arising from text repetition in both human-written and low-random LLM outputs. Our central idea is to capture the repetition structure through a hierarchical two-layer partition, identifying minimal units within which strong dependence exists but across which independence is preserved. Using these minimal units as basic components, we develop a new non-asymptotic efficiency measure for evaluating detection rules that take the form of sum-based scores over the minimal units. This formulation naturally casts the search for optimal detection rules as a minimax problem. We then apply our framework to two watermarking schemes---the Gumbel-max watermark and the inverse-transform watermark. For both schemes, we derive the corresponding optimal detection rules and show, both theoretically and empirically, that our rules enable valid Type I error control while achieving comparable or even higher detection power. Moreover, our framework provides a theoretical justification for the widely used heuristic of discarding repeated statistics. At a broader level, our contribution of incorporating pseudorandomness collisions into watermark analysis advances the development of statistical foundations for LLMs \citep{su2025large}.

Building on this foundation, our work opens several promising directions for future research.
First, our framework empirically demonstrates the benefit of assigning different regularity levels $\Delta$ to different minimal units. Further efforts could focus on more accurately approximating the NTP distribution for a given text \citep{li2025likelihood}.
Second, our current analysis adopts a $\Delta$-regular belief class of NTP distributions to represent the least favorable case. Exploring alternative or more refined belief classes may sharpen efficiency guarantees and yield stronger detection rules, particularly when the existing worst-case formulation is overly conservative.
Last, many downstream statistical tasks merit reexamination under pseudorandomness collisions. Examples include detection under human edits \citep{li2024robust} and estimation of the proportion of watermarked tokens in AI-mixed text \citep{li2025optimal}. Both problems can be reformulated with minimal units as the basic component, offering a principled alternative to methods that still assume perfect pseudorandomness.

Beyond methodological development, our study also connects to a classical statistical problem called content authenticity.
Traditional approaches such as stylometry and authorship attribution identify an author’s linguistic fingerprints from stylistic patterns like word-length distributions or function-word usage \citep{mendenhall1887characteristic,mosteller1964inference,juola2006authorship,stamatatos2009survey}.
Plagiarism detection represents another related line, leveraging information-retrieval techniques to identify surface-level overlaps with existing corpora \citep{maurya2024comparative}.
Watermark detection, however, is fundamentally distinct, as its objective is not to detect unconscious stylistic features or verbatim copies, but to verify the presence of a deliberately embedded statistical signal with explicitly specified properties \citep{li2024statistical}.
This distinction makes the reliability of the underlying statistical dependence crucial—precisely the aspect that pseudorandomness collisions undermine. At the same time, our framework may inspire new revisitations of these classical authenticity problems, where one could deliberately embed structured dependence or repeated linguistic patterns via watermarking to enhance detectability and robustness in the era of generative AI.

%% file: tex/appendix.tex
\begin{appendix}
\onecolumn

\begin{center}
{\huge {Supplementary Material}}
\end{center}

This Supplementary Material contains the remaining proofs and technical details.
The proof that supports the general framework is collected in Section \ref{proof:main}.
The proofs about the Gumbel-max watermark are presented in Section \ref{append:gumbel}.
Section \ref{append:inverse} includes the proofs of results for the inverse transform watermark.
Sections \ref{append:simulation} and \ref{append:LLM} contain experiment details for simulation and real-world examples, respectively.

\section{Proof for the General Framework}
\label{proof:main}

\subsection{Proof of Theorem \ref{thm:power}}
\label{proof:power}

\begin{proof}[Proof of Theorem \ref{thm:power}]
By Markov's inequality, it follows that for any $\theta \ge 0$,
\[
\PB_{1}(S_n < \gamma_{n,\alpha}) 
= \PB_{1}(e^{-\theta S_n} > e^{-\theta \gamma_{n,\alpha}})
\le e^{\theta \gamma_{n,\alpha}}\, \EB_{1}[\re^{-\theta S_n}].
\]
Recall that $S_n = \sum_{\VM \in \Pi} h_{\VM}(Y_{\VM})$ and scores for each minimal unit $h_{\VM}(Y_{\VM})$ are independent.
It then follows that
\[
\EB_{1}[\re^{-\theta S_n}] 
= \prod_{\VM \in \Pi}\phi_{\bP_{\VM}, h_{\VM}}(\theta).
\]
Taking logarithms yields
\[
\log \EB_{1}[\re^{-\theta S_n}] = \sum_{\VM \in \Pi} \log \phi_{\bP_{\VM}, h_{\VM}}(\theta).
\]
Thus, the Type II error satisfies
\[
1 - \EB_{1}[T_n] \le \exp\!\left(\theta \gamma_{n,\alpha} +  
\sum_{\VM \in \Pi} \log \phi_{\bP_{\VM}, h_{\VM}}(\theta)\right).
\]
Dividing by $N_n$ (the total number of minimal units $|\Pi|$), we have for any $\theta \ge 0$,
\begin{align}
\label{eq:help1}
\left(1 - \EB_{1}[T_n]\right)^{1/N_n} 
&\le \exp\!\left(\frac{\theta \gamma_{n,\alpha}}{N_n} + 
\frac{1}{N_n}  \sum_{\VM \in \Pi} \log \phi_{\bP_{\VM}, h_{\VM}}(\theta)\right)\nonumber \\
&\le \exp\!\left(\frac{\theta \gamma_{n,\alpha}}{N_n} + 
\frac{1}{N_n} \sum_{\VM \in \Pi}  \sup_{\bP_{\VM} \subseteq \PM_{\VM}}\log \phi_{\bP_{\VM}, h_{\VM}}(\theta)\right).
\end{align}

To proceed with the proof, we introduce a new quantity, denoted by $D_{n, \Pow}(\bh)$:
\begin{align}
\label{eq:Bndefinition}
D_{n, \Pow}(\bh) := -\inf_{\theta \ge 0} \left\{\theta \cdot \frac{\gamma_{n,\alpha}}{N_n} + \frac{1}{N_n} \sum_{\VM \in \Pi} \sup_{\bP_{\VM} \subseteq \PM_{\VM}} \log  \phi_{\bP_{\VM}, h_{\VM}}(\theta) \right\}.
\end{align}
Therefore, by taking the minimum with respect to $\theta \ge 0$ in \eqref{eq:help1}, we have that
\begin{align*} 
&\exp(-R_{n,\Pow}(\bh)) 
=  \sup_{\bP_{\VM} \subseteq \PM_{\VM}, \forall \VM} \left(1 - \EB_{1}[T_n]\right)^{1/N_n}  \\
&\quad \le \exp\!\left( \inf_{\theta \ge 0} \left\{\theta \cdot \frac{\gamma_{n,\alpha}}{N_n} +  \frac{1}{N_n}\sum_{\VM \in \Pi} \sup_{\bP_{\VM} \subseteq \PM_{\VM}} \log  \phi_{\bP_{\VM}, h_{\VM}}(\theta) \right\}\right) = \exp(-D_{n, \Pow}(\bh)),
\end{align*}
which implies that we have $R_{n,\Pow}(\bh) \ge D_{n, \Pow}(\bh)$ for any scores $\bh$.

\begin{lem}
\label{lem:minimaxgamman}
Let Assumptions \ref{asmp:main}~\ref{subasmp:indepedence} and~\ref{subasmp:bounded-variance} hold with $0 < C_{\mathrm{var}} < \infty$ the uniform variance bound for each $h_{\VM}(Y_\VM)$, that is, $\Var_0(h_{\VM}(Y_\VM)) \le  C_{\mathrm{var}}$ for any minimal unit $\VM$.
It follows that for any $\alpha \in (0, 1)$,
\begin{align}
\left| \frac{\gamma_{n,\alpha}}{N_n} - \mu_n \right| \le \varepsilon_0 = \sqrt{\frac{C_{\mathrm{var}}}{N_n \cdot \min(\alpha, 1-\alpha)}}
\quad \text{where} \quad
\mu_n = \frac{1}{N_n} \sum_{\VM \in \Pi}  \EB_0[h_{\VM}].
\end{align}
\end{lem}

\begin{lem}
\label{lem:bounded-optimal-theta}
Under Assumption~\ref{asmp:main}, there exists a universal constant $\overline{M} > 0$, independent of $n$ and the partition $\Pi$, such that for any family of belief classes $\Pow = \{\PM_{\VM}\}_{\VM \in \Pi}$, the optimal value of $\theta$ in the definitions of both $D_{n,\Pow}$ and $B_{n,\Pow}$ lies within the interval $[0, \overline{M}]$.
\end{lem}

Recall that 
\[
B_{n, \Pow}(\bh) := -\inf_{\theta \ge 0} \frac{1}{N_n} \sum_{\VM \in \Pi}  
\left\{ \theta \, \EB_0[h_{\VM}] \,+\,  
\sup_{\bP_{\VM} \subseteq \PM_{\VM}} \log \phi_{\bP_{\VM}, h_{\VM}}(\theta) \right\}.
\]
By Lemma~\ref{lem:bounded-optimal-theta}, there exists a universal constant $\overline{M} > 0$ that doesn't depend on $n$ and $\Pi$ such that the optimal $\theta$ in the definition of both $D_{n,\Pow}(\bh)$ and $B_{n,\Pow}(\bh)$ are uniformly bounded above by $\overline{M}$.  
Combining this with Lemma~\ref{lem:minimaxgamman}, we obtain the approximation bound
\begin{equation}
\label{eq:small-gap}
\bigl| B_{n, \Pow}(\bh) - D_{n, \Pow}(\bh) \bigr| 
\,\le\, \eps_0 \cdot \overline{M}
\,=\, \Theta\!\left( \frac{1}{\sqrt{N_n}} \right),
\end{equation}
where $\eps_0$ is the approximation error from Lemma~\ref{lem:minimaxgamman}, 
and $\overline{M}$ is the bound from Lemma~\ref{lem:bounded-optimal-theta}.
 
\end{proof}

Finally, we provide the proofs of Lemma~\ref{lem:minimaxgamman} and Lemma~\ref{lem:bounded-optimal-theta}.

\begin{proof}[Proof of Lemma \ref{lem:minimaxgamman}]
Let $\mu_n = \EB_0[S_n / N_n]$ denote the expectation of the score $S_n$ under the null. By the definition of $S_n$, we have
\[
\mu_n = \EB_0\left[\frac{1}{N_n}\sum_{\VM \in \Pi} h_{\VM}(Y_{\VM})\right] = \frac{1}{N_n} \sum_{\VM \in \Pi} \EB_0[h_{\VM}(Y_{\VM})].
\]
Since the minimal units are independent under $H_0$, the variance of $S_n / N_n$ can be bounded as
\[
\Var_0\left(\frac{S_n}{N_n}\right) = \frac{1}{N_n^2} \Var_0(S_n) = \frac{1}{N_n^2} \sum_{\VM \in \Pi} \Var_0(h_{\VM}(Y_{\VM})).
\]
Using the uniform variance bound $\Var_0(h_{\VM}(Y_{\VM})) \le C_{\mathrm{var}}$, we have $\Var_0\left(\frac{S_n}{N_n}\right) \le \frac{1}{N_n^2} \sum_{\VM \in \Pi} C_{\mathrm{var}} = \frac{C_{\mathrm{var}}}{N_n}.$
Hence, by Chebyshev's inequality, it follows that for any $\varepsilon > 0$
\[
\PB_0\left(\left|\frac{S_n}{N_n} - \mu_n\right| \ge \varepsilon\right) \le \frac{\Var_0(S_n/N_n)}{\varepsilon^2} \le \frac{C_{\mathrm{var}}}{N_n \varepsilon^2}.
\]
When we set
\[
\varepsilon = \sqrt{\frac{C_{\mathrm{var}}}{N_n \cdot \min(\alpha, 1-\alpha)}}.
\]
This choice implies that $\frac{C_{\mathrm{var}}}{N_n \varepsilon_0^2} = \min(\alpha, 1-\alpha)$. Therefore, we have:
\begin{itemize}
    \item $\PB_0(S_n/N_n \ge \mu_n + \varepsilon_0) \le \PB_0(|S_n/N_n - \mu_n| \ge \varepsilon_0) \le \min(\alpha, 1-\alpha) \le \alpha$.
    \item $\PB_0(S_n/N_n \le \mu_n - \varepsilon_0) \le \PB_0(|S_n/N_n - \mu_n| \ge \varepsilon_0) \le \min(\alpha, 1-\alpha) \le 1-\alpha$.
\end{itemize}
We now use these bounds to constrain $\gamma_{n,\alpha}$. By definition, $\PB_0(S_n \ge \gamma_{n,\alpha}) = \alpha$.
For the upper bound, since $\PB_0(S_n \ge (\mu_n + \varepsilon_0)N_n) < \alpha$, it must be that the threshold $\gamma_{n,\alpha}$ is smaller than $(\mu_n + \varepsilon_0)N_n$. Thus,
\[
\gamma_{n,\alpha} \le (\mu_n + \varepsilon_0)N_n \implies \frac{\gamma_{n,\alpha}}{N_n} \le \mu_n + \varepsilon_0.
\]
For the lower bound, by definition $\PB_0(S_n < \gamma_{n,\alpha}) = 1-\alpha$.
Since $\PB_0(S_n < (\mu_n - \varepsilon_0)N_n) < 1-\alpha$, it must be that the threshold $\gamma_{n,\alpha}$ is larger than $(\mu_n - \varepsilon_0)N_n$. Thus,
\[
\gamma_{n,\alpha} \ge (\mu_n - \varepsilon_0)N_n \implies \frac{\gamma_{n,\alpha}}{N_n} \ge \mu_n - \varepsilon_0.
\]
Combining the upper and lower bounds, we have:
\[
\mu_n - \varepsilon_0 \le \frac{\gamma_{n,\alpha}}{N_n} \le \mu_n + \varepsilon_0,
\]
which is equivalent to:
\[
\left| \frac{\gamma_{n,\alpha}}{N_n} - \mu_n \right| \le \varepsilon_0 = \sqrt{\frac{C_{\mathrm{var}}}{N_n \cdot \min(\alpha, 1-\alpha)}}.
\]
This completes the proof. The second part can be proved similarly.
\end{proof}

\begin{proof}[Proof of Lemma \ref{lem:bounded-optimal-theta}]
The claim for $B_{n,\Pow}$ follows directly from Assumption~\ref{asmp:main} \ref{subasmp:well-posedness}. We now prove the result for $D_{n,\Pow}$.  
Let $\theta_{B}^\star$ and $\theta_{D}^\star$ denote the optimal values of $\theta$ for $B_{n,\Pow}$ and $D_{n,\Pow}$, respectively.  
For any fixed $\theta \ge 0$, define
\begin{align}
   b_{n, \Pow, \bh}(\theta) &:=  \theta \, \EB_0[h_{\VM}]
   + \frac{1}{N_n} \sum_{\VM \in \Pi}  
   \sup_{\bP_{\VM} \subseteq \PM_{\VM}} \log \phi_{\bP_{\VM}, h_{\VM}}(\theta), \\ 
   d_{n, \Pow, \bh}(\theta) &:= \frac{\theta \gamma_{n,\alpha}}{N_n} 
   + \frac{1}{N_n} \sum_{\VM \in \Pi}  
   \sup_{\bP_{\VM} \subseteq \PM_{\VM}} \log \phi_{\bP_{\VM}, h_{\VM}}(\theta).
\end{align}
Both functions are convex in $\theta$ and we have $b'_{n, \Pow, \bh}(\theta_B^\star)  = 0$ and $d'_{n, \Pow, \bh}(\theta_D^\star)  = 0$.
By Assumption~\ref{asmp:main}~\ref{subasmp:well-posedness}, there exists a universal constant $\overline{M} > 0$, independent of $\Pi$, such that $\theta_{B}^\star < \overline{M}$ and $b'_{n, \Pow, \bh}(\overline{M}) > c > 0$ for some constant $c$. Moreover, Lemma~\ref{lem:minimaxgamman} implies that when $N_n$ is sufficiently large, we also have $d'_{n, \Pow, \bh}(\overline{M}) > c/2 > 0$.  
Therefore, for large enough $N_n$, the minimizer $\theta_{D}^\star$ must also satisfy $\theta_{D}^\star < \overline{M}$, completing the proof.
\end{proof}

\subsection{Asymptotic Tightness}
\label{sec:proof-asymptotic-efficiency-tightness}

In this subsection, we show that the lower bound in Theorem~\ref{thm:power} is asymptotically tight under a set of standard regularity conditions. We first introduce the assumptions required for this result. The interpretation and justification of Assumption \ref{asmp:main-lower-bound} are in Section \ref{sec:justassump}.

\begin{asmp}[Regularity conditions for lower bound tightness]
\label{asmp:main-lower-bound}
We assume that
\begin{enumerate}[(i)]
    \item \textbf{(Finite maximizers)} For each minimal unit $\VM$ and all $\theta \ge 0$, the supremum of the MGF over the belief class $\PM_{\VM}$ is achieved on a finite subset $\PM_{\VM}^\star \subseteq \PM_{\VM}$:
    \[
    \sup_{\bP_{\VM} \subseteq \PM_{\VM}} \phi_{\bP_{\VM}, h_{\VM}}(\theta) = \sup_{\bP_{\VM} \subseteq \PM_{\VM}^\star} \phi_{\bP_{\VM}, h_{\VM}}(\theta).
    \]

    \item \textbf{(Informative scores)} The score functions $\bh=\{h_{\VM}\}_{\VM \in \Pi}$ are informative in the sense that, for every minimal unit $\VM$, $\EB_0[h_{\VM}(Y_{\VM})] < \EB_{1, \bP_{\VM}}[h_{\VM}(Y_{\VM})]$ for all $\bP_{\VM} \subseteq \PM_{\VM}$.

    \item \textbf{(CGF regularity)} For all $\bP_{\VM} \subseteq \PM_{\VM}^\star$, the cumulant generating function $\log \phi_{\bP_{\VM}, h_{\VM}}(\theta)$ is well-defined and smooth on its domain. Moreover, for any compact set $K$ inside its domain, there exist constants $0 < \sigma^2_{\min}(K), C_k(K) < \infty$ such that for all $\theta \in K$:
    \begin{itemize}
        \item[(i)] $\sigma^2_{\min}(K) \le \displaystyle \frac{\rd^2}{\rd\theta^2} \log \phi_{\bP_{\VM}, h_{\VM}}(\theta) \le C_2(K) $.
        \item[(ii)] $\displaystyle \left| \frac{\rd ^k}{\rd \theta^k} \log \phi_{\bP_{\VM}, h_{\VM}}(\theta) \right| \le C_k(K)$ for $k = 3, 4$. 
    \end{itemize}

\item \textbf{(Score density regularity)} 
The set of size-one minimal units is $\Pi_1 := \{\VM \in \Pi : |\VM|=1\}$, representing all non-repetitive tokens. We assume that these units are sufficiently large and regular. In particular, there exist universal constants $c, \lambda > 0$ and $C_{BV} < \infty$ such that, for all $N_n > 0$, the size of $\Pi_1$ satisfies $|\Pi_1| \ge c N_n^{\lambda}$. Furthermore, for any $\VM \in \Pi_1$, the score density has uniformly bounded total variation:
\[
\sup_{\bP_{\VM} \subseteq \PM_{\VM}^\star} \mathrm{TV}(\rho_{\bP_{\VM}}) \le C_{BV},
\]
where $\rho_{\bP_{\VM}}$ denotes the alternative PDF of $h_{\VM}(Y_{\VM})$ under NTP distributions $\bP_{\VM}$, and $\mathrm{TV}(\rho) := \int_{-\infty}^{\infty} |\rho'(x)| ,\rd x$ is the total variation of $\rho$.
\end{enumerate}
\end{asmp}

The assumptions in Assumption~\ref{asmp:main-lower-bound} are mild and largely standard in statistical analysis. The finite maximizer condition simply restricts attention to a finite representative subset of distributions without narrowing the generality of the belief class. The informativeness requirement ensures that the scores meaningfully distinguish between null and alternative distributions, which is fundamental for any detection framework. The regularity of the cumulant generating function (CGF) is a standard smoothness condition, guaranteeing that variance and higher-order moments remain controlled on compact sets. Finally, the score density regularity condition leverages the abundance of non-repetitive tokens in typical texts, making the growth and bounded-variation requirements natural and broadly satisfied in practice. Together, these conditions provide technical tractability while remaining weak enough to encompass a wide range of realistic scenarios. The verification of those assumptions in our case is in Section \ref{sec:justassump}.

\begin{thm}[Formal version of Remark \ref{rem:tight-lower-bound}]
\label{thm:asymptotic_efficiency_tight}
Suppose Assumptions~\ref{asmp:main} and~\ref{asmp:main-lower-bound} hold. Then, the lower bound $B_{n, \Pow}(\bh)$, defined in \eqref{eq:non_asymptotic_bound_B}, is asymptotically tight, in the sense that
\[
\left| R_{n, \PM}(\bh) - B_{n, \Pow}(\bh) \right| \leq \omega_{N_n},
\]
where $\omega_{N_n}$ is a deterministic function of $N_n$ satisfying $\omega_{N_n} \to 0$ as $N_n \to \infty$.
\end{thm}

\begin{proof}[Proof of Theorem \ref{thm:asymptotic_efficiency_tight}]
Under Assumption~\ref{asmp:main}, Theorem~\ref{thm:power} guarantees that
\[
R_{n, \PM}(\bh) \ge B_{n, \Pow}(\bh) - \omega_{N_n}.
\]
To establish asymptotic tightness, it remains to prove the upper bound:
\[
R_{n, \PM}(\bh) \le B_{n, \Pow}(\bh) + \omega_{N_n},
\]
under the additional Assumption~\ref{asmp:main-lower-bound}.
To this end, it suffices to show that
\begin{equation}
\label{eq:R-upper-bound}
R_{n, \PM}(\bh) \le D_{n, \Pow}(\bh) + \omega_{N_n},
\end{equation}
where $D_{n, \Pow}(\bh)$ is the intermediate quantity defined in~\eqref{eq:Bndefinition}. Once this is shown, the result follows from the bound
\[
|D_{n, \Pow}(\bh) - B_{n, \Pow}(\bh)| = \Theta\left( \frac{1}{\sqrt{N_n}} \right)
\]
established in~\eqref{eq:small-gap}, completing the proof.

To proceed with the proof, we then introduce some notations.
For each minimal unit $\VM$, the number of possible assignments $\bP_{\VM} \subseteq \PM_{\VM}^\star$ is finite, given that $|\PM_{\VM}^\star|$ is finite.
We denote this finite collection of structured assignments by $\QM^\star$, defined as
\[
\QM^{\star} = \left\{ \{\bP_{\VM}\}_{\VM \in \Pi} :
\bP_{\VM} \subseteq \PM_{\VM}^{\star} \right\}.
\]
Let $\QM^\star = \{\QM_1^\star, \dots, \QM_K^\star\}$ be an enumeration of all such combinations, and define the probability simplex over $\{1, 2, \dots, K\}$ by
\[
\Lambda := \left\{ \lambda = (\lambda_1, \dots, \lambda_K) \in \mathbb{R}^K : \lambda_i \ge 0,\ \sum_{i=1}^K \lambda_i = 1 \right\}.
\]
Each element $\QM_i^\star \in \QM^\star$ corresponds to a particular assignment of NTP distributions that potentially attains the supremum in $\sup_{\bP_{\VM} \subseteq \PM_{\VM}} \phi_{\bP_{\VM}, h_{\VM}}(\theta)$. Specifically, we write $\QM_i^\star = (\bQ_{i, \VM})_{\VM \in \Pi}$, where each minimal unit $\VM$ is assigned the distributions $\bQ_{i, \VM}$. Thus, the index $i \in [K]$ indexes $K$ distinct type-wise configurations of NTP distributions across the entire partition $\Pi$.

Now, we are ready to prove this upper bound \eqref{eq:R-upper-bound}. It follows that
\begin{align}
-R_{n,\Pow}(\bh)\label{eq:upper-boundgoodbad1justtest}
&= \frac{1}{N_n} \sup_{\bP_{\VM} \subseteq \PM_{\VM}, \forall \VM} \log \EB_{1}(1- \EB_{1}[T_n|\{\bP_{\VM}\}_{\VM}]), \nonumber \\
&\overset{(a)}{\ge}   \max_{\lambda \in \Lambda} \frac{1}{N_n}\sum_{i=1}^K \lambda_i \log\left(1-\EB_{1}[T_n \mid \{\bP_{\VM}\}_{\VM} = \QM^\star_i]\right) \\
&\overset{(b)}{\ge}    \max_{\lambda \in \Lambda} \frac{1}{N_n}\sum_{i=1}^K \lambda_i\left[ \min_{\theta \ge0} \left(\theta \cdot \frac{\gamma_{n,\alpha}}{ N_n} + \frac{1}{N_n}  \sum_{\VM \in \Pi}  \log  \phi_{\bQ_{i, \VM}, h_{\VM}}(\theta) \right)-\alpha_{N_n}\right]\\
&  =  \max_{\lambda \in \Lambda} \min_{\theta_i\ge0, \forall i} \frac{1}{N_n}\sum_{i=1}^K \lambda_i\left[  \left(\theta_i \cdot \frac{\gamma_{n,\alpha}}{ N_n} + \frac{1}{N_n}  \sum_{\VM \in \Pi}    \log  \phi_{\bQ_{i, \VM}, h_{\VM}}(\theta_i) \right)-\alpha_{N_n}\right]\\
&\overset{(c)}{=}   \min_{\theta_i \ge 0, \forall i}\max_{\lambda \in \Lambda} \frac{1}{N_n}\sum_{i=1}^K \lambda_i\left[  \left(\theta_i \cdot \frac{\gamma_{n,\alpha}}{ N_n} + \frac{1}{N_n}  \sum_{\VM \in \Pi}  \log  \phi_{\bQ_{i, \VM}, h_{\VM}}(\theta_i) \right)-\alpha_{N_n}\right]\label{eq:upper-boundgoodbad22}
\end{align}
Here, (a) follows from the fact that $\PM_{\VM}^{\star} \subseteq \PM_{\VM}$ and each $\QM_i^\star$ specifies a valid instance of $\{\bP_{\VM}\}_{\VM\in\Pi}$; (b) applies Lemma~\ref{lem:lower}, which provides a non-asymptotic large deviation bound for independent but non-identically distributed random variables; and (c) uses the minimax theorem in Lemma~\ref{lem:minimax} to exchange the order of the maximum and the infimum (where we view $(\theta_1, \ldots, \theta_K)$ as a new $\theta$ to apply this lemma). The term $\alpha_{N_n}$ denotes a positive deterministic function of $N_n$ that converges to zero as $N_n \to \infty$.

Consequently, we have
\begin{align*}
-R_{n,\Pow}(\bh) &\overset{(a)}{\ge} \min_{\theta \ge 0}\left\{\theta \cdot \frac{\gamma_{n,\alpha}}{ N_n} + \frac{1}{N_n} \sum_{\VM \in \Pi}   \sup_{\bP_{\VM}^{\star} \subseteq \PM^{\star}}  \log  \phi_{\bP_{\VM}^{\star}, h_{\VM}}(\theta) \right\} - \alpha_{N_n}\\
&\overset{(b)}{=} \min_{\theta \ge 0}\left\{\theta \cdot \frac{\gamma_{n,\alpha}}{ N_n} + \frac{1}{N_n} \sum_{\VM \in \Pi}   \sup_{\bP_{\VM} \subseteq \PM_{\VM}}  \log  \phi_{\bP_{\VM}, h_{\VM}}(\theta) \right\} - \alpha_{N_n}.
\end{align*}
where (a) follows from the fact that the maximum over the simplex is attained at an extreme point, that is, there exists some $i \in [K]$ such that each $\bQ_{i, \VM}$ in $\QM_i^\star = (\bQ_{i, \VM})_{\VM \in \Pi}$ achieves the supremum $\sup_{\bP_{\VM} \subseteq \PM^{\star}} \log \phi_{\bP_{\VM}, h_{\VM}}(\theta)$; and (b) uses the condition that $\sup_{\bP_{\VM} \subseteq \PM_{\VM}} \phi_{\bP_{\VM}, h_{\VM}}(\theta) = \sup_{\bP_{\VM} \subseteq \PM^{\star}}\phi_{\bP_{\VM}, h_{\VM}}(\theta)$ for all $\theta \ge 0$.

As a result, we obtain the bound
\begin{align}
-R_{n, \Pow}(\bh) &\ge \min_{\theta \ge 0} \left\{ \theta \cdot \frac{\gamma_{n,\alpha}}{N_n} + \frac{1}{N_n} \sum_{\VM \in \Pi} \sup_{\bP_{\VM} \subseteq \PM_{\VM}} \log \phi_{\bP_{\VM}, h_{\VM}}(\theta) \right\} - \alpha_{N_n},
\end{align}
which implies the upper bound
\[
R_{n, \PM}(\bh) \le D_{n, \Pow}(\bh) + \alpha_{N_n}.
\]
This completes the proof.

\begin{lem}[Non-i.i.d. large deviation lower bound]
\label{lem:lower}
Let Assumptions~\ref{asmp:main} and \ref{asmp:main-lower-bound} hold, and let $\QM_i^* = \{ \bQ_{i, \VM} \}_{\VM \in \Pi}$ denote a given assignment of NTP distributions. Then, we have
\[
\frac{1}{N_n} \log \left(1 - \EB_{1} \left[T_n~|~\{\bP_{\VM}\}_{\VM} = \QM_i^\star \right] \right)
\ge  \min_{\theta \ge 0} 
\left\{ \theta \cdot \frac{\gamma_{n,\alpha}}{N_n} + \sum_{\VM \in \Pi} \frac{1}{N_n} \log \phi_{\bQ_{i, \VM}, h_{\VM}}(\theta) \right\} - \alpha_{N_n},
\]
where $\alpha_{N_n}$ is a non-negative function of $N_n$ that converges to zero as $N_n \to \infty$, and is independent of the choice of $\QM_i^\star$.
\end{lem}

\begin{lem}[Minimax theorem]
\label{lem:minimax}
Let $\PM^\star = \{\bP_1^\star, \dots, \bP_K^\star\}$ be a finite set, and let $L(\bP^{\star}, \theta)$ be a function defined on $\PM^\star \times \Theta$, where $\Theta \subset \mathbb{R}^d$ is a convex set. Assume that for each fixed $\bP^{\star} \in \PM^\star$, the function $L(\bP^{\star}, \cdot)$ is continuous and convex in $\theta$. 
Let $\Lambda := \left\{ \lambda \in \mathbb{R}^K : \lambda_i \ge 0,\ \sum_{i=1}^K \lambda_i = 1 \right\}$ denote the probability simplex over $\PM^\star$, and define
\[
F(\lambda, \theta) := \sum_{i=1}^K \lambda_i L(\bP_i^\star, \theta).
\]
Then,
\[
\max_{\lambda \in \Lambda} \min_{\theta \in \Theta} F(\lambda, \theta)
= \min_{\theta \in \Theta} \max_{\lambda \in \Lambda} F(\lambda, \theta)
= \min_{\theta \in \Theta} \sup_{\bP^{\star} \in \PM^{\star}} L(\bP^{\star}, \theta).
\]
\end{lem}
\end{proof}

We provide the proof of Lemma~\ref{lem:minimax} below. The proof of Lemma \ref{lem:lower} is deferred to Section \ref{sec:proof-LDP}, as it is more technical and lengthy.

\begin{proof}[Proof of Lemma \ref{lem:minimax}]
Note that the function $F(\lambda, \theta)$ is convex in $\theta$ for each fixed $\lambda$, and linear (hence concave) in $\lambda$ for each fixed $\theta$.
By assumption, $L(\bP_i^\star, \cdot)$ is continuous and convex on the convex domain $\Theta$, so $F$ is concave-convex and jointly continuous on the product space $\Lambda \times \Theta$. Since $\Delta$ is convex and compact, and $\Theta$ is convex, we may apply Sion’s minimax theorem to exchange the order of $\min$ and $\max$:
\[
\min_{\theta \in \Theta} \max_{\lambda \in \Lambda} F(\lambda, \theta)
= \max_{\lambda \in \Lambda} \min_{\theta \in \Theta} F(\lambda, \theta).
\]
Because the maximum over $\lambda \in \Lambda$ of a convex combination $\sum_i \lambda_i L(\bP_i^\star, \theta)$ is achieved at a vertex of the simplex, we observe:
\begin{gather*}
\min_{\theta \in \Theta} \max_{\lambda \in \Lambda} F(\lambda, \theta)
= \min_{\theta \in \Theta} \max_{\bP^\star \in \PM^\star} L(\bP^\star, \theta).
\end{gather*}
\end{proof}

\subsection{Proof of Lemma \ref{lem:lower}}
\label{sec:proof-LDP}

\begin{proof}[Proof of Lemma \ref{lem:lower}]
At a high level, we analyze the quantity
\[
1 - \EB_{1} \left[T_n~|~\{\bP_{\VM}\}_{\VM} = \QM_i^\star \right] = \PB_{1}(S_n\leq \gamma_{n, \alpha})
\]
by isolating its dominant term and deriving the lower bound stated in the lemma.
Since all NTP distributions are fixed by the assignment $\QM_i^{\star} = \{\bQ_{i, \VM}\}_{\VM \in \Pi}$, we omit this dependence from the notation for clarity.
 
\paragraph{Step 1: Define tilted random variables.}
We begin by defining the random variables:
\[
X_{\VM} := -h_{\VM}(Y_{\VM}).
\]
for each minimal unit $\VM$. Let $F_{\VM}$ denote the alternative CDF of $X_{\VM}$. 
Next, for any $\theta \ge 0$, we define the tilted random variable $\bar X_{\VM}$ with its CDF given by
\[
\bar F_{\VM}(x) := \frac{1}{\phi_{\bQ_{i, \VM}, h_{\VM}}(\theta)} \int_{-\infty}^x \re^{\theta y} \rd F_{\VM}(y),
\]
where the normalizing constant
\[
\phi_{\bQ_{i, \VM}, h_{\VM}}(\theta)
:= \EB_{1, \bQ_{i, \VM}}[\re^{-\theta h_{\VM}(Y_{\VM})}]
= \int_{-\infty}^\infty \re^{\theta y} \rd F_{\VM}(y)
\]
is the MGF of $X_{\VM}$. 
The mean and variance of the tilted random variable $\bar X_{\VM}$ are given by:
\begin{align}
\label{eq:mean-and-variance}
\begin{split}
\EB_1[\bar X_{\VM}] 
&= \frac{1}{\phi_{\bQ_{i, \VM}, h_{\VM}}(\theta)} \int_{-\infty}^{\infty} y\, \re^{\theta y} \rd F_{\VM}(y) 
= \frac{\rd}{\rd \theta} \log \phi_{\bQ_{i, \VM}, h_{\VM}}(\theta) =: m_{\VM}(\theta), \\
\Var_1(\bar X_{\VM}) 
&= \frac{1}{\phi_{\bQ_{i, \VM}, h_{\VM}}(\theta)} \int_{-\infty}^{\infty} (y - \EB_1[\bar X_{\VM}])^2 \re^{\theta y} \rd F_{\VM}(y) 
= \frac{\rd^2}{\rd \theta^2} \log \phi_{\bQ_{i, \VM}, h_{\VM}}(\theta) =: \sigma^2_{\VM}(\theta).
\end{split}
\end{align}

\paragraph{Step 2: Reformulate the tail probability.}

Recall the test statistic $S_n = \sum_{\VM \in \Pi} h_{\VM}(Y_{\VM}) = -\sum_{\VM \in \Pi} X_{\VM}$. Thus, for any $x \in \RB$, we can write $\PB_{1}(S_n\leq -N_n x) = \PB_{1}\left(\sum_{\VM \in \Pi} X_{\VM} \geq N_n x\right)$. We then express the tail probability $\PB_1\left(\sum_{\VM \in \Pi} X_{\VM} \geq N_n x\right)$ in terms of the CDF of tilted random variables, as stated in the following lemma, whose proof can be found in Section \ref{proof:prob-decomposition}.
\begin{lem}
\label{lem:prob-decomposition}
Let $\bar{H}(t)$ be the CDF of the standardized sum $\frac{\sum_{\VM \in \Pi} \bar X_{\VM}  - N_n m(\theta)}{\sqrt{N_n \sigma^2(\theta)}}$. Then we have
\[
\PB_{1}\left(\sum_{\VM \in \Pi} X_{\VM} \geq N_n x\right) 
= \prod_{\VM \in \Pi} \phi_{\bQ_{i,\VM}, h_{\VM}}(\theta)\re^{-\theta N_n m(\theta)} 
\int_{\frac{N_n x-N_n m(\theta)}{\sqrt{N_n \sigma^2(\theta)}}}^{\infty} \re^{-\theta \sqrt{N_n \sigma^2(\theta)} t} \rd \bar{H}_n(t),
\]
where the mixture mean and variance are defined as
\[
m(\theta) :=  \frac{1}{N_n} \sum_{\VM \in \Pi} m_{\VM}(\theta), \qquad
\sigma^2(\theta) :=  \frac{1}{N_n} \sum_{\VM \in \Pi} \sigma^2_{\VM}(\theta).
\]
with each $m_{\VM}(\theta)$ and $\sigma^2_{\VM}(\theta)$ given in \eqref{eq:mean-and-variance}.
\end{lem}

\paragraph{Step 3: Decompose the tail integral via Edgeworth expansion.}
We next apply Edgeworth expansion to approximate the PDF of $\bar{H}_n(t)$.
\begin{lem}
\label{lem:Edgeworth-main} 
Let $\varphi$ denote the PDF of the standard normal distribution $\NM(0, 1)$. Under Assumptions \ref{asmp:main} and \ref{asmp:main-lower-bound}, for any $x \in \RB$, we have:
\[
\frac{\rd \bar{H}_n(x)}{\rd x} = \varphi(x) + \frac{\lambda_{3,N_n}}{6\sqrt{N_n}}(x^3-3x)\varphi(x) + R_{N_n}(x),
\] 
where $R_{N_n}(x)$ is a residual term satisfying $\sup_{x \in \RB} |R_{N_n}(x)| = o(1/\sqrt{N_n})$, and
\[
\lambda_{3,N_n} = \frac{1}{N_n} \sum_{\VM \in \Pi}  \frac{\EB[|X_\VM - \EB_{1, \bQ_{\VM_\tau}}[X_\VM]|^3]}{\sigma^3(\theta)}.
\]
\end{lem}
The proof of Lemma \ref{lem:Edgeworth-main} can be found in Section \ref{proof:Edgeworth-main}.
Using the above expansion, we evaluate the integral in the tail expression whose proof can be found in Section \ref{proof:residual-integral}
\begin{lem}
\label{lem:residual-integral}
Under Assumption \ref{asmp:main-lower-bound}, for any $\theta$ in the fixed interval (e.g., from Lemma \ref{lem:root_stability}), then
\[
\int_{0}^{\infty} \re^{-\theta \sqrt{N_n \sigma^2(\theta)} t} \rd \bar{H}_n(t) = \Theta\left(\frac{1}{\sqrt{N_n}}\right).
\]
where $\Theta(\cdot)$ denote asymptotic equivalence up to constant factors.\footnote{That is, for two positive sequences $a_n$ and $b_n$, we write $a_n = \Theta(b_n)$ if there exist constants $0 < c < C < \infty$ such that $c \cdot b_n \le a_n \le C \cdot b_n$ for all sufficiently large $n$.}
\end{lem}

\paragraph{Step 4: Putting the pieces together.}
Combining Lemmas \ref{lem:prob-decomposition}, \ref{lem:Edgeworth-main}, and \ref{lem:residual-integral}, and setting $x = m(\theta)$, we obtain that
\begin{align}
\PB_{1}(S_n\leq  -N_n x) 
&= \prod_{\VM \in \Pi} \phi_{\bQ_{i, \VM}, h_{\VM}}(\theta) \re^{-\theta N_n m(\theta)} \cdot \Theta\left(\frac{1}{\sqrt{N_n}}\right).
\end{align}
Taking logarithms and normalizing, we complete the proof by noting that
\begin{align}
\frac{1}{N_n} \log \PB_{1}(S_n\leq -N_n x)
&= \frac{1}{N_n} \log\prod_{\VM \in \Pi} \left(\phi_{\bQ_{i, \VM}, h_{\VM}}(\theta) \re^{-\theta N_n m(\theta)} \cdot \Theta\left(\frac{1}{\sqrt{N_n}}\right)\right)\\
&= -\theta x + \frac{1}{N_n} \sum_{\VM \in \Pi} \log\left(\phi_{\bQ_{i, \VM}, h_{\VM}}(\theta)\right) - \Theta\left(\frac{\log N_n}{N_n}\right),\\
&\overset{(a)}{=} \theta \frac{\gamma_{n, \alpha}}{N_n} + \frac{1}{N_n} \sum_{\VM \in \Pi} 1 \log\left(\phi_{\bQ_{i, \VM}, h_{\VM}}(\theta)\right) - \Theta\left(\frac{\log N_n}{N_n}\right),\\
&\overset{(b)}{\ge} \inf_{\theta \ge 0} \left\{ \theta \frac{\gamma_{n, \alpha}}{N_n} + \frac{1}{N_n} \sum_{\VM \in \Pi} \log\left(\phi_{\bQ_{i, \VM}, h_{\VM}}(\theta)\right) ) \right\} - \Theta\left(\frac{\log N_n}{{N_n}}\right)
\end{align}
where (a) follows by setting $x = -\frac{\gamma_{n, \alpha}}{N_n}$, and (b) uses the conclusion that the solution to $m(\theta) = - \frac{\gamma_{n, \alpha}}{N_n}$ satisfies $\theta \ge 0$ from Lemma \ref{lem:root_stability}. The proof of Lemma \ref{lem:root_stability} is in Section \ref{proof:root_stability}.

\begin{lem}[Stability of roots for mixture equations]\label{lem:root_stability}
Let Assumptions~\ref{asmp:main} and \ref{asmp:main-lower-bound} hold. Consider the equation $m(\theta) = - \frac{\gamma_{n, \alpha}}{N_n}$, where $m(\theta)$ is given in Lemma \ref{lem:prob-decomposition}. 
Then the root of this equation is well-defined and lies within a fixed interval $[\underline{M}, \overline{M}]$, where the constants $\underline{M}, \overline{M}> 0$ are independent of $N_n$ and the specific choice of NTP assignment $\QM_i^\star = (\bQ_{i, \VM})_{\VM \in \Pi}$.
\end{lem}

\end{proof}

\subsection{Proof of Lemma \ref{lem:prob-decomposition}}
\label{proof:prob-decomposition}

\begin{proof}[Proof of Lemma \ref{lem:prob-decomposition}]
Let $W_n(x)$, $\bar W_n(x)$, and $\bar{H}_n(x)$ denote the CDFs of the random variables $\sum_{\VM \in \Pi} X_{\VM}$, $\sum_{\VM \in \Pi} \bar X_{\VM}$, and the standardized sum $\frac{\sum_{\VM \in \Pi} \bar X_{\VM}  - N_n m(\theta)}{\sqrt{N_n \sigma^2(\theta)}}$, respectively.  
By definition, it follows that
\begin{align}
\bar{H}_n(x) = \bar W_n\left(\sqrt{N_n \sigma^2(\theta)} x + N_n m(\theta)\right) .
\end{align}
Let $\ri$ denote the imaginary unit. The characteristic function of $W_n(x)$ is given by
\begin{align}
  w_n(z) := \EB_{1, \QM_i^\star}[\re^{\ri z\sum_{\VM \in \Pi} X_{\VM}}] = \prod_{\VM \in \Pi} \phi_{\bQ_{i, \VM}, h_{\VM}}(\ri z).
\end{align}
Similarly, the characteristic function of $\bar W_n(x)$ is
\begin{align}
  \bar w_n(z) := \EB_{1, \QM_i^\star}[\re^{\ri z\sum_{\VM \in \Pi} \bar X_{\VM}}] = \prod_{\VM \in \Pi} \bar \phi_{\bQ_{i, \VM}, h_{\VM}}(\ri z),
\end{align}
where by definition, we have
\[
\bar \phi_{\bQ_{i, \VM}, h_{\VM}}(\ri z) = \phi_{\bQ_{i, \VM}, h_{\VM}}(\ri (z-\ri\theta))/\phi_{\bQ_{i, \VM}, h_{\VM}}(\theta)
\]
is the MGF of the centered variable $\bar X_{\VM}$ under the NTP distribution $\bQ_{i, \VM}$. Using this relation, we obtain the identity:
\begin{align}
  w_n(z) = \bar w_n(z+\ri \theta) \cdot \prod_{\VM \in \Pi} \phi_{\bQ_{i, \VM}, h_{\VM}}(\theta).
\end{align}
If the imaginary part of $z$ is zero, the left side of the last equation is the characteristic function of $W_n(x)$, while the right side is the characteristic function of 
\[
\prod_{\VM \in \Pi} \phi_{\bQ_{i, \VM}, h_{\VM}}(\theta)
\int_{-\infty}^x \re^{-\theta y} \rd \bar W_n(y).
\]
Thus, for all $x \in \mathbb{R}$, we have\begin{align}
  W_n(x) = \prod_{\VM \in \Pi} \phi_{\bQ_{i, \VM}, h_{\VM}}(\theta) \int_{-\infty}^x \re^{-\theta y} \rd \bar W_n(y).
\end{align}

Now, make the change of variables $y = N_n m(\theta) + \sqrt{N_n \sigma^2(\theta)} t$, which yields
\begin{align}
  W_n\left(x\right)& = \prod_{\VM \in \Pi} \phi_{\bQ_{i, \VM}, h_{\VM}}(\theta) \int_{-\infty}^{\frac{x-N_n m(\theta)}{\sqrt{N_n \sigma^2(\theta)}}} \re^{-\theta (N_n m(\theta) + \sqrt{N_n \sigma^2(\theta)} t)} \rd \bar{H}_n(t)\\
&= \prod_{\VM \in \Pi} \phi_{\bQ_{i, \VM}, h_{\VM}}(\theta) \re^{-\theta N_n m(\theta)} \int_{-\infty}^{\frac{x-N_n m(\theta)}{\sqrt{N_n \sigma^2(\theta)}}} \re^{-\theta \sqrt{N_n \sigma^2(\theta)} t} \rd \bar{H}_n(t).
\end{align}
Therefore, by substituting $x \gets N_n x$, we obtain
\[
1-W_n(N_n x) = \prod_{\VM \in \Pi} \phi_{\bQ_{i, \VM}, h_{\VM}}(\theta) \re^{-\theta N_n m(\theta)} \int_{\frac{N_n x-N_n m(\theta)}{\sqrt{N_n \sigma^2(\theta)}}}^{\infty} \re^{-\theta \sqrt{N_n \sigma^2(\theta)} t} \rd \bar{H}_n(t),
\]
which concludes the proof.
\end{proof}

\subsection{Proof of Lemma \ref{lem:Edgeworth-main}}
\label{proof:Edgeworth-main}

\begin{proof}[Proof of Lemma \ref{lem:Edgeworth-main}]

At a high level, we apply the Edgeworth expansion to approximate the PDF of $\bar{H}_n$ using functionals of the standard Gaussian distribution. We will make use of the following lemma and verify that its conditions are satisfied in our setting.

\begin{lem}[Classical Edgeworth expansion]
\label{lem:edgeworth-inid}
Let $X_1,\dots,X_n$ be independent, zero-mean real-valued random variables with variances $\sigma_i^2 = \Var(X_i)$ and finite third moments. Define
\[
S_n := \sum_{i=1}^n X_i, \quad 
B_n := \sqrt{\frac{1}{n} \sum_{i=1}^n \sigma_i^2}, \quad 
Z_n := \frac{S_n}{\sqrt{n} B_n}, \quad \text{and} \quad
\lambda_{3,n} := \frac{1}{n} \sum_{i=1}^n \frac{\EB[X_i^3]}{B_n^3}.
\]
Suppose the following conditions hold:
\begin{enumerate}[(i)]
    \item $\liminf\limits_{n \to \infty} B_n > 0$ and $\limsup_{n \to \infty} \frac{1}{n}\EB[|X_j|^3] < \infty$.
    \item For some positive $\tau < 1/2$, $\frac{1}{n} \sum_{j=1}^n \EB\left[|X_j|^3 \mathbf{1}_{|X_j| > n^{\tau}}\right] \to 0$ as $n \to \infty$.
    \item \textbf{(Cramér's condition)} For every fixed $\varepsilon > 0$, ${n} \int_{|t|>\varepsilon} \prod_{j=1}^n |v_j(t)| dt \to 0$ as $n \to \infty$ where $v_j(t) = \EB[\exp(t\ri X_j)]$ is the characteristic function of $X_j$. 
\end{enumerate}
Then, the Edgeworth expansion satisfies
\[
\sup_{x \in \RB} \left|p_{Z_n}(x) - \left[ \varphi(x) + \frac{\lambda_{3,n}}{6\sqrt{n}}(x^3-3x)\varphi(x)\right] \right| = o\left(\frac{1}{\sqrt{n}}\right),
\]
where $\varphi(x)$ and $p_{Z_n}(x)$ are the PDFs of the standard normal distribution $\NM(0, 1)$ and $Z_n$, respectively.
\end{lem}
\begin{proof}[Proof of Lemma \ref{lem:edgeworth-inid}]
The result follows directly from Theorem 7 in Chapter VI, §4 of \citet{petrov2012sums}. Its proof, which we omit, relies on the analysis of the class of random variables denoted by $S(3,1,1)$, as defined in the same reference.
\end{proof}

To apply Lemma \ref{lem:edgeworth-inid}, we define the mean-zero, independent variables $\{\widetilde{X}_{\VM}\}_{\VM \in \Pi}$ by centering the tilted variables:
\[
\widetilde{X}_{\VM} = \bar{X}_{\VM} - \EB_{1}[\bar{X}_{\VM}] = \bar{X}_{\VM} - m_{\VM}(\theta)
\]
where $ m_{\VM}(\theta)$ is defined in \eqref{eq:mean-and-variance}. These variables remain independent because they preserve the dependence structure of the original variables $\{X_{\VM}\}_{\VM \in \Pi}$. The required regularity conditions for applying the Edgeworth expansion are ensured by Assumption \ref{asmp:main-lower-bound}, as we now formalize.

\begin{fact}[Facts about centered tilted distributions]
\label{fact:center-tilted}
Let $[\underline{M}, \overline{M}]$ be the interval defined in Lemma \ref{lem:root_stability} and fix any $\theta \in [\underline{M}, \overline{M}]$. Under Assumption \ref{asmp:main-lower-bound}, the centered tilted variables $\widetilde{X}_{\VM}$ satisfy:
\label{asmp:edgeworth}
\begin{enumerate}
\item \textbf{Uniformly bounded moments:} There exists a constant $C_{\max} > 0$ such that for all $\VM \in \Pi$ and all $\theta \in [\underline{M}, \overline{M}]$,
\[
\EB_1[\widetilde{X}_{\VM}^4] \le C_{\max}.
\]
In particular, this also implies uniform bounds on third moments, that is, $\EB[|\widetilde{X}_{\VM}|^3] \le C'_{\max}$.

\item \textbf{(Uniform Non-degeneracy of Variance)} \textbf{Uniformly bounded variance away from zero:} There exists a constant $\sigma_{\min}^2 > 0$ such that for all $\VM \in \Pi$ and all $\theta \in [\underline{M}, \overline{M}]$,
\[
\Var(\bar{X}_\VM) = \EB[(\widetilde{X}_\VM)^2] = \sigma_{\VM}^2(\theta) \ge \sigma^2_{\min}.
\]
\end{enumerate}
\end{fact}

We now verify the conditions in Lemma \ref{lem:edgeworth-inid} using the above properties:

\begin{itemize}
    \item  \textbf{Condition (i):} The term
    \[
    \frac{1}{N_n} \sum_{\VM \in \Pi} \EB[|\widetilde{X}_{\VM}|^3] \leq C'_{\max}< \infty,
    \]
    is uniformly bounded. Moreover, \[
    \frac{1}{N_n} \sum_{\VM \in \Pi} \EB[|\widetilde{X}_{\VM}|^2] = \frac{1}{N_n} \sum_{\VM \in \Pi} \sigma_{\VM}^2(\theta) \geq \sigma_{\min}^2>0.
    \]

    \item \textbf{Condition (ii):} Since we have uniform bounds on the fourth moments, for any $\tau < 1/2$, when $n \to \infty$, we have 
    \begin{align}
      \frac{1}{N_n} \sum_{\VM \in \Pi} \EB[|\widetilde{X}_{\VM}|^3 \mathbf{1}_{|\widetilde{X}_{\VM}| > N_n^{\tau}}] \leq \frac{1}{N_n^{1+\tau}} \sum_{\VM \in \Pi} \EB[|\widetilde{X}_{\VM}|^4 \mathbf{1}_{|\widetilde{X}_{\VM}| > N_n^{\tau}}] \leq \frac{C_{\max}}{N_n^{\tau}} \to 0.
\end{align}
 \item \textbf{Condition (iii):} Cramér's condition is satisfied as a consequence of the results presented in \citet{petrov2012sums} (Chapter VI, §4, Lemma 10 and the subsequent discussion), combined with our assumptions on score density regularity in Assumption~\ref{asmp:main-lower-bound}.
\end{itemize}
Therefore, all conditions in Lemma \ref{lem:edgeworth-inid} are satisfied for the centered tilted variables $\{\widetilde{X}_{\VM}\}_{\VM \in \Pi}$. Applying the lemma yields
\begin{align}
\frac{\rd\bar{H}_n(x)}{\rd x}
 = \varphi(x) + \frac{\lambda_{3,N_n}}{6\sqrt{N_n}}(x^3-3x)\varphi(x) + R_{N_n}(x),
\end{align} 
where the remainder satisfies
\begin{align}
\sup_{x \in \RB} |R_{N_n}(x)| = o\left(\frac{1}{{\sqrt{N_n}}}\right),
\end{align}
concluding the proof.
\end{proof}

\subsection{Proof of Lemma \ref{lem:residual-integral}}
\label{proof:residual-integral}

\begin{proof}[Proof of Lemma \ref{lem:residual-integral}]

We begin by defining the integral of interest:
\begin{align*}
  I := \int_{0}^{\infty} \re^{-\theta \sqrt{N_n \sigma^2(\theta)} t} \rd \bar{H}_n(t).
\end{align*}
Applying the Edgeworth expansion from Lemma \ref{lem:Edgeworth-main}, we have
\[
\frac{\rd\bar{H}_n(t)}{\rd t}
 = \varphi(t) + \frac{\lambda_{3,N_n}}{6\sqrt{N_n}}(t^3-3t)\varphi(t) + R_{N_n}(t),
\]
where $\varphi(x)$ is the PDF of the standard normal distribution.

Substituting this expansion into the expression for $I$, we obtain
\begin{align}
  I 
&=\int_{0}^{\infty}\re^{-\theta \sqrt{N_n \sigma^2(\theta)} t} \varphi(t) \rd t + \frac{\lambda_{3,N_n}}{6\sqrt{N_n}} \int_{0}^{\infty} \re^{-\theta \sqrt{N_n \sigma^2(\theta)} t} (t^3-3t)\varphi(t) \rd t 
\\ &+\int_{0}^{\infty}\re^{-\theta \sqrt{N_n \sigma^2(\theta)} t}R_{N_n}(t) \rd t.
\end{align}
 
We now analyze each term on the right-hand side:
\begin{itemize}
\item \textbf{First term:} We compute the integral as follows:
\begin{align}
\int_{0}^{\infty} \re^{-\theta \sqrt{N_n \sigma^2(\theta)} t} \varphi(t) \rd t 
&= \int_{0}^{\infty} \frac{1}{\sqrt{2\pi}} e^{-\frac{1}{2}\left(t^2 + 2\theta \sqrt{N_n \sigma^2(\theta)} t\right)} \rd t\\
&= \int_{0}^{\infty} \frac{1}{\sqrt{2\pi}} \re ^{-\frac{1}{2}\left(t + \theta \sqrt{N_n \sigma^2(\theta)}\right)^2 + \frac{\theta^2 N_n \sigma^2(\theta)}{2}} \rd t\\
&= \re^{\frac{\theta^2 N_n \sigma^2(\theta)}{2}} \int_{\theta \sqrt{N_n \sigma^2(\theta)}}^{\infty} \frac{1}{\sqrt{2\pi}}  \re^{-\frac{u^2}{2}} \rd u \\ 
&= \frac{1}{\sqrt{2\pi}} \cdot \frac{1-\Phi(\theta \sqrt{N_n \sigma^2(\theta)})}{\varphi(\theta \sqrt{N_n \sigma^2(\theta)})}.
\end{align}

\begin{lem}[Mill's ratio \citep{feller1968introduction}]
\label{lem:mills}
Let $\Phi(x)$ and $\varphi(x)$ denote the CDF and PDF of the standard normal distribution $\NM(0, 1)$, respectively. Then for all $x > 0$, it holds that
\[
\frac{x}{1 + x^2} < \frac{1 - \Phi(x)}{\varphi(x)} < \frac{1}{x}.
\]
\end{lem}

Using the classical Mill’s ratio bound in Lemma \ref{lem:mills}, we obtain
\[
\int_{0}^{\infty} \re^{-\theta \sqrt{N_n \sigma^2(\theta)} t} \varphi(t) \rd t  
= \Theta\left( \frac{1}{\theta \sqrt{N_n \sigma^2(\theta)}}\right)
=  \Theta\left( \frac{1}{\sqrt{N_n}}\right).
\]

\item \textbf{Second term:} Since $\lambda_{3,N_n} \leq C$ by assumption, we can bound this term as
\begin{align}
\left|\frac{\lambda_{3,N_n}}{6\sqrt{2\pi N_n}} \int_{0}^{\infty} \re^{-\theta \sqrt{N_n \sigma^2(\theta)} t} (t^3-3t) \re^{-\frac{t^2}{2}} \rd t\right|
 &\leq \frac{C}{\sqrt{N_n}} \int_{0}^{\infty} \re^{-\theta \sqrt{N_n \sigma^2(\theta)} t} (t^3+3t) \re^{-\frac{t^2}{2}} \rd t\\
 &\le \frac{C}{\sqrt{N_n}} \left[O\left(\frac{1}{N_n^2}\right) +  O\left(\frac{1}{N_n}\right) \right]  \\
  &= O\left(\frac{1}{N_n^{3/2}}\right).
\end{align}

\item \textbf{Third term:} Using the bound $\sup_{x \in \RB} |R_{N_n}(x)| = o(1/{\sqrt{N_n}})$ from Lemma \ref{lem:Edgeworth-main}, we have
\begin{align}
  \int_{0}^{\infty}\re^{-\theta \sqrt{N_n \sigma^2(\theta)} t}R_{N_n}(t) \rd t= o\left(\frac{1}{{N_n}}\right).
\end{align}
\end{itemize}

Putting all terms together, we conclude that
\begin{align}
  I &= \Theta\left(\frac{1}{\sqrt{N_n}}\right) + O\left(\frac{1}{N_n^{3/2}}\right) + o\left(\frac{1}{{N_n}}\right) =\Theta\left(\frac{1}{\sqrt{N_n}}\right).
\end{align}

\end{proof}

\subsection{Proof of Lemma \ref{lem:root_stability}}
\label{proof:root_stability}

\begin{proof}[Proof of Lemma~\ref{lem:root_stability}]
Recall that the empirical mean and variance functions are defined by
\[
m(\theta) :=  \frac{1}{N_n} \sum_{\VM \in \Pi} m_{\VM}(\theta), \qquad
\sigma^2(\theta) :=  \frac{1}{N_n} \sum_{\VM \in \Pi} \sigma^2_{\VM}(\theta).
\]
We claim that there exists a unique non-negative solution to the equation $m(\theta) = - \frac{\gamma_{n, \alpha}}{N_n}$. This follows from the following facts:
\begin{itemize}
\item The function $m$ is strictly increasing for sufficiently large $N_n$, since $m' = \sigma^2 > 0$.
\item We have $m(0) = -\frac{1}{N_n} \sum_{\VM \in \Pi}  \cdot \EB_{1, \bQ_{i, \VM}}[h_{\VM}] < - \frac{1}{N_n} \sum_{\VM \in \Pi} \EB_{0}[h_{\VM}]$ due to informativeness.
\item Lemma \ref{lem:mean_limit} implies that
\[
\lim_{\theta \to \infty} m(\theta) = \frac{1}{N_n} \sum_{\VM \in \Pi}\mathrm{esssup}(-h_{\VM})
= - \frac{1}{N_n} \sum_{\VM \in \Pi}\mathrm{essinf}(h_{\VM})
> -\frac{1}{N_n} \sum_{\VM \in \Pi} \EB_{0}[h_{\VM}].
\]
\begin{lem}[Asymptotic behavior of the tilted mean]
\label{lem:mean_limit}
Let $X$ be a real-valued random variable with CDF $F$, and define its MGF by $\phi(\theta) := \EB[\re^{\theta X}]$. Assume that $\phi_{\tau}(\theta)$ is finite for all $\theta \in [0, \infty)$.  
Let $m(\theta) := \frac{\rd}{\rd\theta} \log \phi(\theta)$ denote the mean of the exponentially tilted distribution. Then,
\begin{gather*}
\lim_{\theta \to \infty} m(\theta) 
= \mathrm{esssup}(X) := \inf\{x \in \RB : \PB(X \le x) = 1\}.
\end{gather*}
\end{lem}
We defer the proof of this lemma at the end of this section.
\item When $N_n$ is sufficiently large, Lemma \ref{lem:minimaxgamman} implies that $\frac{\gamma_{n, \alpha}}{N_n}$ concentrates to $\frac{1}{N_n}\sum_{\VM \in \Pi}  \EB_{0}[h_{\VM}]$, so that $- \frac{\gamma_{n, \alpha}}{N_n} \in [m(0), \lim\limits_{\theta \to \infty} m(\theta))$.
\end{itemize}

We denote by $\theta^{\star}$ the unique solution to the equation $m(\theta) = - \frac{\gamma_{n, \alpha}}{N_n}$.  
By Lemma~\ref{lem:bounded-optimal-theta}, we already know that $\theta^{\star} \le \overline{M}$.  
It therefore suffices to establish a lower bound $\underline{M}$ such that $\theta^{\star} \ge \underline{M}$.

Fix the interval $K := [0, \overline{M}]$.  
By the CGF regularity in Assumption~\ref{asmp:main-lower-bound}, for any $\theta \in K$ we have
\[
0 < \sigma^2_{\min}(K) \;\le\; m'(\theta) \;\le\; C_2(K).
\]
Consequently,
\[
m(0) + \sigma^2_{\min}(K) \cdot \theta^{\star} 
\;\le\; m(\theta^{\star}) 
= - \frac{\gamma_{n, \alpha}}{N_n} 
\;\le\; m(0) + C_2(K) \cdot \theta^{\star}.
\]
By Lemma~\ref{lem:minimaxgamman}, when $N_n$ is sufficiently large, $\frac{\gamma_{n, \alpha}}{N_n}$ concentrates around $\frac{1}{N_n}\sum_{\VM \in \Pi}\EB_{1, \QM_{1, \VM}}[h_{\VM}]$. 
Hence, for large $N_n$, it follows that
\[
 \frac{\frac{1}{N_n}\sum_{\VM \in \Pi} \bigl(\EB_{1, \QM_{1, \VM}}[h_{\VM}] - \EB_0 [h_{\VM}]\bigr)}{C_2(K)} 
 \;\le\; \theta^\star 
 \;\le\; \frac{\frac{1}{N_n}\sum_{\VM \in \Pi} \bigl(\EB_{1, \QM_{1, \VM}}[h_{\VM}] - \EB_0 [h_{\VM}]\bigr)}{\sigma^2_{\min}(K)}.
\]
We complete the proof by setting
\[
\underline{M} \;=\; \frac{1}{C_2([0, \overline{M}])} 
\inf_{\VM \in \Pi} \;\inf_{\bP_{\VM} \subseteq \PM_{\VM}^{\star}}
\Bigl[\EB_{1, \bP_{\VM}}[h_{\VM}] - \EB_0[h_{\VM}] \Bigr],
\]
which is positive by the informativeness condition in Assumption~\ref{asmp:main-lower-bound}.
\end{proof}

At the end, we provide the proof of Lemma \ref{lem:mean_limit} below.

\begin{proof}[Proof of Lemma \ref{lem:mean_limit}]
The function $m(\theta)$ is the expected value of $X$ under the exponentially tilted probability measure:
\[
m(\theta) = \frac{\int_{-\infty}^{\infty} x \re^{\theta x} \, \rd F(x)}{\int_{-\infty}^{\infty} \re^{\theta x} \, \rd F(x)}.
\]
Let $x^{\star} := \mathrm{esssup}(X) = \inf\{x \in \mathbb{R} : \PB(X \le x) = 1\}$. We aim to show that $\lim_{\theta \to \infty}(\theta) = x^{\star}.$

\paragraph{Step 1: Upper bound.} For any $\varepsilon > 0$, define $B := x^{\star} + \varepsilon$. Since $\PB(X > B) = 0$, the tilted mean satisfies
\[
m(\theta) \le \frac{\int_{-\infty}^{B} x \re^{\theta x} \, \rd F(x)}{\int_{-\infty}^{B} \re^{\theta x} \, \rd F(x)} \le B.
\]
Therefore, for all $\theta$, $m(\theta) \le x^{\star} + \varepsilon$. Taking $\limsup$ and then letting $\varepsilon \to 0$ gives
\[
\limsup_{\theta \to \infty} m(\theta) \le x^{\star}.
\]

\paragraph{Step 2: Lower bound.} Fix $\delta > 0$ and let $A := x^{\star} - \delta$. By the definition of $x^{\star}$, we have $\PB(X \ge A) > 0$. Decompose $m(\theta)$ as
\[
m(\theta) 
= \frac{\int_{x < A} x \re^{\theta x} \, \rd F(x) + \int_{x \ge A} x \re^{\theta x} \, \rd F(x)}{\int_{x < A} \re^{\theta x} \, \rd F(x) + \int_{x \ge A} \re^{\theta x} \, \rd F(x)}.
\]
Rewrite both numerator and denominator by factoring out $\re^{\theta A}$:
\[
m(\theta) 
= \frac{\int_{x < A} x \re^{\theta(x - A)} \, \rd F(x) + \int_{x \ge A} x \re^{\theta(x - A)} \, \rd F(x)}{\int_{x < A} \re^{\theta(x - A)} \, \rd F(x) + \int_{x \ge A} \re^{\theta(x - A)} \, \rd F(x)}.
\]
As $\theta \to \infty$, the integrals over $x < A$ vanish by the dominated convergence theorem, since $x - A < 0$ in this range and the MGF is finite. Thus,
\[
\lim_{\theta \to \infty} m(\theta) 
= \lim_{\theta \to \infty} \frac{\int_{x \ge A} x \re^{\theta(x - A)} \, \rd F(x)}{\int_{x \ge A} \re^{\theta(x - A)} \, \rd F(x)}.
\]
Since $x \ge A$ on the support of both integrals, we have the pointwise bound
\[
\frac{\int_{x \ge A} x \re^{\theta(x - A)} \, \rd F(x)}{\int_{x \ge A} \re^{\theta(x - A)} \, \rd F(x)} \ge A = x^{\star} - \delta.
\]
Therefore, $\liminf_{\theta \to \infty} m(\theta) \ge x^{\star} - \delta$. Since $\delta > 0$ is arbitrary, we conclude
\[
\liminf_{\theta \to \infty} m(\theta) \ge x^{\star}.
\]
Combining both steps, we have $\liminf_{\theta \to \infty} m(\theta) \ge x^{\star}$ and complete the proof.
\end{proof}

\subsection{Verification of Regularity Conditions for Considered Watermarks}
\label{sec:justassump}
In this section, we show that the required conditions are satisfied for the established optimal detection rules of the two watermarking schemes under study. The independence structure in Assumption~\ref{asmp:main} is already justified by the sensitivity of hash functions and therefore does not require further verification. We thus focus on the remaining conditions.

\subsubsection{Inverse Transform Watermark}
We begin with the inverse transform watermark, as its verification is more straightforward.
First, Assumption~\ref{asmp:heavy-tokensa} cannot be verified in practice since it is introduced as a simplifying assumption for theoretical analysis. Thus, it suffices to check the remaining two conditions in Assumption~\ref{asmp:main}.

On the one hand, the bounded variance condition in Assumption~\ref{asmp:main}\ref{subasmp:bounded-variance} holds immediately because $[h_{\VM}^{\mathrm{inv}}]{[-M, M]}$ is uniformly bounded by $M$.
On the other hand, the well-posedness condition in Assumption~\ref{asmp:main}~\ref{subasmp:well-posedness} is automatically satisfied because in deriving \eqref{eq:Inverse-R-lower-bound} we essentially set the minimizer to $\theta=1$, which is uniformly bounded.
To make this intuition more rigorous, we can instead argue more directly: for the score functions $\bh=\{h_{\VM}\}_{\VM \in \Pi}$, we have
\begin{align}
\bar{R}_{n, \Pow}(\bh)&
\overset{(a)}{\ge} \liminf_{|\Voca|\to \infty} D_{n, \Pow}(\bh) \nonumber  \\
&\overset{(b)}{\ge} -  \limsup_{|\Voca| \to \infty}  \left(\frac{\gamma_{n,\alpha}}{N_n}+\frac{1}{N_n} \sum_{\VM \in \Pi} 
 \sup_{\bP_{\VM} \subseteq \overline{\PM}_{\Delta_{\VM}}} \log \phi_{\bP_{\VM}, h_{\VM}}(1)
\right)  \nonumber  \\
&\overset{(c)}{\ge} - \limsup_{|\Voca| \to \infty} \frac{1}{N_n} \sum_{\VM \in \Pi} \left(\EB_0[h_{\VM}(Y_{\VM})] + \sup_{\bP_{\VM} \subseteq \overline{\PM}_{\Delta_{\VM}}} \log \phi_{\bP_{\VM}, h_{\VM}}(1)
\right)- \omega_{N_n} \nonumber \\
&\overset{}{=} -  \frac{1}{N_n} \sum_{\VM \in \Pi} \limsup_{|\Voca| \to \infty} \left(\EB_0[h_{\VM}(Y_{\VM})] + \sup_{\bP_{\VM} \subseteq \overline{\PM}_{\Delta_{\VM}}} \log \phi_{\bP_{\VM}, h_{\VM}}(1)
\right)- \omega_{N_n}. \tag*{\eqref{eq:Inverse-R-lower-bound}}
\end{align}
Here, $(a)$ follows from \eqref{eq:help1}, $(b)$ from setting $\theta=1$ in the definition of $D_{n, \Pow}(\bh)$ in \eqref{eq:Bndefinition}, and $(c)$ from Lemma~\ref{lem:minimaxgamman}.
As a result, we still arrive at \eqref{eq:Inverse-R-lower-bound}.
 
\subsubsection{Gumbel-max Watermark}
The main effort is devoted to the Gumbel-max watermark, as we need to verify both conditions in Assumption~\ref{asmp:main} as well as the additional Assumption~\ref{asmp:main-lower-bound}, which is required for establishing the asymptotic tightness in Remark~\ref{rem:tight-lower-bound}.
For clarity, we fix a minimal unit $\VM$ and denote the two proposed score functions as
\begin{equation}
\label{eq:two-hs}
h_{\bS_{\Delta}^\star}(y) := 
\frac{(|\VM|\wedge|\Voca|)\,\Delta}{(|\VM|\wedge|\Voca|-1)(1-\Delta)}\,
\log y
\quad \text{and} \quad
h_{\bP_{\Delta}^\star}(y) := \log \left(y^{\frac{\Delta}{1-\Delta}}+y^{\frac{1-\Delta}{\Delta}}\right).
\end{equation}

\paragraph{Verification of Assumption \ref{asmp:main}.}
Note that both optimal scores $h_{\bS_{\Delta}^\star}$ and $h_{\bP_{\Delta}^\star}$ are log-likelihood ratio functions corresponding to the least-favorable distribution vectors $\bS_{\Delta}^\star$ and $\bP_{\Delta}^\star$, respectively. By direct computation, their moment generating functions are finite and their variances are uniformly bounded. Under these optimal scores, the minimization in $\theta$ is achieved at $\theta=1$, which is uniformly bounded. Hence, the well-posedness condition is satisfied.

\paragraph{Verification of Assumption \ref{asmp:main-lower-bound}.}
There are four conditions in Assumption~\ref{asmp:main-lower-bound}, and we verify them one by one.
\begin{enumerate}[(i)]
    \item \textbf{\textit{(Finite maximizers)}}
    This condition follows from Lemma~\ref{lem:maximum-CDF-polyhedron}, together with the fact that both $h_{\bS_{\Delta}^\star}$ and $h_{\bP_{\Delta}^\star}$ are increasing. Hence, only $\bS_{\Delta}^\star$ and $\bP_{\Delta}^\star$ can serve as least-favorable distribution vectors.
    
    \item \textbf{\textit{(Informative scores)}}
    By Lemma~\ref{lem:gumbel-alternative}, the null and alternative CDFs are given by $F_0(y) = y$ and $F_{\bS}(y) = (\sum_{w}S_{w})^{-1} \sum_{w} S_{w} y^{1/S_{w}}$ with $\bS = (S_1,\ldots,S_{|\Voca|})$, respectively. From Lemma~\ref{lem:domain}, we know that $S_w \in [0,1-\Delta)$ for any $w$, so $S_w < 1$ and thus $y^{1/S_w} < y$ for any $y \in (0,1)$. Consequently,
    \begin{align*}
        F_{\bS}(y) = \frac{\sum_{w}S_{w}y^{1/S_{w}}}{\sum_{w}S_{w}} < \frac{\sum_{w}S_{w}y}{\sum_{w}S_{w}} = y = F_0(y).
    \end{align*}
    Thus, the alternative distribution of $Y$ is stochastically dominated by the null distribution. Since both $h_{\bS_{\Delta}^\star}$ and $h_{\bP_{\Delta}^\star}$ are strictly increasing, integration by parts shows that $\EB_{1,\bP_{\mathcal{V}}}[h(Y)] > \EB_{0}[h(Y)]$ for $h \in \{h_{\bS_{\Delta}^\star}, h_{\bP_{\Delta}^\star}\}$.

    \item \textbf{\textit{(CGF regularity)}}
    Recall that the CGF is defined as the logarithm of the MGF. Formally, for a score function $h$, the MGF is $\phi_{\bS}(\theta) = \EB_{1,\bS}[\exp(-\theta h(Y_{\mathcal{V}}))]$ and the CGF is given by $K(\theta) = \log \phi_{\bS}(\theta)$. For simplicity, we denote the  alternative density by $f_{\bS}(y) = F'_{\bS}(y) = (\sum{w'} S_{w'})^{-1} \sum_w y^{1/S_w - 1}$.
    \begin{itemize}
        \item For $h_{\bS_{\Delta}^\star}$, the MGF takes the explicit form
    \[
    \phi_{\bS}(\theta) 
    = \int_0^1 \re^{-\theta c \log y} f_{\bS}(y)\rd y 
    = \int_0^1 y^{-\theta c} \frac{\sum_w y^{1/S_w - 1}}{\sum_{w'}S_{w'}} \rd y
    = \frac{1}{\sum_{w'}S_{w'}} \sum_w \frac{S_w}{1 - \theta c S_w}.
    \]
    \item For $h_{\bP_{\Delta}^\star}$, the MGF is
    \[
    \phi_{\bS}(\theta) = \int_0^1 (y^{\frac{1-\Delta}{\Delta}} + y^{\frac{\Delta}{1-\Delta}})^{-\theta} f_{\bS}(y) dy.
    \]
    \end{itemize}
    On any compact set $[0,\overline{M}]$, the derivatives of $\log \phi_{\bS}(\theta)$ are smooth and bounded, which yields uniform upper bounds. For the lower bound, note that $K''(\theta)$ equals the variance of $-h(Y)$ under a tilted measure, which is strictly positive as long as $h(Y)$ is not constant. By continuity, $K''(\theta)$ is uniformly bounded below on $[0,\overline{M}]$ by some constant $\sigma_{\min}^2(K) > 0$. These constants can be chosen independently of any specific $\bS \in \mathcal{D}_\Delta$, thanks to compactness of $[0,\overline{M}]$ and smoothness of the CGF.
    
    \item \textbf{\textit{(Score density regularity)}}
    The last condition is verified directly by Lemma~\ref{lem:bounded-TV}.    
\end{enumerate}

\begin{lem}[Bounded total variation]
\label{lem:bounded-TV}
Let $\mathrm{TV}(\rho) := \int_{-\infty}^{\infty} |\rho'(x)| \rd x$ denote the total variation of a PDF $\rho$.
When $|\VM|=1$ and $\Delta \in (0, 1/2)$, with $h_{\bS_{\Delta}^\star}$ and $h_{\bP_{\Delta}^\star}$ defined in \eqref{eq:two-hs}, we have a universal constant $C_R >0$ such that
\[
\mathrm{TV}(h_{\bS_{\Delta}^\star}) \le |\Voca|
\qquad \text{and} \qquad
\mathrm{TV}(h_{\bP_{\Delta}^\star}) \le |\Voca| + C_R.
\]
\end{lem}

\begin{proof}[Proof of Lemma \ref{lem:bounded-TV}]
Let $Z = h(Y)$ denote the score for a minimal unit $\mathcal{V}$.
When $|\VM|=1$, there is no repetition and each $\bS$ reduces to $\bP$.
By Lemma \ref{lem:gumbel-alternative}, the alternative PDF of $Y$ is 
\[
f_{\bP}(y) = F'_{\bP}(y) = \sum_{w} y^{1/P_w - 1}.
\]
By a change of variables, the PDF of $Z$, denoted by $\rho_{\bP}$, is
$\rho_{\bP}(z) = f_{\bP}(g(z))|g'(z)|$, where $y = g(z)$ is the inverse of $z = h(y)$.

\paragraph{Case 1: $h_{\bS_{\Delta}^\star}(y) = C \log y$.}  
Here $C = \frac{\Delta}{1-\Delta}$, and the inverse function is $y = g(z) = e^{z/C}$ for $z \in (-\infty,0)$ with derivative $g'(z) = \tfrac{1}{C}e^{z/C}$. The density of $Z$ is
\[
\rho_{\bP}(z) = f_{\bP}(e^{z/C}) \cdot \frac{1}{C}e^{z/C} 
= \frac{1}{C} \sum_{w} e^{z/(C P_w)}.
\]
Since $\rho_{\bP}'(z) > 0$ for $z \in (-\infty,0)$, it follows that
\[
\mathrm{TV}(\rho_{\bP}) 
= \int_{-\infty}^0 \rho_{\bP}'(z)\,\rd z 
= \rho_{\bP}(0) - \lim_{z \to -\infty}\rho_{\bP}(z)
= \frac{|\Voca|}{C}
\le |\Voca|,
\]
where the last inequality holds because $C>1$ when $\Delta \in (0,1/2)$.

\paragraph{Case 2: $h_{\bP_{\Delta}^\star}(y) = \log(y^C + y^{1/C})$.}  
Here $C = \frac{\Delta}{1-\Delta} \in (0,1)$. By definition,
\[
\mathrm{TV}(\rho_{\bP}) 
= \int |\rho_{\bP}'(z)| \, \rd z 
= \int_0^1 \left| \frac{d}{dy}\left(\frac{f_{\bP}(y)}{h'(y)}\right) \right|\rd y
\le \underbrace{\int_0^1 \left|\frac{f_{\bP}'(y)}{h'(y)}\right|\rd y}_{(I)}
+ \underbrace{\int_0^1 f_{\bP}(y)\left|\frac{h''(y)}{(h'(y))^2}\right|\rd y}_{(II)}.
\]

We first analyze the term (II). For simplicity, we define $R(y) := \big|h''(y)/(h'(y))^2\big|$, which is independent of $\bP$.  
As $y \to 1$, $h'(1) = (C+1/C)/2 \neq 0$, so $R(1)<\infty$.  
As $y \to 0$, $h'(y) \sim C/y$ and $h''(y) \sim -C/y^2$, hence $\lim_{y \to 0} R(y) = 1/C$.  
Since $R(y)$ is continuous on $(0,1]$ with finite boundary limits, it is uniformly bounded by some constant $C_R < \infty$. Thus,
\[
(II) \le C_R \int_0^1 f_{\bP}(y)\rd y = C_R.
\]

Next, we analyze the term (I). Since $P_w < 1$ for any $\bP \in \mathcal{P}_\Delta$, we have $f_{\bP}'(y) = \sum_w (1/P_w - 1) y^{1/P_w - 2} > 0$, so the absolute value can be removed. Integration by parts gives
\[
(I) = \int_0^1 \frac{f_{\bP}'(y)}{h'(y)}\rd y
= \left[\frac{f_{\bP}(y)}{h'(y)}\right]_0^1 + \int_0^1 f_{\bP}(y)\frac{h''(y)}{(h'(y))^2}dy.
\]
At $y=1$, $f_{\bP}(1)=|\Voca|$.  
As $y\to 0$, using $h'(y)\sim C/y$,
\[
\lim_{y\to 0} \frac{f_{\bP}(y)}{h'(y)} 
= \frac{1}{C}\lim_{y\to 0}\sum_w y^{1/P_w} = 0.
\]
The integral term is bounded in magnitude by (II). Hence,
\[
(I) \le \frac{|\Voca|}{h'(1)} + C_R = \frac{2|\Voca|}{C+1/C} + C_R.
\]
Combining the bounds for (I) and (II) yields
\[
\mathrm{TV}(\rho_{\bP}) \le \frac{2|\Voca|}{C+1/C} + 2C_R \le |\Voca| + 2C_R
\]
which is finite and uniform over $\bP \in \mathcal{P}_\Delta$.

In both cases, $\mathrm{TV}(\rho_{\bP})$ is finite and bounded by a multiple of $|\Voca|$, completing the proof.
\end{proof}

\section{Proof for Gumbel-max Watermarks in Section \ref{sec:gumbel}}
 
\label{append:gumbel}

\subsection{Proof of Lemma \ref{lem:gumbel-alternative}}
\begin{proof}[Proof of Lemma \ref{lem:gumbel-alternative}]
We assert that $Y_{t_1} = \cdots = Y_{t_k}$ if and only if $w_{t_1} =\cdots = w_{t_k}$. This follows from the fact that if $w_{t_1} \neq w_{t_2}$, then $Y_{t_1}$ and $Y_{t_2}$ are independent by Assumption \ref{asmp:subblocks-independence}. Since each $Y_t$ has a smooth CDF, the probability $\PB_1(Y_{t_1} = Y_{t_2} \mid w_{t_1} \neq w_{t_2}) = 0$, making such an event almost surely impossible.

Recall that the NTP distribution for $w_{t_i}$ is given by $\bP_{t_i}$. Since the same pseudorandom variable is used to generate all $w_{t_i}$ for $t_i \in \VM$, we denote it by $\zeta = (U_w)_{w \in \Voca}$. Consequently, each token satisfies $w_{t_i} = \SM(\bP_{t_i}, \zeta)$ for all $t_i \in \VM$.
Under the event $Y_{t_1} = \cdots = Y_{t_k}$, we define $w_{t_1} = \cdots = w_{t_k} = w$. This implies that $w = \SM(\bP_{t}, \zeta)$ for all $t \in \VM$.
Therefore, it follows that
\begin{align*}
\PB_1(Y_{t_1} \le y ~|~ Y_{t_1}  = \cdots = Y_{t_k}, \bP_{\VM})
&=\PB_1(Y_{t_1} \le y,~w_{t_1} = \cdots = w_{t_k}  |\, Y_{t_1} = \cdots = Y_{t_k}, \bP_{\VM})\\
&=\sum_{w \in \Voca}\PB_1(Y_{t_1} \le y,~w_{t_1} = \cdots = w_{t_k} =w |\, Y_{t_1} = \cdots = Y_{t_k}, \bP_{\VM})\\
&=\sum_{w \in \Voca} \PB_1(Y(w,\zeta) \le y ~|~ w = \SM(\bP_{t}, \zeta),~\forall t \in \VM)\\
&=\sum_{w \in \Voca}\PB_1\left( U_{w}\leq y~\Big|~U_{w'} \leq U_{\token}^{\max_{t \in \VM}\left(\frac{P_{t,w'}}{P_{t,w}}\right)}, \forall w' \ne \token\right).
\end{align*}
Given that $\{U_w\}_{w \in \Voca}$ are i.i.d. $\UM(0, 1)$, direct calculation yields that
\begin{align*}
\PB_1\left( U_{w}\leq y, U_{w'} \leq U_{\token}^{\max_{t \in \VM}\left(\frac{P_{t,w'}}{P_{t,w}}\right)}, \forall w' \ne \token\right)
= S_w y^{1/S_w}.
\end{align*}
As a result, it follows from Bayes' theorem that
\[
\PB_1\left( U_{w}\leq y~\Big|~ U_{w'} \leq U_{\token}^{\max_{t \in \VM}\left(\frac{P_{t,w'}}{P_{t,w}}\right)}, \forall w' \ne \token\right)
= \frac{S_w y^{1/S_w}}{\sum_{w' \in \Voca} S_{w'}}.
\]
Summing the last probability over all $w \in \Voca$ completes the proof.
\end{proof}

\subsection{Proof of Lemma \ref{lem:saddle}}
\begin{proof}[Proof of Lemma \ref{lem:saddle}]

\textbf{Proof of the ``if'' direction.}
If there exists a vector $\bS^\star \in \DM_{\Delta}$ such that
\begin{equation} 
\max_{\bS \in \DM_{\Delta}} L(h_{\bS^{\star}}, \bS) =   L(h_{\bS^{\star}}, \bS^\star),
\tag*{\eqref{eq:sup-over-S}}
\end{equation}
on the one hand, it follows from the Donsker–Varadhan representation that
\[
\min_h \max_{\bS \in \DM_{\Delta}} L(h, \bS)
\ge \min_h L(h, \bS^{\star}) = -\mathrm{KL}(F_0 \| F_{\bS^\star}).
\]
On the other hand, by the condition \eqref{eq:sup-over-S}, it follows that
\begin{align*}
\min_h \max_{\bS \in \DM_{\Delta}} L(h, \bS)
\le \max_{\bS \in \DM_{\Delta}} L(h_{\bS^\star}, \bS)
=L(h_{\bS^{\star}}, \bS^\star)= -\mathrm{KL}(F_0 \| F_{\bS^\star}).
\end{align*}
Combining the above two directions, we know that $(h_{\bS^{\star}}, \bS^{\star})$ is a solution pair of the minimax problem \eqref{eq:target-minimax}.

\textbf{Proof of the ``only if'' direction.} 
Suppose the pair $(h^{\star}, \bS^{\star})$ solves the minimax problem \eqref{eq:target-minimax}. By definition, we have
\begin{align*} L(h^{\star}, \bS^{\star}) = \min_{h} \max_{\bS \in \DM_{\Delta}} L(h, \bS) = \max_{\bS \in \DM_{\Delta}} L(h^{\star}, \bS) = \min_{h} L(h, \bS^{\star}). \end{align*} The last equality holds if and only if $h^{\star} = h_{\bS^{\star}}$ (up to a constant shift), by the Donsker–Varadhan representation. The second equality corresponds exactly to the optimality condition \eqref{eq:sup-over-S}.

\end{proof}

\subsection{Proof of Lemma \ref{lem:domain}}
 
\begin{proof}[Proof of Lemma \ref{lem:domain}]
Recall that for any $\bS \in \DM_{\Delta}$ there exists  $\bP_{\VM} \subseteq \FPM$ such that for each $w \in \Voca$,
\begin{equation}
\label{eq:def-Sw}
S_w = \left(\sum_{w'\in \Voca}  \max_{t \in \VM}\frac{P_{t,w'}}{P_{t, w}}  \right)^{-1}.
\end{equation}
According to this definition, the permutation invariance of $ \DM_{\Delta} $ follows directly by permuting the order of entries in each NTP distribution in $ \bP_{\VM} $.
We now turn to prove the remaining part.
Using the fact that $ \frac{P_{t,w'}}{P_{t, w}} \le \max_{t \in \VM}\frac{P_{t,w'}}{P_{t, w}} \le \sum_{t \in \VM}\frac{P_{t, w'}}{P_{t, w}}$ for any $t \in \VM$, we have that
\begin{equation}
\label{eq:relation-S}
\left( \sum_{t \in \VM} \frac{1}{P_{t, w}}-(|\VM|-1)  \right)^{-1}\le S_w \le \min_{t \in \VM} P_{t, w}.
\end{equation}

With this result, we are now ready to prove the three bullet points. \begin{enumerate}
 \item[(i)] By the definition in \eqref{eq:def-Sw}, it is clear that $0 \le S_w$. Using \eqref{eq:relation-S}, it follows that $S_w \le \min_{t \in \VM} P_{t, w} \le 1-\Delta$ due to $\bP_{\VM} \subseteq \FPM$.
 \item[(ii)] By \eqref{eq:relation-S}, we have that $\sum_{w} S_w \le \sum_{w} \min_{t \in \VM} P_{t, w} \le \sum_w P_{t, w} =1$.
 \item[(iii)] By some algebraic manipulation, the target inequality $\frac{\max_{w} S_w}{1-\Delta} \le 1 - \frac{1-\sum_{w} S_w }{|\VM|\wedge |\Voca|}$ is equivalent to 
 \begin{equation}
    \label{eq:third-condition}
\left(1 + \frac{\Delta}{|\VM|\wedge |\Voca|-1} \right)S_w \le (1-\Delta)\cdot \left( \frac{\sum_{w'\neq w} S_{w'}}{|\VM|\wedge |\Voca|-1} +1\right),~\forall~w \in \Voca.
 \end{equation}
 We then turn to prove \eqref{eq:third-condition}. 
 Fix any $w \in \Voca$.
 If $S_w = 0$, then \eqref{eq:third-condition} holds trivially. Otherwise, if $S_w > 0$, then by the relation \eqref{eq:relation-S}, we have $0 < S_w \le \min_{t \in \VM} P_{t, w}$. This implies that $P_{t, w}$ is strictly positive for all indices $t \in \VM$.

In this case, on the one hand, it follows that
\begin{align}
\label{eq:Lp}
S_w 
=  \left(1 + \sum_{w' \neq w}  \max_{t \in \VM}\frac{P_{t,w'}}{P_{t, w}}  \right)^{-1}
\le
\left(1 +   \frac{\sum_{w' \neq w}\max_{t \in \VM} P_{t,w'}}{1-\Delta}\right)^{-1}
\end{align}
where the inequality holds because $P_{t, w} \le 1-\Delta$ for all $t$ and $w$.

On the other hand, we have that
\begin{align}
\label{eq:Rp}
\sum_{w'\neq w} S_{w'}
&= \sum_{w'\neq w} \left(\sum_{j\in \Voca}  \max_{t \in \VM}\frac{P_{t,j}}{P_{t, w'}}  \right)^{-1} 
\nonumber \\
&\ge  \sum_{w'\neq w} \left(1+\sum_{j \neq w'}  \frac{\max_{t \in \VM}P_{t,j}}{\min_{t \in \VM}P_{t, w'}}  \right)^{-1} \\
&=  \sum_{w'\neq w} \min_{t \in \VM}P_{t, w'} \cdot  \left(\min_{t \in \VM}P_{t, w'} +\sum_{j \neq w'}  \max_{t \in \VM}P_{t,j}  \right)^{-1}\nonumber \\
&\ge \frac{\sum_{w'\neq w} \min_{t \in \VM}P_{t, w'} }{1-\Delta +\sum_{w' \neq w}  \max_{t \in \VM}P_{t,w'}}\nonumber,
\end{align}
where the last inequality follows from the fact that, for any $w' \neq w$,
\begin{align*}
\min_{t \in \VM}P_{t, w'} +\sum_{j \neq w'}  \max_{t \in \VM}P_{t,j} 
&\le \min_{t \in \VM}P_{t, w'} +  \max_{t \in \VM}P_{t,w}  + \sum_{j \notin \{w, w'\}}  \max_{t \in \VM}P_{t,j} \\
&\le\max_{t \in \VM}P_{t, w'} + (1-\Delta)  + \sum_{j \notin \{w, w'\}}  \max_{t \in \VM}P_{t,j} \\
&\le 1-\Delta + \sum_{w' \neq w}  \max_{t \in \VM}P_{t,w'}.
\end{align*}

We observe that \eqref{eq:Lp} provides an upper bound for the left-hand side of \eqref{eq:third-condition} in terms of $ \sum_{w' \neq w} \max_{t \in \VM} P_{t,w'} $, while \eqref{eq:Rp} provides a lower bound for the right-hand side of \eqref{eq:third-condition} involving both $ \sum_{w' \neq w} \max_{t \in \VM} P_{t,w'} $ and $ \sum_{w' \neq w} \min_{t \in \VM} P_{t,w'} $. 
To connect these bounds, we use the following fact that bridges both $ \sum_{w' \neq w} \max_{t \in \VM} P_{t,w'} $ and $ \sum_{w' \neq w} \min_{t \in \VM} P_{t,w'} $:  
\begin{equation}
\label{eq:Lp<Rp}
\sum_{w'\neq w} \min_{t \in \VM} P_{t, w'} + (|\VM|\wedge |\Voca|-1) \cdot \sum_{w' \neq w} \max_{t \in \VM}  P_{t, w'}
  \ge( |\VM|\wedge |\Voca|)\cdot \Delta,
\end{equation}
which follows because $ \min_{t \in \VM} P_{t, w'}  + \sum_{j \notin \{w, w'\}}\max_{t \in \VM}  P_{t, j} \ge 1-\max_{t \in \VM} P_{t, w} \ge \Delta$.

Combining the inequalities \eqref{eq:Lp}, \eqref{eq:Rp}, and \eqref{eq:Lp<Rp}, we complete the proof of \eqref{eq:third-condition} for a fixed $ w $. Since the same argument holds for all $ w $, this establishes \eqref{eq:third-condition}.

\item[(iv)] Finally, we show that $\bS_{\Delta}^{\star} := \left(\frac{1-\Delta}{1+\frac{\Delta}{|\VM|\wedge |\Voca|-1}}, 0, 0, \ldots, 0\right) \in \DM_{\Delta}$ by explicit construction. We consider two cases based on the relative sizes of $|\Voca|$ and $|\VM|$:

\begin{itemize}
\item 
If $|\Voca| \ge |\VM| + 1$, we construct the first NTP distribution $\bP_t$ as follows:
\begin{equation}
\label{eq:constructionv}
\bP_{t} = \begin{pmatrix}
1 - \Delta, & \frac{\Delta}{|\VM| - 1}, & \frac{\Delta}{|\VM| - 1}, & \cdots, & \frac{\Delta}{|\VM| - 1}, & 0, & \cdots, & 0
\end{pmatrix}.
\end{equation}
For $i = 2, \ldots, |\VM| + 1$, we define the $i$-th NTP distribution by setting the first entry to $1 - \Delta$, the $i$-th entry to zero, and all other entries among the first $|\VM| + 1$ positions to $\frac{\Delta}{|\VM| - 1}$. A direct computation shows that all such distributions yield this very $\bS$-vector:
\[
\bS_{\Delta}^{\star} = \begin{pmatrix}
\frac{1 - \Delta}{1 + \frac{\Delta}{|\VM| - 1}}, & 0, & \cdots, & 0
\end{pmatrix}.
\]

\item 
If $|\Voca| \le |\VM|$, we instead construct the first NTP distribution as
\[
\bP_{t} = \begin{pmatrix}
1 - \Delta, & \frac{\Delta}{|\Voca| - 1}, & \frac{\Delta}{|\Voca| - 1}, & \cdots, & \frac{\Delta}{|\Voca| - 1}, & 0, & \cdots, & 0
\end{pmatrix},
\]
and define the remaining NTP distributions by cyclically shifting the zero entry among the first $|\Voca| + 1$ positions while keeping the first entry fixed at $1 - \Delta$. The argument mirrors the previous case and is omitted for brevity.
\end{itemize}

\end{enumerate}
 
\end{proof}

\subsection{Proof of Lemma \ref{lem:schur-convex}}
We first introduce an ancillary lemma that establishes the Schur-convexity of CDF, that is the mapping $ \bS \mapsto F_\bS(y) $, in Lemma \ref{lem:schur-convex-CDF}.

\begin{lem}[Schur-convexity]
\label{lem:schur-convex-CDF} 
For any $y \in [0, 1]$, the map $\bS \mapsto F_\bS(\rd y)$ is Schur-convex in $\bS$.
\end{lem}

\begin{proof}[Proof of Lemma \ref{lem:schur-convex-CDF}]
For simplicity, we define $ G(\bS) = F_\bS(y) $ for any fixed $ y \in [0,1] $. It is straightforward to verify that (i) $ G $ is invariant under permutations of its coordinates, meaning that $ G(\bS) = G(\pi(\bS)) $ for any permutation $ \pi \in \mathrm{Perm}(\Voca) $, and (ii) all first partial derivatives of $ G $ exist. By the Schur–Ostrowski criterion, $ G $ is Schur-convex in $ \bS $ if and only if, for any $ \bS \in \RB^d $ and any $ w, w' \in \Voca $, the following condition holds:
\[
(S_{w}-S_{w'}) \left( \frac{\partial G}{\partial S_w} - \frac{\partial G}{\partial S_{w'}} \right) \ge 0.
\]
Direct calculation shows that 
\[
\frac{\partial (S_wy^{1/S_w})}{\partial S_w} = y^{1/S_w}\left(1+\frac{\ln \frac{1}{y}}{S_w}\right)
\quad \text{and} \quad
\frac{\partial^2 (S_wy^{1/S_w})}{\partial^2 S_w} = y^{1/S_w}\frac{\ln^2 \frac{1}{y}}{S_w^3}.
\]
As a result,
\begin{align*}
(S_{w}-S_{w'}) \left( \frac{\partial G}{\partial S_w} - \frac{\partial G}{\partial S_{w'}} \right)
&=\frac{(S_{w}-S_{w'})}{\sum_{w} S_w} \left[y^{1/S_w}\left(1+\frac{\ln \frac{1}{y}}{S_w}\right)-y^{1/S_{w'}}\left(1+\frac{\ln \frac{1}{y}}{S_{w'}}\right) \right]\\
&=\frac{(S_{w}-S_{w'})^2}{\sum_{w} S_w} \cdot y^{1/\widetilde{S}_w}\frac{\ln^2 \frac{1}{y}}{\widetilde{S}_w^3} \ge 0
\end{align*}
where the last equation uses the mean value theorem and $\widetilde{S}_w$ lies between $S_w$ and $S_{w'}$.
\end{proof}

Now, using Lemma \ref{lem:schur-convex-CDF}, we can proceed to prove Lemma \ref{lem:schur-convex}.

\begin{proof}[Proof of Lemma \ref{lem:schur-convex}]
For any non-decreasing score function $h$, by integration by parts, we have that
\begin{equation}
\label{eq:integral-CDF}
\EB_{F_\bS} [\re^{-h(Y_{\VM})}] = \int \re^{-h(y)} F_\bS(\rd y) = \re^{-h(1)} + \int_0^1 F_\bS(y) \re^{-h(y)} h(\rd y).
\end{equation}
This implies that $ \EB_{F_\bS} [\re^{-h(Y_{\VM})}] $ is a non-negative weighted sum of $ F_\bS(y) $ evaluated over all possible values of $ y $.  
By Lemma \ref{lem:schur-convex-CDF}, we know that the mapping $ \bS \mapsto F_\bS(y) $ is Schur-convex in $ \bS $ for any fixed $ y \in [0,1] $. Using this result and applying Definition \ref{def:schur}, we conclude that the function $ \bS \mapsto \int e^{-h(y)} F_\bS(\rd y) $ is isotonic, order-preserving, and therefore Schur-convex in $ \bS $.

\end{proof}

\subsection{Proof of Lemma \ref{lem:maximum}}
\begin{proof}[Proof of Lemma \ref{lem:maximum}] 
The integration by parts implies that
\begin{equation}
\tag{\ref{eq:integral-CDF}}
 \int \re^{-h(y)} F_\bS(\rd y) = \re^{-h(1)} + \int F_\bS(y) \re^{-h(y)} h(\rd y).
\end{equation}
It suffices to prove that 
\[
\max_{\bS \in \DM_{\Delta}} \int F_\bS(y) \re^{-h(y)} h(\rd y)
\le \max_{\bS \in \DM_{\Delta} \cap \HM_{\Delta}}  \int F_\bS(y) \re^{-h(y)} h(\rd y).
\]
To achieve this, it suffices to prove that
\begin{equation}
\label{eq:target1}
\max_{\bS \in \DM_{\Delta} \setminus \HM_{\Delta}}  \int F_\bS(y) \re^{-h(y)} h(\rd y) \le \int F_{\bS_{\Delta}^{\star}}(y) \re^{-h(y)} h(\rd y).
\end{equation}
where $\bS_{\Delta}^{\star} := \left(\frac{1-\Delta}{1+\frac{\Delta}{|\VM|\wedge |\Voca|-1}}, 0, 0, \ldots, 0\right)$.

From Lemma \ref{lem:domain}, it follows that $\bS_{\Delta}^{\star} \in \DM_{\Delta}$. 
By the definition of $\HM_{\Delta}$, it is clear that $\bS_{\Delta}^{\star} \in \HM_{\Delta}$.
Consequently, $\bS_{\Delta}^{\star} \in \DM_{\Delta} \cap \HM_{\Delta}$.
With this result, once \eqref{eq:target1} is established, it follows that
\[
\max_{\bS \in \DM_{\Delta} \setminus \HM_{\Delta}}  \int F_\bS(y) \re^{-h(y)} h(\rd y)
\le \int F_{\bS_{\Delta}^{\star}}(y) \re^{-h(y)} h(\rd y)
\le \max_{\bS \in \DM_{\Delta} \cap \HM_{\Delta}}  \int F_\bS(y) \re^{-h(y)} h(\rd y),
\]
which completes the proof.

In the following, we will prove \eqref{eq:target1}.
For any $\bS \in \DM_{\Delta} \setminus \HM_{\Delta}$, by Eqn. \eqref{eq:integral-CDF}, Lemma \ref{lem:schur-convex}, and the definition of Schur-convexity, it follows that
\[
\int F_\bS(y) \re^{-h(y)} h(\rd y)
\le \int F_{\bS_1}(y) \re^{-h(y)} h(\rd y),
\]
 where $\bS_1 := \left(\sum_{w} S_w, 0, \ldots, 0\right)$ majorizes the given $\bS$.
Next, note that $\bS \in \DM_{\Delta} \setminus \HM_{\Delta}$ implies $\sum_{w} S_w \le \frac{1-\Delta}{1+\frac{\Delta}{|\VM|\wedge |\Voca|-1}}$. Since $F_{\bP}(y) = y^{1/S_1}$ is increasing in $S_1$ when $\bS$ has only one non-zero entry (that is, $\bS = (S_1, 0, \ldots, 0)$), we deduce that $F_{\bS_1}(y) \le F_{\bS_{\Delta}^{\star}}(y)$ for any $y \in [0, 1]$.
As a result, the last inequality holds, which completes the proof. 
\end{proof}

\subsection{Proof of Lemma \ref{lem:extreme-points}}
\begin{proof}[Proof of Lemma \ref{lem:extreme-points}]
\begin{enumerate}
    \item 
    Since $\KM_{\Delta}$ is the intersection of several half-spaces, it forms a convex polyhedron. We now prove that the extreme points of $\KM_{\Delta}$ are precisely the elements of $\EM_{\Delta}$. 
    
    First, observe that both $\bP_{\Delta}^{\star}$ and $\bS_{\Delta}^{\star}$ belong to $\KM_{\Delta}$, and by the permutation invariance of $\KM_{\Delta}$, we have $\EM_{\Delta} \subseteq \KM_{\Delta}$. This implies $\mathrm{conv}(\EM_{\Delta}) \subseteq \KM_{\Delta}$. 
    
    To prove the reverse inclusion, it suffices to show that any point in $\KM_{\Delta}$ can be expressed as a convex combination of points in $\EM_{\Delta}$. Consider an arbitrary $\bS \in \KM_{\Delta}$, and define $C = \sum_w S_w$ as the sum of its coordinates. By the definition of $\KM_{\Delta}$, the largest coordinate of $\bS$ satisfies  
    \[
    \max_{w}S_w\leq (1-\Delta) \left( 1 - \frac{1-C}{|\VM|\wedge |\Voca|} \right).
    \]
    We assert that $\bS$ is majorized by the following vector 
    \[
    \sS_{\mathrm{new}} = \left( (1-\Delta) (1 - \frac{1-C}{|\VM|\wedge |\Voca|}), C - (1-\Delta) (1 - \frac{1-C}{|\VM|\wedge |\Voca|}), 0, \dots \right),
    \]
    which, by definition, belongs to $\KM_{\Delta}$.
    By Lemma \ref{lem:convex-combination}, $\bS$ can thus be expressed as a convex combination of permutations of $\sS_{\mathrm{new}}$. Since $\sS_{\mathrm{new}}$ itself is a convex combination of $\bP_{\Delta}^{\star}$ and $\bS_{\Delta}^{\star}$, it follows that $\bS$ is a convex combination of points in $\EM_{\Delta}$. This completes the proof. 
    
    \begin{lem}[\citep{marshall1979inequalities}]
    \label{lem:convex-combination}
    Given two vectors $\vx, \vy \in \RB^d$, if $\vx$ majorizes $\vy$, then $\vy$ is a convex combination of $\vx$ and its permutations.
    \end{lem}

    \item   By Lemma \ref{lem:domain}, we know that $ \bS_{\Delta}^{\star} \in \DM_{\Delta} $. Additionally, we have $ \bP_{\Delta}^{\star} \in \DM_{\Delta} $ because setting $ \bP_{t_1} = \cdots = \bP_{t_k} = \bP_{\Delta}^{\star} $ results in a corresponding $ \bS $-vector equal to $ \bP_{\Delta}^{\star} $, which, by definition, belongs to $ \DM_{\Delta} $. Combining these observations with the permutation invariance of $ \DM_{\Delta} $, we conclude that $ \EM_{\Delta} \subseteq \DM_{\Delta} $.  
    On the other hand, by definition, we also have $ \EM_{\Delta} \subseteq \KM_{\Delta} $. The conclusion then follows.
    
    \item 
    It suffices to prove that $\DM_{\Delta} \cap \HM_{\Delta} \subseteq \KM_{\Delta} \subseteq \mathrm{conv}(\DM_{\Delta} \cap \HM_{\Delta})$ since $\KM_{\Delta}$ is convex.
    By Lemma \ref{lem:domain}, we know that $ \DM_{\Delta} \cap \HM_{\Delta} \subseteq \KM_{\Delta} $. We now turn to the opposite direction.  
    By the first point, we have $ \KM_{\Delta} = \mathrm{conv}(\EM_{\Delta}) $. We have that $ \EM_{\Delta} \subseteq \DM_{\Delta} \cap \HM_{\Delta} $ from the second point. Consequently, it follows that $\KM_{\Delta} \subseteq \mathrm{conv}(\DM_{\Delta} \cap \HM_{\Delta})$, which completes the proof.    
    \end{enumerate}
\end{proof}

\subsection{Proof of Lemma \ref{lem:maximum-CDF-polyhedron}}
\begin{proof}[Proof of Lemma \ref{lem:maximum-CDF-polyhedron}]

We note that
\[
\sup_{\bS \in \KM_{\Delta}}  \int F_\bS(y) \re^{-h(y)} h(\rd y)
= \sup_{C \in [ \frac{1-\Delta}{1+\frac{\Delta}{|\VM|\wedge |\Voca|-1}}, 1 ]} \sup_{\bS \in \KM_{\Delta}: \sum_{w} S_w = C} \int F_\bS(y) \re^{-h(y)} h(\rd y).
\]
On the intersection of the plane $\sum_{w} S_w = C$ and $\KM_{\Delta}$, we assert that $\lambda \bS_{\Delta}^{\star} + (1-\lambda) \bP_{\Delta}^{\star}$ majorizes any other points because it has the largest possible first entry.
The calculation shows that here
\[
\lambda = \frac{(1-C)\cdot (|\VM|\wedge |\Voca|-1+\Delta)}{\Delta \cdot |\VM|\wedge |\Voca|} \in [0, 1].
\]
By the definition of Schur-convexity, we have
\[
\sup_{\bS \in \KM_{\Delta}: \sum_{w} S_w = C}  \int F_\bS(y) \re^{-h(y)} h(\rd y) 
=   \int F_{\lambda \bS_{\Delta}^{\star} + (1-\lambda) \bP_{\Delta}^{\star}}(y) \re^{-h(y)} h(\rd y)
=: G(C).
\]
We denote the largest and the second largest entries in $\lambda \bS_{\Delta}^{\star} + (1-\lambda) \bP_{\Delta}^{\star}$ by $S_1$ and $S_2$. 
One can see that
\begin{equation}
S_1 = \lambda \frac{1-\Delta}{1+\frac{\Delta}{|\VM|\wedge |\Voca|-1}} + (1-\lambda) (1-\Delta)
\quad \text{and} \quad
S_2 = (1-\lambda) \Delta.
\end{equation}
It then follows that
\begin{align*}
 F_{\lambda \bS_{\Delta}^{\star} + (1-\lambda) \bP_{\Delta}^{\star}}(y) 
 &=\frac{S_1 y^{1/S_1} + S_2 y^{1/S_2}}{S_1+S_2}\\
&\overset{(a)}{\le}\frac{\lambda \frac{1-\Delta}{1+\frac{\Delta}{|\VM|\wedge |\Voca|-1}} y^{\frac{1+\frac{\Delta}{|\VM|\wedge |\Voca|-1}}{1-\Delta}} + (1-\lambda) (1-\Delta)y^{\frac{1}{1-\Delta}}  + S_2 y^{1/S_2}}{\lambda \frac{1-\Delta}{1+\frac{\Delta}{|\VM|\wedge |\Voca|-1}} + (1-\lambda) (1-\Delta)+S_2}\\
&\overset{(b)}{\le}\frac{\lambda \frac{1-\Delta}{1+\frac{\Delta}{|\VM|\wedge |\Voca|-1}} y^{\frac{1+\frac{\Delta}{|\VM|\wedge |\Voca|-1}}{1-\Delta}} + (1-\lambda) (1-\Delta)y^{\frac{1}{1-\Delta}}  + (1-\lambda) \Delta y^{\frac{1}{\Delta}}}{\lambda \frac{1-\Delta}{1+\frac{\Delta}{|\VM|\wedge |\Voca|-1}} + (1-\lambda) (1-\Delta)+(1-\lambda) \Delta}\\
&\overset{(c)}{=}
\frac{\lambda \frac{1-\Delta}{1+\frac{\Delta}{|\VM|\wedge |\Voca|-1}} F_{\bS_{\Delta}^{\star}}(y)+ (1-\lambda) F_{\bP_{\Delta}^{\star}}(y)}{\lambda \frac{1-\Delta}{1+\frac{\Delta}{|\VM|\wedge |\Voca|-1}} + (1-\lambda) },
\end{align*}
where $(a)$ uses the fact that the map $S \mapsto S y^{1/S}$ is convex in $S$, $(b)$ uses the fact that $y^{1/S_2} \le y^{\frac{1}{\Delta}}$ due to $y \in [0, 1]$ and $\lambda \in [0, 1]$ and $(c)$ follows from arrangement.

Therefore, for any $C \in [ \frac{1-\Delta}{1+\frac{\Delta}{|\VM|\wedge |\Voca|-1}}, 1 ]$, it follows that
\begin{align*}
G(C)
&= \int F_{\lambda \bS_{\Delta}^{\star} + (1-\lambda) \bP_{\Delta}^{\star}}(y)  \re^{-h(y)} h(\rd y)\\
&\le \int  \frac{\lambda \frac{1-\Delta}{1+\frac{\Delta}{|\VM|\wedge |\Voca|-1}} F_{\bS_{\Delta}^{\star}}(y)+ (1-\lambda) F_{\bP_{\Delta}^{\star}}(y)}{\lambda \frac{1-\Delta}{1+\frac{\Delta}{|\VM|\wedge |\Voca|-1}} + (1-\lambda) }  \re^{-h(y)} h(\rd y)\\
&\le \max\left\{ \int F_{\bS_{\Delta}^{\star}} \re^{-h(y)} h(\rd y),  \int F_{\bP_{\Delta}^{\star}} \re^{-h(y)} h(\rd y) \right\},
\end{align*}
where the last inequality uses the fact that the maximum value of a linear function on a line segment is attained at the endpoints.

\end{proof}

\subsection{Optimal Score in the Intermediate Regime}
\label{sec:intermediate-regime-analysis}

In this subsection, we detail the discussion in Remark \ref{rem:beyong-saddle-point}.
Specifically, if we do not require the optimal score function to be part of a saddle point solution, that is, the optimal score function solves the following minimization problem
\begin{equation}
\label{eq:J-loss}
\min_{h} J(h)
~~\text{where}~~
J(h) := \max_{\bS \in \DM_{\Delta}} L(h, \bS)
~~\text{and}~~
L(h, \bS) := \EB_0[h(Y)] + \log \EB_{F_\bS}[\re^{-h(Y)}],
\end{equation}
then the optimal score function always exists in the intermediate regime.
However, it doesn't have a closed form. In the following, we formally state this result.

\begin{lem}
\label{lem:optimal-increasing}
Any score function that minimizes $J$ in \eqref{eq:J-loss} is non-decreasing.
\end{lem}
\begin{proof}[Proof of Lemma \ref{lem:optimal-increasing}]
For any score function $h$, we can construct a non-decreasing transformation $h^\uparrow$ such that $L(h^\uparrow, \bS) \le L(h, \bS)$ for all $\bS \in \DM_\Delta$.
Specifically, let $G_h(z) = \PB_0(h(Y) \le z)$ with $Y \sim \mathrm{Unif}(0,1)$ under $H_0$, and define $h^\uparrow$ as the generalized inverse of $G_h$:
\[
h^\uparrow(y) = G_h^{-1}(y) := \inf\{z \in \RB : G_h(z) \ge y\}.
\]
By construction, $h^{\uparrow}$ is non-decreasing. Moreover, for any $z \in \RB$,
\[
\PB_0(h^\uparrow(Y) \le z) = \PB_0(G_h^{-1}(Y)\le z) 
= \PB_0(Y \le G_h(z)) 
= G_h(z) 
= \PB_0(h(Y)\le z),
\]
so $h^\uparrow(Y)$ and $h(Y)$ have the same distribution under $H_0$.

We now examine the two terms in $L(h,\bS)$.  
For the first term, $\EB_0[h(Y)] = \EB_0[h^\uparrow(Y)]$ since the distributions coincide.  
Let $f_\bS$ denote the alternative PDF, which is non-decreasing in $y$.
Consider the second term $\int_0^1 \re^{-h(y)} f_\bS(y) \rd y$.
The Hardy--Littlewood inequality \cite[Chapter~2]{bennett1988interpolation} implies that the integral of the product of two functions is minimized when the functions are ordered in opposite monotonicity.
Since $f_\bS(y)$ is non-decreasing, this integral is minimized when $\re^{-h(y)}$ is non-increasing, which is equivalent to $h(y)$ being non-decreasing.
Hence, $L(h^\uparrow, \bS) \le L(h, \bS)$ for all $\bS \in \DM_\Delta$.
\end{proof}

By Lemma \ref{lem:optimal-increasing} and Lemmas \ref{lem:maximum}, \ref{lem:extreme-points}, and \ref{lem:maximum-CDF-polyhedron}, for any non-decreasing function $h$, 
\begin{equation}
\label{eq:J-simplified-loss}
J(h) = \max \{ L(h, \bP_{\Delta}^{\star}), \, L(h, \bS_{\Delta}^{\star}) \},
\end{equation}
where $\bP_{\Delta}^{\star}$ and $\bS_{\Delta}^{\star}$ are the two distribution vectors defined in Lemma \ref{lem:extreme-points}.
Since $L(h, \bS)$ is strictly convex in $h$ for any fixed $\bS$, the above objective, being the pointwise maximum of two strictly convex functions, is also strictly convex. 
This ensures the existence and uniqueness of the minimizer of $J$, which is characterized in the following lemma.

\begin{lem}[Optimal score function]
\label{lem:optimal-score-under-intermediate-regime}
When $\Delta \in (\Delta_1^\star, \Delta_2^\star)$, that is, we have $L(h_{\bP_{\Delta}^\star}, \bP_{\Delta}^\star) < L(h_{\bP_{\Delta}^\star},\bS_{\Delta}^\star)$ and $L(h_{\bS_{\Delta}^\star}, \bS_{\Delta}^\star) < L(h_{\bS_{\Delta}^\star},\bP_{\Delta}^\star)$, the optimal score that minimizes $J$ defined in \eqref{eq:J-loss} is
\[
h_{\lambda^\star}^{\mathrm{gum}}(y) = \log(\lambda^\star \cdot y^{\frac{(|\VM|\wedge|\Voca|)\,\Delta}{(|\VM|\wedge|\Voca|-1)(1-\Delta)}} + (1-\lambda^\star) \cdot (y^{\frac{\Delta}{1-\Delta}} + y^{\frac{1-\Delta}{\Delta}}))
\]
where $\lambda^\star$ is the solution to this equation $L(h_{\lambda}^{\mathrm{gum}}, \bP_{\Delta}^{\star}) = L(h_{\lambda}^{\mathrm{gum}}, \bS_{\Delta}^{\star})$.
\end{lem}

As shown in Lemma \ref{lem:optimal-score-under-intermediate-regime}, the optimal score $h_{\lambda^\star}^{\mathrm{gum}}$ takes the form of a log-likelihood ratio score associated with a mixture alternative distribution, where the mixing parameter $\lambda^\star$ has no closed-form expression. 
For this reason, we do not pursue it further in the main text.

\begin{proof}[Proof of Lemma \ref{lem:optimal-score-under-intermediate-regime}]
For simplicity, let $h^\star$ denote the optimal score function. 
We first claim that the unique minimizer $h^{\star}$ must satisfy the equalization condition:
\begin{equation}
\label{eq:equalization}
L(h^{\star}, \bP_{\Delta}^{\star}) = L(h^{\star}, \bS_{\Delta}^{\star}). 
\end{equation}
Suppose, for contradiction, that $L(h^{\star}, \bP_{\Delta}^{\star}) > L(h^{\star}, \bS_{\Delta}^{\star})$. 
Then we have $J(h^{\star}) = L(h^{\star}, \bP_{\Delta}^{\star})$ from \eqref{eq:J-simplified-loss}, so $h^{\star}$ is also a local minimizer of $L(h, \bP_{\Delta}^{\star})$. 
By strict convexity, this forces $h^{\star}$ to equal the unique global minimizer $h_{\bP_{\Delta}^{\star}}$. 
But substituting back yields $L(h_{\bP_{\Delta}^\star}, \bP_{\Delta}^\star) > L(h_{\bP_{\Delta}^\star}, \bS_{\Delta}^\star)$, which contradicts the condition that $L(h_{\bP_{\Delta}^\star}, \bP_{\Delta}^\star) < L(h_{\bP_{\Delta}^\star}, \bS_{\Delta}^\star).$
A similar argument rules out the case $L(h^{\star}, \bS_{\Delta}^{\star}) > L(h^{\star}, \bP_{\Delta}^{\star})$. 
Thus, the equalization condition \eqref{eq:equalization} must hold. This means that at the optimal point $h^{\star}$, both component functions are active and attain the same value.

From the first-order stationary condition, the zero function belongs to the subdifferential set at $h^{\star}$, that is, $0 \in \partial J(h^{\star})$. By standard convex analysis, the subdifferential of the maximum of functions is the convex hull of the gradients of the active functions, namely $\partial J(h^{\star}) = \mathrm{conv} \{ \nabla_h L(h^{\star}, \bP_{\Delta}^{\star}), \nabla_h L(h^{\star}, \bS_{\Delta}^{\star}) \}.$
This implies the existence of a mixing parameter $\lambda^{\star} \in (0,1)$ such that
\begin{equation}
\label{eq:subgradient-optimality}
\lambda^{\star} \nabla_h L(h^{\star}, \bP_{\Delta}^{\star}) + (1-\lambda^{\star}) \nabla_h L(h^{\star}, \bS_{\Delta}^{\star}) = 0. 
\end{equation}
We remark that we must have $\lambda^{\star} \in (0,1)$; otherwise $h^{\star}$ would equal $h_{\bP_{\Delta}^\star}$ or $h_{\bS_{\Delta}^\star}$, which contradicts $\Delta \in (\Delta_1^\star, \Delta_2^\star)$.

We now derive the explicit form of $h^{\star}$.
Let $f_{\bP_{\Delta}^{\star}}$ and $f_{\bS_{\Delta}^{\star}}$ denote the alternative PDFs associated with $F_{\bP_{\Delta}^{\star}}$ and $F_{\bS_{\Delta}^{\star}}$, respectively, and let $f_0$ be the null PDF. The functional gradient of $L(h,\bS)$ with respect to $h$ at a point $y$ is
$$
\nabla_h L(h, \bS)(y) = f_0(y) - \frac{\re^{-h(y)} f_{\bS_{\Delta}^{\star}}(y)}{\EB_{F_{\bS}}[\re^{-h}]}.
$$
By the equalization condition \eqref{eq:equalization}, the denominators are equal: $\EB_{F_{\bP_{\Delta}^{\star}}}[\re^{-h^{\star}}] = \EB_{F_{\bS_{\Delta}^\star}}[\re^{-h^{\star}}]$. Let this common value be $C^\star$. Substituting the gradients into the optimality condition \eqref{eq:subgradient-optimality} gives
$$
f_0(y) - \frac{\re^{-h^{\star}(y)}}{C^*} \left( \lambda^{\star} f_{\bP_{\Delta}^{\star}}(y) + (1-\lambda^{\star}) f_{\bS_{\Delta}^{\star}}(y) \right) = 0.
$$
Solving for $h^{\star}(y)$ and noting that the additive constant $-\log C^\star$ does not affect detection performance, we obtain the explicit form of the optimal score function:
\begin{equation*}
h^{\star}(y) = \log \left( \frac{\lambda^{\star} f_{\bP_{\Delta}^{\star}}(y) + (1-\lambda^{\star}) f_{\bS_{\Delta}^{\star}}(y)}{f_0(y)} \right).
\end{equation*}
Thus, the optimal score function is precisely the log-likelihood ratio between the null distribution $f_0$ and a mixture of the two extremal alternative distributions.
  
\end{proof}

\section{Proof for Inverse Transform Watermarks in Section~\ref{sec:inverse}}
\label{append:inverse}

We begin by introducing the notation and terminology used throughout this section, as the analysis of the inverse transform watermark involves several technical components.

\paragraph{General notation.}
Throughout the proof, we use $(\cdot)_+$ to denote the positive part function, that is, $(x)_+ = \max\{x, 0\}$. For a function $f: A \to \mathbb{R}$ and a constant $M > 0$, we define the \emph{clipped extension} $[f]_{[-M, M]}: \mathbb{R} \to \mathbb{R}$ as a continuous function satisfying:
\begin{equation}
\label{eq:clip-def}
[f]_{[-M, M]}(x) = 
\begin{cases}
f(x), & \text{if } x \in A \text{ and } f(x) \in [-M, M], \\
M, & \text{if } x \in A \text{ and } f(x) > M, \\
-M, & \text{if } x \in A \text{ and } f(x) < -M, \\
\text{a continuous value in } [-M, M], & \text{if } x \notin A.
\end{cases}
\end{equation}
We denote the permutation group over $\Voca$ by $\Perm(\Voca)$, and use $\pi \in \Perm(\Voca)$ to represent a permutation of the vocabulary. The permutation $\pi$ acts on token indices, so that $\pi(w)$ denotes the token to which $w$ is mapped. For brevity, we denote the set $\{1, 2, \ldots, m\}$ by $[m]$.

\paragraph{Belief classes.}

We formally reformulate the conditions from Assumption~\ref{asmp:heavy-tokensa} and collect all NTP distributions within a minimal unit $\mathcal{V}$ of type $\tau$ that satisfy Assumption~\ref{asmp:heavy-tokensa} into the class $\QM_{\tau,\bm{\Delta}}$.
As defined, $\QM_{\tau,\bm{\Delta}}$ depends only on the type $\tau$ and the regularity levels $\bm{\Delta} = (\Delta_t)_{t \in \mathcal{V}}$, as this information is sufficient to determine all valid NTP distributions in the asymptotic regime we consider.

\begin{defn}[Fixed-parameter belief class]
\label{def:q_class}
For a minimal unit $\mathcal{V} = \IM^\zeta$ of type $\tau$ and a sequence of regularity levels $\bm{\Delta} = (\Delta_t)_{t \in \mathcal{V}}$ with each $\Delta_t \in [\Delta, 1-\delta]$ as in Assumption~\ref{asmp:heavy-tokensa}, we define the class $\QM_{\tau, \bm{\Delta}}$ as the set of all joint NTP distributions $\bP_{\mathcal{V}}$ over the tokens in $\mathcal{V}$ that satisfy Assumption~\ref{asmp:heavy-tokensa}:
\begin{equation}
\label{eq:PM-inv-eta-new}
\QM_{\tau,\bm{\Delta}} 
= \left\{ \bP_{\VM}: \ \forall t \in \IM^\zeta, \ P_{t,w_t} = P_{t,(1)} = 1 - \Delta_t \ \text{and} \ \log|\Voca| \cdot P_{t,(2)} \leq \eps_{|\Voca|} \right\},
\end{equation}
where $\bP_{\VM} := (\bP_t)_{t \in \VM}$ is the collection of marginal distributions of tokens in $\VM$, and $P_{t,(1)}, P_{t,(2)}$ denote the largest and second-largest probabilities in the NTP distribution $\bP_t$.
\end{defn}

\subsection{Proof of Lemma \ref{lem:asymptotic-joint-distributions-inverse}}

\begin{proof}[Proof of Lemma \ref{lem:asymptotic-joint-distributions-inverse}]

To establish the asymptotic distribution of the pseudorandom numbers and tokens, we first characterize their exact joint distribution in Lemma~\ref{lem:inverse-perfect-general-technicaljoint}. Since our analysis focuses on a fixed minimal unit $\IM_k^\zeta$, we omit the subscript $k$ for simplicity and denote it by $\VM = \IM_k^\zeta$, which contains $m$ sub-blocks. We adopt this notational convention throughout the proof of Theorem~\ref{thm:inverse-asymptotic-distribution} as well.

\begin{lem}[Exact joint distribution]\label{lem:inverse-perfect-general-technicaljoint}
Fix a minimal unit (or block) $\IM^{\zeta}$ consisting of $m$ sub-blocks, denoted by $\IM_{\ell}^Y$ for $\ell \in [m]$, such that $\bigcup_{\ell=1}^m \IM_{\ell}^Y = \IM^{\zeta}$.
Let Assumption \ref{asmp:blocks-independence} hold.
Assume the shared pseudorandom variables for this block are $(U, \pi)$, where $U \in [0,1]$ is uniform and $\pi \in \Perm(\Voca)$ is a permutation of the vocabulary.
Denote the token associated with each sub-block $\IM_{\ell}^Y$ by $w_\ell$ for $\ell \in [m]$.

Then the joint distribution of $(U, \pi(w_1), \ldots, \pi(w_m))$ conditioned on the fixed block $\IM^\zeta$ is given by
\begin{align}
& \PB_1\left(U \leq r,\; \pi(w_i') = w_i \text{ for } i = 1,\dots,m \;\middle|\; \IM^\zeta \right) \\
=\, & \frac{
    \frac{1}{|\Voca|!} \sum\limits_{\substack{\pi \in \Perm(\Voca) \\ \pi(w_\ell') = w_\ell,\; \ell \in [m]}} 
    \PB\left(U \in \bigcap_{\ell=1}^m \bigcap_{t \in \IM^Y_\ell} \left(a_{\pi, w_\ell'-1}^{(t)},\; a_{\pi, w_\ell'}^{(t)}\right) \cap [0, r] \right)
}{
    \frac{1}{|\Voca|!} \sum\limits_{\substack{w_1', \ldots, w_m' \\ \mathrm{distinct}}} \sum\limits_{\substack{\pi \in \Perm(\Voca) \\ \pi(w_\ell') = w_\ell,\; \ell \in [m]}} 
    \PB\left(U \in \bigcap_{\ell=1}^m \bigcap_{t \in \IM^Y_\ell} \left(a_{\pi, w_\ell'-1}^{(t)},\; a_{\pi, w_\ell'}^{(t)}\right) \right)
}
\end{align}
where the endpoint $a_{\pi, w_\ell}^{(t)}$ is defined by
\[
a_{\pi,w_l}^{(t)} = \sum_{j=1}^{w_l} P_{t,\pi(j)}, \quad \forall t \in \IM^Y_{\ell}, \quad \forall \ell \in [m].
\]
\end{lem}

The proof of Lemma~\ref{lem:inverse-perfect-general-technicaljoint} is provided in Section~\ref{sec:inverse-perfect-general-technicaljoint-proof}. Following the convention in \citet{li2024statistical}, we analyze the expectation of an arbitrary test function $J$ of $(U, \pi(w_1), \ldots, \pi(w_m))$, as it characterizes the joint distribution as well, is equivalent to studying the CDF, and facilitates the analysis of the asymptotic behavior.

\begin{cor}
\label{cor:inverse-perfect-general-technicaljoint}
Under the same notation and assumption as in Lemma~\ref{lem:inverse-perfect-general-technicaljoint}, for any measurable test function $J \colon [0,1]^{m+1} \to [0, \infty)$, we have:
\begin{align}
  &\EB_{1, \bP_{\IM^\zeta}}\left[J(U, \eta(\pi(w_1)),\ldots, \eta(\pi(w_m)))\right] \nonumber \\
  &=  \frac{
    \sum_{\substack{w'_1,\ldots,w'_m \\ \mathrm{distinct}}}
    \sum_{\substack{\pi \in \Perm(\Voca) \\ \pi(w_\ell') = w_\ell,\forall \ell \in [m]}}
    \left(
      \int_{\max\limits_{\ell, t} a_{\pi, w'_\ell - 1}^{(t)}}^{\min\limits_{\ell, t} a_{\pi, w'_\ell}^{(t)}}
      J(u, \eta(w'_1), \ldots, \eta(w'_m)) \, \rd u
    \right)
    \1_{\min\limits_{\ell, t} a_{\pi, w'_\ell}^{(t)} \ge \max\limits_{\ell, t} a_{\pi, w'_\ell - 1}^{(t)}}
  }{
    \sum_{\substack{w'_1,\ldots,w'_m \\ \mathrm{distinct}}}
    \sum_{\substack{\pi \in \Perm(\Voca) \\ \pi(w_\ell') = w_\ell, \forall \ell \in [m]}}
    \PB\left(
      U \in \bigcap_{\ell=1}^m \bigcap_{t \in \IM^Y_\ell}
      \left(a_{\pi, w'_\ell - 1}^{(t)}, a_{\pi, w'_\ell}^{(t)}\right)
    \right)
  }.
\end{align}
where the $\min\limits_{\ell, t}$ or $\max\limits_{\ell, t}$ are taken over all sub-blocks $\ell \in [m]$ and all token indices $t \in \IM^Y_\ell$. 
\end{cor}

With Lemma~\ref{lem:inverse-perfect-general-technicaljoint} in place, we then derive the asymptotic joint distribution of $(U, \pi(w_1), \ldots, \pi(w_m))$ when $|\Voca| \to \infty$.
To do so, we examine the limiting expectation $\EB\left[J(U, \eta(\pi(w_1)), \ldots, \eta(\pi(w_m)))\right]$ for any arbitrary test function $J$.

\begin{theorem}[Asymptotic distribution under $H_1$]
\label{thm:inverse-perfect-general-technical-main}
Let $\IM^\zeta$ be a minimal unit consisting of $m$ sub-blocks $\{\IM^Y_{\ell}\}_{\ell=1}^{m}$, and let $\{w_{\ell}\}_{\ell=1}^{m} \subseteq \Voca$ denote the distinct tokens representing these sub-blocks. 
Let Assumptions \ref{asmp:blocks-independence} and \ref{asmp:heavy-tokensa} hold.
Let $(\Delta_t)_{t \in \IM^\zeta}$ be the per-time regularity levels, where each $\Delta_t \in [\Delta, 1 - \delta]$, and define the sub-block regularity vector $ (\bDelta_1, \ldots, \bDelta_m)$ by $\bDelta_{\ell} := \max_{t \in \IM^Y_{\ell}} \Delta_t.$

Then for any measurable function $J\colon [0,1]^{m+1} \to [0, \infty)$, the expectation in Corollary~\ref{cor:inverse-perfect-general-technicaljoint} converges as $|\Voca| \to \infty$ to
\begin{align}
&\lim_{|\Voca| \to \infty} \EB_{1, \bP_{\IM^\zeta}}\bigl[J(U, \eta(\pi(w_1)), \dots, \eta(\pi(w_m)))\bigr] \nonumber \\
&\quad = \frac{1}{I_m(\bar{\bm{\Delta}})} 
\int_{[0,1]^m}  \int_{\max\limits_{\ell \in [m]} \bDelta_{\ell} x_\ell}^{\min\limits_{\ell \in [m]} (1 - \bDelta_{\ell} + \bDelta_{\ell} x_\ell)} J(u, x_1, \dots, x_m) 
\1_{\left\{\min\limits_{\ell \in [m]} (1 - \bDelta_{\ell} + \bDelta_{\ell} x_\ell) \ge \max\limits_{\ell \in [m]} \bDelta_{\ell} x_\ell\right\}} \, \rd u \rd x_1 \cdots \rd x_m,
\end{align}
where the normalization constant $I_m(\bar{\bm{\Delta}})$ is the volume of the integration region:
\[
I_m(\bar{\bm{\Delta}}) := \int_{[0,1]^m} \left( \min\limits_{\ell \in [m]} (1 - \bDelta_{\ell} + \bDelta_{\ell} x_\ell) - \max\limits_{\ell \in [m]} \bDelta_{\ell} x_\ell \right)_+ \,\rd x_1 \cdots \rd x_m.
\]
Moreover, the convergence holds uniformly over any 1-Lipschitz test functions $J$, any NTP distributions $\bP_{\IM^\zeta}$ within the class $\QM_{\tau, \bm{\Delta}}$, and any regularity vectors $\bar{{\Delta}}$.
\end{theorem}

The proof of Theorem \ref{thm:inverse-perfect-general-technical-main} can be found in Section \ref{sec:inverse-perfect-general-technical-main-proof}.
The arbitrariness of the test function $J$ in Theorem~\ref{thm:inverse-perfect-general-technical-main} directly implies the following weak convergence.

\begin{cor}[Asymptotic distribution under $H_1$]
\label{cor:inverse-perfect-general-technical-main}
Under the same notation and assumptions as in Theorem~\ref{thm:inverse-perfect-general-technical-main}, the joint vector
\[
(U,\eta(\pi(w_1)),\ldots,\eta(\pi(w_m)))
\]
converges in distribution to a random vector $(U, X_1, \ldots, X_m)$, where $X_1, \ldots, X_m$ are i.i.d.\ $\mathrm{Unif}(0,1)$ random variables, and $U$ is independently drawn from the interval
\[
\left[\max\limits_{\ell \in [m]} \{\bDelta_{\ell} X_\ell\},\ \min\limits_{\ell \in [m]} \{1 - \bDelta_{\ell} + \bDelta_{\ell} X_\ell\}\right]
\]
conditioned on the event that this interval is non-empty, that is, $\max\limits_{\ell \in [m]} \{\bDelta_{\ell} X_\ell\} \le \min\limits_{\ell \in [m]}\{1 - \bDelta_{\ell} + \bDelta_{\ell} X_\ell\}$.
\end{cor}
\begin{proof}[Proof of Lemma \ref{cor:inverse-perfect-general-technical-main}]
This follows directly from Theorem~\ref{thm:inverse-perfect-general-technical-main}, together with the Portmanteau theorem (Theorem 13.16 in~\cite{klenke2020probability}) and Lemma~\ref{lem:joint-law-UkX}, which together allow us to translate convergence in expectation into weak convergence.
\end{proof}

By an argument similar to that of Theorem~\ref{thm:inverse-perfect-general-technical-main}, we can show that
\begin{lem}[Asymptotic distribution under $H_0$]
\label{lem:asymp-H0}
Under the null hypothesis $H_0$, the joint distribution of $(U, \eta(\pi(w_1)), \ldots, \eta(\pi(w_m)))$ converges weakly to that of $(U, X_1, \ldots, X_m)$, where $U, X_1, \ldots, X_m$ are i.i.d. $\mathrm{Unif}(0,1)$.
\end{lem}
\end{proof}

\subsection{Proof of Theorem \ref{thm:inverse-asymptotic-distribution}}
\label{sec:optimal-detection-rule}

\begin{proof}[Proof of Theorem \ref{thm:inverse-asymptotic-distribution}]

This theorem establishes the asymptotic joint distribution of the pivotal statistics within a minimal unit. To achieve this, we will make use of the results in Lemma \ref{lem:asymptotic-joint-distributions-inverse}.

\paragraph{Null joint distribution of pivotal statistics.}

We first derive the joint distribution under $H_0$ for the pivotal statistics $(Y_t)_{t \in \IM^\zeta}$ within a minimal unit.  
By Lemma \ref{lem:asymptotic-joint-distributions-inverse}, as $|\Voca| \to \infty$, each $Y_\ell = |U - \eta(\pi(w_\ell))|$ converges weakly to $|U - X_\ell|$ under $H_0$, where $U, X_1, \ldots, X_m$ are i.i.d. random variables uniformly distributed on $[0,1]$.  
With a slight abuse of notation, we relabel $Y_\ell := |U - X_\ell|$ for $\ell \in [m]$ to simplify notation.

We begin by analyzing the conditional CDF of $Y_\ell$ given $U=u$.  
Fix $u \in [0,1]$ and take any $y \in [0,1]$. The conditional CDF of $Y_\ell$ is:
\begin{align}
  \PB(Y_\ell \leq y \mid U = u) &= \PB(|U - X_\ell| \leq y \mid U = u) \\
  &= \PB(u - y \leq X_\ell \leq u + y \mid U = u) \\
  &= [(u + y) \wedge 1] - [(u - y) \vee 0].
\end{align}
The corresponding conditional PDF is then
\begin{align}
  f_{Y_\ell \mid U}(y \mid u) = 
  \begin{cases}
    2, & \text{if } 0 < y < \min(u, 1 - u), \\
    1, & \text{if } \min(u, 1 - u) < y < \max(u, 1 - u), \\
    0, & \text{otherwise}.
  \end{cases}
\end{align}
Consequently, under $H_0$, the joint PDF of $\bm{Y}=(Y_1, \ldots, Y_m)$ given $U=u$ is:
\begin{align}
f_{\bm{Y}}(y_1,\ldots,y_m) 
  = \int_0^1 f_{Y_1,\ldots,Y_m \mid U}(y_1,\ldots,y_m \mid u) \, \rd u
  = \int_0^1 \prod_{\ell=1}^m f_{Y_\ell \mid U}(y_\ell \mid u) \, \rd u.
\end{align}
To simplify this expression, define two index sets depending on $u$:
\begin{align}
  I_1(u) &= \{\ell : 0 < y_\ell < \min(u, 1 - u)\}, \\
  I_2(u) &= \{\ell : y_\ell \geq \max(u, 1 - u)\}.
\end{align}
Then the joint density becomes:
\begin{align} \label{eq:joint-density-representationh01}
    f_0(y_1,\ldots,y_m):= f_{\bm{Y}}(y_1,\ldots,y_m) = \int_{0}^1 2^{|I_1(u)|} \1_{I_2(u) = \emptyset} \, \rd u.
\end{align}
We will later provide an alternative representation of this density that is more convenient for theoretical analysis but more complex in form. For numerical computations, however, the integral expression in \eqref{eq:joint-density-representationh01} is preferable.

\paragraph{Alternative joint distribution of pivotal statistics.}

We now derive the joint distribution under $H_1$ for the pivotal statistics $(Y_t)_{t \in \IM^\zeta}$ within a minimal unit. According to Lemma \ref{lem:asymptotic-joint-distributions-inverse}, under $H_1$, the tuple $(U, \eta(\pi(w_1)), \ldots, \eta(\pi(w_m)))$ converges in distribution to $(U, X_1, \ldots, X_m)$, where $X_1, \ldots, X_m$ are i.i.d.\ $\mathrm{Unif}(0,1)$ random variables, and $U$ is drawn independently and uniformly from the interval
\[
\left[\max\limits_{\ell \in [m]} \{\bDelta_{\ell} X_\ell\},\ \min\limits_{\ell \in [m]} \{1 - \bDelta_{\ell} + \bDelta_{\ell} X_\ell\} \right],
\]
conditioned on the interval being non-empty. For convenience, we denote $Y_\ell := |U - X_\ell|$ for $\ell \in [m]$.

To obtain the joint density of $(Y_1, \ldots, Y_m)$, we consider the transformation
\[
\Phi: (U, X_1, \ldots, X_m) \mapsto (U, Y_1, \ldots, Y_m) = (U, |U - X_1|, \ldots, |U - X_m|).
\]
Since $\Phi$ is continuous and the joint law of $(U, X_1, \ldots, X_m)$ is absolutely continuous, we may ignore boundary events (e.g., $Y_\ell = 0$ for some $\ell$) which have zero measure.

However, $\Phi$ is not injective due to the absolute values. To apply the change-of-variable formula, we partition the domain into disjoint regions where $\Phi$ becomes bijective. For each sign vector $\bm{\sigma} = (\sigma_1, \ldots, \sigma_m) \in \{-1,1\}^m$, define the region
\begin{align}
\mathcal{R}_{\bm{\sigma}} := \left\{ (u, x_1, \ldots, x_m) \in [0,1]^{m+1} : \sign(x_\ell - u) = \sigma_\ell \text{ for all } \ell \right\}.
\end{align}
Within $\mathcal{R}_{\bm{\sigma}}$, we have $x_\ell = u + \sigma_\ell y_\ell$ and $\Phi$ is bijective with Jacobian determinant of absolute value 1. Thus, the joint density of $(U, Y_1, \ldots, Y_m)$ on this region is directly given by the density of $(U, X_1, \ldots, X_m)$ evaluated at $(u, x_1, \ldots, x_m) = (u, u + \sigma_1 y_1, \ldots, u + \sigma_m y_m)$.

To integrate out the nuisance parameter $U$, we first characterize the feasible values of $u$ given $\bm{y} = (y_1, \ldots, y_m)$ and a fixed sign vector $\bm{\sigma} = (\sigma_1, \ldots, \sigma_m)$. The first requirement is that each reconstructed $x_\ell = u + \sigma_\ell y_\ell$ must lie within the unit interval $[0,1]$, which leads to the constraint:
\begin{align*}
L_{\bm{\sigma}}(\bm{y}) := \max_{\ell} (-\sigma_\ell y_\ell) \le u \le \min_{\ell} (1 - \sigma_\ell y_\ell) =: U_{\bm{\sigma}}(\bm{y}).
\end{align*}
Second, we enforce the conditional event to hold from Corollary~\ref{thm:inverse-perfect-general-technical-main}, which requires
\[
\bDelta_{\ell} x_\ell \le u \le 1 - \bDelta_{\ell} + \bDelta_{\ell} x_\ell \quad \text{for all } \ell.
\]
Substituting $x_\ell = u + \sigma_\ell y_\ell$ and solving for $u$ leads to
\[
\frac{\bDelta_{\ell} \sigma_\ell y_\ell}{1 - \bDelta_{\ell}} \le u \le 1 + \frac{\bDelta_{\ell} \sigma_\ell y_\ell}{1 - \bDelta_{\ell}}.
\]
Taking the maximum lower bound and minimum upper bound across $\ell$, we define
\begin{align*}
Y_{\bm{\sigma}}^+(\bm{y}) &:= \max\left( \left\{ \frac{\bDelta_{\ell} y_\ell}{1 - \bDelta_{\ell}} : \sigma_\ell = +1 \right\} \cup \{0\} \right), \\
Y_{\bm{\sigma}}^-(\bm{y}) &:= \max\left( \left\{ \frac{\bDelta_{\ell} y_\ell}{1 - \bDelta_{\ell}} : \sigma_\ell = -1 \right\} \cup \{0\} \right),
\end{align*}
which yield the additional constraint:
\[
Y_{\bm{\sigma}}^+(\bm{y}) \le u \le 1 - Y_{\bm{\sigma}}^-(\bm{y}).
\]

Combining both sets of constraints, the overall feasible range for $u$ is the interval
\[
\left[A_{\bm{\sigma}}^{\bar{\bm{\Delta}}}(\bm{y}), B_{\bm{\sigma}}^{\bar{\bm{\Delta}}}(\bm{y})\right], \quad \text{where} \quad 
\begin{cases}
A_{\bm{\sigma}}^{\bar{\bm{\Delta}}}(\bm{y}) := \max\{L_{\bm{\sigma}}(\bm{y}), Y_{\bm{\sigma}}^+(\bm{y})\}, \\
B_{\bm{\sigma}}^{\bar{\bm{\Delta}}}(\bm{y}) := \min\{U_{\bm{\sigma}}(\bm{y}), 1 - Y_{\bm{\sigma}}^-(\bm{y})\}.
\end{cases}
\]
Since the density of $(U, X_1, \ldots, X_m)$ is constant and equals $1/I_m(\bar{\bm{\Delta}})$ over its support, the contribution from each region is proportional to the length of this feasible interval:
\[
\ell_{\bm{\sigma}}^{\bar{\bm{\Delta}}}(\bm{y}) := \left( B_{\bm{\sigma}}^{\bar{\bm{\Delta}}}(\bm{y}) - A_{\bm{\sigma}}^{\bar{\bm{\Delta}}}(\bm{y}) \right)_+.
\]
Summing over all $2^m$ sign vectors, the joint density of $\bm{Y} = (Y_1, \ldots, Y_m)$ is
\[
f_{\bar{\bm{\Delta}}}(\bm{y}) :=
f_{\bm{Y}}^{\bar{\bm{\Delta}}}(\bm{y}) = \frac{1}{I_m(\bar{\bm{\Delta}})} \sum_{\bm{\sigma} \in \{-1,1\}^m} \ell_{\bm{\sigma}}^{\bar{\bm{\Delta}}}(\bm{y}).
\]
\begin{rem}
  As a sanity check, when $\bar{\bm{\Delta}} = \0$, we recover the null case. In that case, $Y_{\bm{\sigma}}^+ = Y_{\bm{\sigma}}^- = 0$, and the constraint reduces to $L_{\bm{\sigma}}(\bm{y}) \le u \le U_{\bm{\sigma}}(\bm{y})$, matching the joint density under $H_0$.  
\end{rem}
\end{proof}

\subsection{Proof of Corollary \ref{cor:inverse-m=1}}

\begin{proof}[Proof of Corollary~\ref{cor:inverse-m=1}]

This corollary follows by simplifying the expressions in Theorem~\ref{thm:inverse-asymptotic-distribution}. 

We begin by analyzing the alternative distribution.
When $m = 1$, the general density formula simplifies to:
\[
f_{Y_1}^{\Delta_1}(y_1) = \frac{1}{1 - \Delta_1} \sum_{\sigma \in \{-1,1\}} \left(B_\sigma^{\Delta_1}(y_1) - A_\sigma^{\Delta_1}(y_1)\right) \vee 0.
\]

\paragraph{Case $\sigma = +1$.}
Using the definitions from the theorem, we compute:
\begin{align*}
A_{+1}^{\Delta_1}(y_1) &= \max\left\{ -y_1, \frac{\Delta_1 y_1}{1 - \Delta_1} \right\} = \frac{\Delta_1 y_1}{1 - \Delta_1}, \\
B_{+1}^{\Delta_1}(y_1) &= \min\left\{ 1 - y_1, 1 \right\} = 1 - y_1.
\end{align*}
The corresponding contribution is:
\[
\left(1 - y_1 - \frac{\Delta_1 y_1}{1 - \Delta_1}\right) \vee 0 = \left(1 - \frac{y_1}{1 - \Delta_1}\right) \vee 0.
\]

\paragraph{Case $\sigma = -1$.}
We have:
\begin{align*}
A_{-1}^{\Delta_1}(y_1) &= \max\left\{ y_1, 0 \right\} = y_1, \\
B_{-1}^{\Delta_1}(y_1) &= \min\left\{ 1 + y_1, 1 - \frac{\Delta_1 y_1}{1 - \Delta_1} \right\} = 1 - \frac{\Delta_1 y_1}{1 - \Delta_1}.
\end{align*}
The resulting contribution is:
\[
\left(1 - \frac{\Delta_1 y_1}{1 - \Delta_1} - y_1\right) \vee 0 = \left(1 - \frac{y_1}{1 - \Delta_1}\right) \vee 0.
\]

Since both sign cases yield the same value, we obtain the final density by summing and applying the normalization:
\[
f_{Y_1}^{\Delta_1}(y_1) = \frac{2}{1 - \Delta_1} \left(1 - \frac{y_1}{1 - \Delta_1}\right),
\]
which expands to the triangular form in~\eqref{eq:joint-density-representationh1goodaand}.

Next, we consider the null distribution. When $m = 1$, the formula from Theorem~\ref{thm:inverse-asymptotic-distribution} becomes:
\[
f_{Y_1}(y_1) = \int_{u : y_1 < \max(u, 1 - u)} 2^{\1(y_1 < \min(u, 1 - u))} \, \rd u.
\]

\paragraph{Case $0 < y_1 \le 1/2$.}
In this case, $y_1 < \min(u, 1 - u)$ if and only if $u \in (y_1, 1 - y_1)$, and $y_1 < \max(u, 1 - u)$ for all $u \in (0,1)$. Therefore, the density is:
\[
f_{Y_1}(y_1) = \int_{y_1}^{1 - y_1} 2 \, \rd u + \int_{0}^{y_1} \rd u + \int_{1 - y_1}^1 \rd u = 2(1 - 2y_1) + y_1 + y_1 = 2(1 - y_1).
\]

\paragraph{Case $1/2 < y_1 < 1$.}
Here, $y_1 \ge \min(u, 1 - u)$, so the integrand is always 1. The condition $y_1 < \max(u, 1 - u)$ is satisfied when $u \in (0, 1 - y_1) \cup (y_1, 1)$, yielding:
\[
f_{Y_1}(y_1) = \int_{0}^{1 - y_1} \rd u + \int_{y_1}^1 \rd u = (1 - y_1) + (1 - y_1) = 2(1 - y_1).
\]

In both cases, we conclude that $f_{Y_1}(y_1) = 2(1 - y_1)$ for all $y_1 \in (0,1)$, completing the proof.
 
\end{proof}

\subsection{Proof of Lemma~\ref{lem:inverse-perfect-general-technicaljoint}}
\label{sec:inverse-perfect-general-technicaljoint-proof}

\begin{proof}[Proof of Lemma~\ref{lem:inverse-perfect-general-technicaljoint}]
The randomness in this setting arises from the pseudorandom variables $U$ and $\pi$. Given the fixed minimal unit $\IM^\zeta$, we aim to compute
\begin{equation}
  \PB_1\bigl(U \leq r,\, \pi(w_\ell) = w_\ell' \text{ for } \ell \in [m]~|~\IM^\zeta\bigr).
\end{equation}

For each permutation $\pi$ of the vocabulary, we can evaluate the probability that $U \in [0, r]$ under the constraint that $\pi(w_\ell) = w_\ell'$ for all $\ell \in [m]$. Recall the definition of the inverse transform decoder: for any token $w$,
\[
\mathcal{S}^{\text{inv}}(\bm{P}, \zeta) = w \quad \text{if and only if} \quad 
\sum_{w' \in \Voca} P_{w'} \cdot \mathbf{1}_{\{\pi(w') < \pi(w)\}} 
\le U \le 
\sum_{w' \in \Voca} P_{w'} \cdot \mathbf{1}_{\{\pi(w') \le \pi(w)\}}.
\]

In our setting, knowing that $\pi(w_\ell) = w_\ell'$ for all $\ell \in [m]$ and using the definition that $a_{\pi,\, w_\ell}^{(t)} = \sum_{j = 1}^{w_\ell} P_{t, \pi(j)}$, the above condition becomes, for each $t \in \IM^Y_\ell$ (where $w_\ell$ is the token associated with sub-block $\IM^Y_\ell$),
\begin{align}
a_{\pi^{-1},\, w_\ell'-1}^{(t)} 
&= \sum_{w' \in \Voca} P_{t,w'} \cdot \mathbf{1}_{\{ \pi(w') < w_\ell' \}} 
\le U \le 
\sum_{w' \in \Voca} P_{t,w'} \cdot \mathbf{1}_{\{ \pi(w') \le w_\ell' \}} 
= a_{\pi^{-1},\, w_\ell'}^{(t)}.
\end{align}

The corresponding feasible region for $U$ is thus the intersection
\begin{align}
  \bigcap_{\ell=1}^m \bigcap_{t \in \IM^Y_\ell} \left(a_{\pi^{-1}, w_\ell' - 1}^{(t)},\, a_{\pi^{-1}, w_\ell'}^{(t)}\right).
\end{align}
Summing over all permutations $\pi$ and all tuples of mutually distinct tokens $w_1', \ldots, w_m'$ (that is, $\mathrm{distinct}$), we obtain:
\begin{align}
  &\PB_1\bigl(U \leq r,\, \pi(w_\ell) = w_\ell' \text{ for } \ell \in [m]~|~\IM^\zeta\bigr) \\
  =\; & \frac{1}{|\Voca|!} \sum_{\substack{\pi \in \Perm(\Voca) \\ \pi(w_\ell') = w_\ell,\; \ell \in [m]}} 
  \PB\left(U \in \bigcap_{\ell=1}^m \bigcap_{t \in \IM^Y_\ell} \left(a_{\pi, w_\ell' - 1}^{(t)},\, a_{\pi, w_\ell'}^{(t)}\right) \cap [0, r] \right).
\end{align}
Note that since we sum over all permutations $\pi$, the roles of $\pi$ and $\pi^{-1}$ are interchangeable in the expression above. Hence, we replace $\pi$ with $\pi^{-1}$ in the last equation for notational simplicity.

To obtain the normalization constant (that is, the denominator of the conditional probability), we set $r = 1$ and sum over all distinct $w_1',\ldots,w_m'$:
\begin{align}
  &\sum_{\substack{w_1',\ldots,w_m' \\ \mathrm{distinct}}} 
  \PB\bigl(U \leq 1,\, \pi(w_\ell) = w_\ell' \text{ for } \ell \in [m]~|~\IM^\zeta\bigr) \\
  =\; & \frac{1}{|\Voca|!} 
  \sum_{\substack{w_1',\ldots,w_m' \\ \mathrm{distinct}}} 
  \sum_{\substack{\pi \in \Perm(\Voca) \\ \pi(w_\ell') = w_\ell,\; \ell \in [m]}} 
  \PB\left(U \in \bigcap_{\ell=1}^m \bigcap_{t \in \IM^Y_\ell} \left(a_{\pi, w_\ell' - 1}^{(t)},\, a_{\pi, w_\ell'}^{(t)}\right) \cap [0, 1] \right).
\end{align}

Thus, the conditional probability can be expressed as
\begin{align}
  &\PB\bigl(U \leq r,\, \pi(w_\ell) = w_\ell' \text{ for } \ell \in [m] \,\big|\, \IM^\zeta\bigr) \\
  =\; & \frac{
    \frac{1}{|\Voca|!} 
    \sum_{\substack{\pi \in \Perm(\Voca) \\ \pi(w_\ell') = w_\ell,\; \ell \in [m]}} 
    \PB\left(U \in \bigcap_{\ell=1}^m \bigcap_{t \in \IM^Y_\ell} \left(a_{\pi, w_\ell' - 1}^{(t)},\, a_{\pi, w_\ell'}^{(t)}\right) \cap [0, r] \right)
  }{
    \frac{1}{|\Voca|!} 
    \sum_{\substack{w_1',\ldots,w_m' \\ \mathrm{distinct}}} 
    \sum_{\substack{\pi \in \Perm(\Voca) \\ \pi(w_\ell') = w_\ell,\; \ell \in [m]}} 
    \PB\left(U \in \bigcap_{\ell=1}^m \bigcap_{t \in \IM^Y_\ell} \left(a_{\pi, w_\ell' - 1}^{(t)},\, a_{\pi, w_\ell'}^{(t)}\right) \right)
  }.
\end{align}
\end{proof}

\subsection{Proof of Theorem~\ref{thm:inverse-perfect-general-technical-main}}
\label{sec:inverse-perfect-general-technical-main-proof}

\begin{proof}[Proof of Theorem~\ref{thm:inverse-perfect-general-technical-main}]

In this proof, we aim to show that the absolute error
\begin{align}
\label{eq:expectation-differencewith-zero}
\left|\EB_{1, \bP_{\IM^\zeta}}\left[J(U,\eta(\pi(w_1)),\ldots,\eta(\pi(w_m)))\right] - \frac{1}{I_m(\bar{\bm{\Delta}})}\int_{[0,1]^m} \left(\int_{\IM(\bm{x},\bar{\bm{\Delta}})} J(u, \bm{x})\,\rd u\right)\1_{\IM(\bm{x},\bar{\bm{\Delta}}) \ne \emptyset} \rd \bm{x} \right|
\end{align}
converges to zero as the vocabulary size $|\Voca|$ tends to infinity, provided that the underlying NTP distributions $\bP_{\IM^\zeta}$ satisfy Assumption~\ref{asmp:heavy-tokensa}.
In the expression~\eqref{eq:expectation-differencewith-zero}, we simplify the original target integral by letting $\bm{x} = (x_1, \ldots, x_m)$ and $\rd \bm{x} = \rd x_1 \cdots \rd x_m$, and by rewriting the normalization constant via
\begin{equation}
\label{eq:Im-delta}
\IM(\bm{x},\bar{\bm{\Delta}}) = \left[\max\limits_{\ell \in [m]}\{\bDelta_{\ell} x_{\ell}\},\; \min\limits_{\ell \in [m]}\{1 - \bDelta_{\ell} + \bDelta_{\ell} x_{\ell}\}\right],
\end{equation}
where we define $\bDelta_{\ell} := \max_{t \in \IM^Y_{\ell}} \Delta_t$.
As specified in Definition~\ref{def:q_class}, the NTP distributions $\bP_{\IM^\zeta}$ are assumed to belong to the class $\QM_{\type(\bP_{\IM^\zeta}), \bm{\Delta}}$ from Assumption \ref{asmp:heavy-tokensa}.

Let $J\colon[0,1]^{m+1} \to [0,\infty)$ be a 1-Lipschitz function. 
Without loss of generality, we assume $J(0,0,\ldots,0) = 0$. This is justified because replacing $J$ with $J - C$ for any constant $C$ does not affect the absolute error term in~\eqref{eq:expectation-differencewith-zero} by using the fact that $I_m(\bar{\bm{\Delta}}) = \int_{[0, 1]^m} |\IM(\bm{x}, \bar{\bm{\Delta}})| \cdot \1_{\IM(\bm{x}, \bar{\bm{\Delta}}) \ne \emptyset} \, \rd \bm{x}.$

We are now ready to analyze the asymptotic behavior of $\EB_{1, \bP_{\IM^\zeta}}\left[J(U,\eta(\pi(w_1)),\ldots,\eta(\pi(w_m)))\right].$
An exact formulation is provided by Corollary~\ref{cor:inverse-perfect-general-technicaljoint}:
\begin{align}
  &\EB_{1, \bP_{\IM^\zeta}}\left[J(U, \eta(\pi(w_1)),\ldots, \eta(\pi(w_m)))\right] \nonumber \\
  &=  \frac{
    \sum_{\substack{w'_1,\ldots,w'_m \\ \mathrm{distinct}}}
    \sum_{\substack{\pi \in \Perm(\Voca) \\ \pi(w_\ell') = w_\ell,\forall \ell \in [m]}}
      \int_{\max\limits_{\ell, t} a_{\pi, w'_\ell - 1}^{(t)}}^{\min\limits_{\ell, t} a_{\pi, w'_\ell}^{(t)}}
      J(u, \eta(w'_1), \ldots, \eta(w'_m)) \, \rd u
    \1_{\min\limits_{\ell, t} a_{\pi, w'_\ell}^{(t)} \ge \max\limits_{\ell, t} a_{\pi, w'_\ell - 1}^{(t)}}
  }{
    \sum_{\substack{w'_1,\ldots,w'_m \\ \mathrm{distinct}}}
    \sum_{\substack{\pi \in \Perm(\Voca) \\ \pi(w_\ell') = w_\ell, \forall \ell \in [m]}}
    \PB\left(
      U \in \bigcap_{\ell=1}^m \bigcap_{t \in \IM^Y_\ell}
      \left(a_{\pi, w'_\ell - 1}^{(t)}, a_{\pi, w'_\ell}^{(t)}\right)
    \right)
  },\label{eq:target-expectation-0}
\end{align}
where $\min\limits_{\ell, t}$ denotes $\min_{\ell} \min_{t \in \IM^Y_{\ell}}$ and similarly $\max\limits_{\ell, t}$ denotes $\max_{\ell} \max_{t \in \IM^Y_{\ell}}$ for simplicity.

\paragraph{Numerator.}
We begin by analyzing the numerator of \eqref{eq:target-expectation-0}:
\begin{align}\label{eq:expectation-numerator-0}
   \frac{1}{|\Voca|!} \sum_{\substack{w'_1,\ldots,w'_m \\ \mathrm{distinct}}}
    \sum_{\substack{\pi \in \Perm(\Voca) \\ \pi(w_\ell') = w_\ell,\forall \ell \in [m]}}
    \left(
      \int_{\max\limits_{\ell, t} a_{\pi, w'_\ell - 1}^{(t)}}^{\min\limits_{\ell, t} a_{\pi, w'_\ell}^{(t)}}
      J(u, \eta(w'_1), \ldots, \eta(w'_m)) \, \rd u
    \right)
    \1_{\min\limits_{\ell, t} a_{\pi, w'_\ell}^{(t)} \ge \max\limits_{\ell, t} a_{\pi, w'_\ell - 1}^{(t)}}
\end{align}

Our first step is to rewrite this expression by introducing a random permutation $\pi$. The sum over permutations can then be expressed as an expectation:
\begin{align*}
\EB_{\pi} \left[ \sum_{\substack{w'_1,\ldots,w'_m \\ \text{distinct}}} \1_{\pi(w'_\ell) = w_\ell, \forall \ell \in [m]} \int_{\max\limits_{\ell, t} a_{\pi, w'_\ell - 1}^{(t)}}^{\min\limits_{\ell, t} a_{\pi, w'_\ell}^{(t)}}
J(u,\eta(w'_1),\ldots,\eta(w'_m)) \, \rd u \cdot \1_{\min\limits_{\ell, t} a_{\pi, w'_\ell}^{(t)} \ge \max\limits_{\ell, t} a_{\pi, w'_\ell - 1}^{(t)}} \right].
\end{align*}
By linearity of expectation, we can exchange the expectation and the outer summation over the source tokens $w'_1, \ldots, w'_m$:
\begin{align}
\label{eq:numerator-swapped}
\sum_{\substack{w'_1,\ldots,w'_m \\ \text{distinct}}} \EB_{\pi} \left[ \1_{\pi(w'_\ell) = w_\ell, \forall \ell \in [m]} \int_{\max\limits_{\ell, t} a_{\pi, w'_\ell - 1}^{(t)}}^{\min\limits_{\ell, t} a_{\pi, w'_\ell}^{(t)}}
J(u,\eta(w'_1),\ldots,\eta(w'_m)) \, \rd u \cdot \1_{\min\limits_{\ell, t} a_{\pi, w'_\ell}^{(t)} \ge \max\limits_{\ell, t} a_{\pi, w'_\ell - 1}^{(t)}} \right].
\end{align}
For any fixed set of distinct source tokens $\{w'_\ell\}_{\ell=1}^m$ and distinct target tokens $\{w_\ell\}_{\ell=1}^m$, let
\[
C = \{\pi(w'_\ell) = w_\ell, \forall \ell \in [m]\}
\]
denote the event that $\pi$ maps $w'_\ell$ to $w_\ell$ for each $\ell \in [m]$. This event corresponds to a specific set of permutations, and its total count is $(|\Voca|-m)!$. By the law of total expectation, we can write
\[
\EB[X \cdot \1_C] = \EB[X \mid C] \cdot \PB(C), \quad \text{where} \quad \PB(C) = \frac{(|\Voca| - m)!}{|\Voca|!}
\]
for any random variable $X$.
Substituting this into \eqref{eq:numerator-swapped}, we obtain the following simplified form:
\begin{equation}
\label{eq:numerator-simplified-1}
\begin{split}
\frac{1}{\prod_{\ell = 0}^{m - 1} (|\Voca| - \ell)}
  &\sum_{\substack{w'_{1}, \ldots, w'_{m} \\ \mathrm{distinct}}}
  \EB_{\pi} \Bigg[
    \left(
      \int_{\max\limits_{\ell, t} a^{(t)}_{\pi,w'_\ell-1} \Delta_t}^{
        \min\limits_{\ell, t}a_{\pi,w'_\ell}^{(t)} }
      J\left(u, \eta(w'_1), \ldots, \eta(w'_m)\right) \, \rd u
    \right) \\
& \qquad \qquad\qquad \cdot \1_{
      \min\limits_{\ell, t} a_{\pi,w'_\ell}^{(t)}
      \ge \max\limits_{\ell, t} a^{(t)}_{\pi,w'_\ell-1}
    }
    \,\Bigg|\, \forall \ell,\ \pi(w'_\ell) = w_\ell
  \Bigg].
\end{split}
\end{equation}

Next, we simplify the integration limits, $\max\limits_{\ell, t} a_{\pi,w'\ell-1}^{(t)}$ and $\min\limits_{\ell, t} a_{\pi,w'\ell}^{(t)}$, by applying concentration inequalities under the conditional distribution $\pi \mid C$.
In particular, applying Lemma~\ref{lem:Lipschitz-max} to the maximum function, we obtain
\begin{align}\label{eq:max-concentrationab1}
  &\left| \max_{\substack{\ell \in [m] \\ t \in \IM^Y_{\ell}}} a_{\pi,w'_\ell-1}^{(t)}- \max_{\substack{\ell \in [m] \\ t \in \IM^Y_{\ell}}} \frac{(w'_\ell-1)\Delta_t}{|\Voca|-1} \right|\leq \max_{\substack{\ell \in [m] \\ t \in \IM^Y_{\ell}}}  \left| a_{\pi,w'_\ell-1}^{(t)} - \frac{(w'_\ell-1)\Delta_t}{|\Voca|-1} \right|.
\end{align}
Using Lemma~\ref{lem:concentration-max-cite} below, we bound the right-hand side of \eqref{eq:max-concentrationab1} by $O\left(\frac{1}{|\Voca|} + \sqrt{\eps_{|\Voca|} \log |\Voca|}\right)$, which vanishes as $|\Voca| \to \infty$. The proof of Lemma \ref{lem:concentration-max-cite} can be found in Section \ref{proof:concentration-max-cite}.

\begin{lem}[Concentration of $\max\limits_{\ell, t}a_{\pi,w'_\ell-1}^{(t)}$]\label{lem:concentration-max-cite}
Under Assumption~\ref{asmp:heavy-tokensa}, let $\pi$ be a uniformly random permutation over $\Voca$. Then, for any distinct source tokens $\{w'_\ell\}_{\ell=1}^m$ and target tokens $\{w_\ell\}_{\ell=1}^m$, we have
\begin{align}
\max_{\substack{w'_1,\ldots,w'_m \\ \mathrm{distinct}}} & \EB_{\pi} \left[\max_{\substack{\ell \in [m] \\ t \in \IM^Y_{\ell}}} \left| a_{\pi,w'_\ell-1}^{(t)} - \frac{(w'_\ell-1)\Delta_t}{|\Voca|-1} \right| \Bigg|  \forall \ell \in [m], \pi(w'_\ell) = w_\ell \right] \nonumber \\   
& \leq  O(m) \cdot \max_{\substack{\ell \in [m] \\ t \in \IM^Y_{\ell}}} \left(\frac{1}{|\Voca|} + \sqrt{P_{t, (2)}\log|\Voca|} + P_{t,(2)} \log|\Voca| \right)
\end{align}
where $\Delta_t$ is the regularity level for $\bP_t$, and $O(\cdot)$ hides universal constants.
\end{lem}
Note that
\[
a_{\pi,\, w_\ell'}^{(t)} = \sum_{j=1}^{w_\ell'} P_{t, \pi(j)} = 1 - \Delta_t + \sum_{j=1}^{w_\ell' - 1} P_{t, \pi(j)} = 1 - \Delta_t + a_{\pi,\, w_\ell' - 1}^{(t)}.
\]
Using this identity, we can approximate the lower integration limit $\min\limits_{\ell, t} a_{\pi, w'_\ell}^{(t)}$ as follows:
\begin{align}\label{eq:max-concentrationab12} 
  \left|\min_{\substack{\ell \in [m] \\ t \in \IM^Y_{\ell}}} a_{\pi,w_\ell'}^{(t)} 
  - \min_{\substack{\ell \in [m] \\ t \in \IM^Y_{\ell}}}
  \left[ 1-\frac{(|\Voca|-w_\ell')\Delta_t}{|\Voca|-1}\right]\right|
  &\leq \max_{\substack{\ell \in [m] \\ t \in \IM^Y_{\ell}}} \left|a_{\pi,w_\ell'}^{(t)} - \left(1-\frac{(|\Voca|-w_\ell')\Delta_t}{|\Voca|-1}\right)\right| \nonumber \\
  &= \max_{\substack{\ell \in [m] \\ t \in \IM^Y_{\ell}}}  \left| a_{\pi,w'_\ell-1}^{(t)} - \frac{(w'_\ell-1)\Delta_t}{|\Voca|-1} \right|.
\end{align}

By using this approximation to the upper and lower integral limits, we assert that quantity \eqref{eq:numerator-simplified-1} will be close to the following quantity:
\begin{equation}
\label{eq:numerator-simplified-1.1}
\begin{split}
\frac{1}{\prod_{\ell = 0}^{m - 1} (|\Voca| - \ell)}
  &\sum_{\substack{w'_{1}, \ldots, w'_{m} \\ \mathrm{distinct}}}
  \EB_{\pi} \Bigg[
    \left(
      \int_{\max\limits_{\ell, t} \frac{w'_\ell - 1}{|\Voca| - 1} \Delta_t}^{
        \min\limits_{\ell, t} \left[ 1 - \frac{|\Voca| - w'_\ell}{|\Voca| - 1} \Delta_t \right]}
      J\left(u, \eta(w'_1), \ldots, \eta(w'_m)\right) \, \rd u
    \right) \\
& \qquad \qquad\qquad \cdot \1_{
      \min\limits_{\ell, t} \left[ 1 - \frac{(|\Voca| - w'_\ell) \Delta_t}{|\Voca| - 1} \right]
      \ge \max\limits_{\ell, t} \frac{(w'_\ell - 1) \Delta_t}{|\Voca| - 1}
    }
    \,\Bigg|\, \forall \ell,\ \pi(w'_\ell) = w_\ell
  \Bigg]
\end{split}
\end{equation}
This is because 
\begin{align}
|\eqref{eq:numerator-simplified-1}-\eqref{eq:numerator-simplified-1.1}|
 &\overset{(a)}{\leq}  4\|J\|_{\infty} \cdot \max_{\substack{w_1',\ldots,w_m' \\\mathrm{distinct}}} \EB_{\pi} \left[
 \max_{\substack{\ell \in [m] \\ t \in \IM^Y_{\ell}}} \left|a_{\pi,w_\ell'-1}^{(t)} - \frac{(w_\ell'-1)\Delta_t}{|\Voca|-1}\right|\Bigg|\pi(w_\ell') = w_\ell, \forall \ell \right] \nonumber \\ 
  &\overset{(b)}{\le} O(m) \cdot \|J\|_{\infty} \cdot \max_{\substack{\ell \in [m] \\ t \in \IM^Y_{\ell}}} \left(\frac{1}{|\Voca|} + \sqrt{P_{t, (2)}\log|\Voca|} + P_{t,(2)} \log|\Voca| \right), \label{eq:integral-approximation-error1}
\end{align}
where $(a)$ follows from Lemma~\ref{lem:integral-approximation-error}, with $\|J\|_{\infty} := \sup_{\bm{x} \in [0,1]^{m+1}} |J(\bm{x})|$ denoting the supremum norm of $J$ over $[0,1]^{m+1}$ and $(b)$ follows from Lemma~\ref{lem:concentration-max-cite}, where $O(1)$ denotes a universal constant.

Therefore, it suffices to analyze the expression in~\eqref{eq:numerator-simplified-1.1}. 
Once the upper and lower integration limits are approximated, the entire integrand becomes independent of $\pi$, allowing us to safely remove the expectation over $\pi$.
To study the resulting deterministic quantity, we define the function
\begin{align*}
  \Phi(x_1,\ldots,x_m) 
  &= \left(\int_{\max\limits_{\ell, t}\Delta_t x_\ell}^{\min\limits_{\ell, t}[1-\Delta_t + \Delta_t x_\ell]} J(u, x_1,\ldots,x_m) \rd u \right)\1_{\left\{\min\limits_{\ell, t} \{1-\Delta_t + \Delta_t x_\ell\} \ge \max\limits_{\ell, t}\Delta_t x_\ell \right\} }  \\ 
  &= \left(\int_{\max\limits_{\ell\in [m]}\bDelta_{\ell} x_\ell}^{\min\limits_{\ell\in [m]}[1-\bDelta_{\ell} + \bDelta_{\ell} x_\ell]} J(u, x_1,\ldots,x_m) \rd u \right)\1_{\left\{\min\limits_{\ell\in [m]} [1-\bDelta_{\ell} + \bDelta_{\ell} x_\ell] \ge \max\limits_{\ell\in [m]}\bDelta_{\ell} x_\ell\right\}},
\end{align*}
where $\bDelta_{\ell} := \max_{t \in \IM^Y_{\ell}} \Delta_t$ denotes the maximum regularity level associated with sub-block $\IM^Y_\ell$.
It is straightforward to verify that $\Phi$ is Lipschitz continuous with respect to the $L^\infty$ norm on $[0,1]^m$, owing to the Lipschitz continuity of $J$ and the boundedness of the variables $\{\bDelta_\ell, x_\ell\}_{\ell=1}^m \subseteq [0,1]$.

With this definition, we can rewrite~\eqref{eq:numerator-simplified-1.1} as
\begin{align}
\eqref{eq:numerator-simplified-1.1}&=\frac{1}{\prod_{i=0}^{m-1} (|\Voca|-i)}  \sum_{\substack{w_1',\ldots,w_m' \\ \mathrm{distinct} }}   \Phi(\eta(w_1'),\ldots,\eta(w_m')) \nonumber \\
& \overset{(a)}{=} \frac{1}{|\Voca|^m} \sum_{\substack{w_1',\ldots,w_m' }} \Phi(\eta(w_1'),\ldots,\eta(w_m')) + O\left(\frac{\|J\|_{\infty}}{|\Voca|}\right), \nonumber \\
&\overset{(b)}{=} \int_{[0, 1]^m} \Phi(x_1, \ldots, x_m) \rd x_1 \cdots \rd x_m + O\left(\frac{1}{|\Voca|}\right) +  O\left(\frac{\|J\|_{\infty}}{|\Voca|}\right),
\label{eq:integral-approximation-error2}
\end{align}
where
\begin{itemize}
    \item[(a)] follows from the expansion $\prod_{i=0}^{m-1} (|\Voca|-i) = |\Voca|^m \left[1 + O\left(\frac{m^2}{|\Voca|}\right)\right]$ and the observation that the number of non-fully-distinct $m$-tuples $(\eta(w_1'), \ldots, \eta(w_m'))$ is at most $O(m^2|\Voca|^{m-1})$, with each summand bounded in magnitude by $\|J\|_{\infty}$;
    \item[(b)] follows from approximating the Riemann sum over the uniform grid $\{\eta(w') = (w'-1)/(|\Voca|-1): w'\in\Voca\}$ of mesh size $1/(|\Voca|-1)$, which discretizes $[0,1]$ evenly. Since $\Phi$ is Lipschitz, the resulting Riemann sum converges to the Lebesgue integral with error $O(1/|\Voca|)$ per coordinate, yielding a total approximation error of $O(1/|\Voca|)$.
\end{itemize}

Combining the results above, we conclude that
\begin{align}
&\text{Numerator of}~\eqref{eq:target-expectation-0}
=\eqref{eq:numerator-simplified-1}
\overset{\eqref{eq:integral-approximation-error1}}{=} \eqref{eq:numerator-simplified-1.1} + o(1)
\overset{\eqref{eq:integral-approximation-error2}}{=} 
\int_{[0, 1]^m} \Phi(x_1, \ldots, x_m) \,\rd x_1 \cdots \rd x_m + o(1), \nonumber \\
&=\int_{[0, 1]^m} \int_{\max\limits_{\ell\in [m]}\bDelta_{\ell} x_\ell}^{\min\limits_{\ell\in [m]}[1-\bDelta_{\ell} + \bDelta_{\ell} x_\ell]} J(u, x_1,\ldots,x_m) \1_{\left\{\min\limits_{\ell\in [m]} [1-\bDelta_{\ell} + \bDelta_{\ell} x_\ell] \ge \max\limits_{\ell\in [m]}\bDelta_{\ell} x_\ell\right\}}\rd u \rd x_1 \cdots \rd x_m  + o(1)\nonumber \\
&=\int_{[0,1]^m} \left(\int_{\IM(\bm{x},\bar{\bm{\Delta}})} J(u, \bm{x})\,\rd u\right)\1_{\IM(\bm{x},\bar{\bm{\Delta}}) \ne \emptyset} \rd \bm{x}  + o(1), \label{eq:expectation-numerator-1combine}
\end{align}
where the $o(1)$ term vanishes uniformly as $|\Voca| \to \infty$, over all 1-Lipschitz functions $J$, all $\bm{\bDelta} \in [\Delta, 1-\delta]^{m}$, and all $\bP_{\IM^\zeta} \in \QM_{\tau, \bm{\Delta}}$.

\paragraph{Denominator.}
We now turn to the denominator of \eqref{eq:target-expectation-0}. Since it corresponds to the numerator with the constant function $J \equiv 1$, we set $J(u, x_1, \ldots, x_m) := 1$ and obtain
\begin{align}
\text{Denominator of}~\eqref{eq:target-expectation-0}
  &= \frac{1}{|\Voca|!} \sum_{\substack{w_1', \ldots, w_m' \\ \mathrm{distinct}}} \sum_{\substack{\pi \in \Perm(\Voca) \\ \pi(w_\ell') = w_\ell\; \forall \ell \in [m]}} \PB\left(U \in \bigcap_{\ell \in [m]} \bigcap_{t \in \IM^Y_\ell} (a_{\pi, w_\ell'-1}^{(t)}, a_{\pi, w_\ell'}^{(t)}) \right) \nonumber \\
  &= \int_{[0,1]^m} |\IM(\bm{x}, \bar{\bm{\Delta}})| \cdot \1_{\IM(\bm{x}, \bar{\bm{\Delta}}) \ne \emptyset} \, \rd \bm{x} + o(1) \nonumber \\
  &= I_m(\bar{\bm{\Delta}}) + o(1), \label{eq:expectation-denominator-1combine2}
\end{align}
uniformly over all 1-Lipschitz functions $J$, all parameter vectors $\bm{\bDelta} \in [\Delta, 1-\delta]^{m}$, and all distributions $\bP_{\IM^\zeta}$ in the class $\QM_{\tau, \bm{\Delta}}$ defined in~\eqref{eq:PM-inv-eta-new}.

Finally, combining~\eqref{eq:expectation-numerator-1combine} and~\eqref{eq:expectation-denominator-1combine2} yields the desired result.

\end{proof}

\subsection{Proof of Lemma \ref{lem:concentration-max-cite}}
\label{proof:concentration-max-cite}
\begin{proof}[Proof of Lemma \ref{lem:concentration-max-cite}]

The result follows from the concentration inequality in Lemma~\ref{thm:perm-concentration}, which applies to sums over randomly permuted arrays.
Let
\[
C = \{\pi(w_\ell') = w_\ell, \forall \ell \in [m]\}
\]
denote the event that the random permutation $\pi$ maps each $w_\ell'$ to $w_\ell$.
To apply Lemma~\ref{thm:perm-concentration}, we fix an index $\ell \in [m]$ and define
\[
b^{(t)}_{i,j} = P_{t,j} \cdot \mathbf{1}_{\{i \leq w_\ell'\,,\; j \not \in \{w_1, \ldots, w_m\}\}}, \quad \text{for } i,j \in \Voca.
\]
Recall that $a_{\pi,\, w_\ell'-1}^{(t)} = \sum_{j=1}^{w_\ell'-1} P_{t, \pi(j)}$ and a direct calculation shows that for any $t \in \IM^Y_{\ell}$,
\begin{align}\label{eq:a-pi-wl-1-b}
 \left| \sum_{j\in \Voca} b^{(t)}_{j,\pi(j)}  -   a_{\pi,w_\ell'-1}^{(t)} \right| \leq m P_{t,(2)}.
\end{align}
Note that, conditioned on the event $C$, the permutation $\pi$ is uniformly distributed as a bijection from $\Voca \setminus \{w_1',\ldots,w_m'\}$ to $\Voca \setminus \{w_1,\ldots,w_m\}$. Therefore,
\begin{align*}
\EB_\pi\left[ \sum_{j\in \Voca} b^{(t)}_{j,\pi(j)} \bigg| C \right] 
&= \sum_{j \in \Voca \setminus \{w_1',\ldots,w_m'\}} 
\EB_\pi\left[ P_{t, \pi(j)} \1_{j \le w_\ell'} \,\middle|\, C \right] \\
&= \left|\left\{j \le w_\ell' :  j \notin \{w_1',\ldots,w_m'\}\right\}\right| \cdot \frac{1}{|\Voca|-m} \cdot \sum_{j \in \Voca \setminus \{w_1,\ldots,w_m\}} P_{t, j}.
\end{align*}
Observe that
\[
w_\ell' - m \le \left|\left\{ j \le w_\ell' : j \notin \{w_1',\ldots,w_m'\}\right\}\right| \le w_\ell' - 1.
\]
Using this, we obtain
\begin{align}
\left| \frac{\left|\left\{j \le w_\ell' :  j \notin \{w_1',\ldots,w_m'\}\right\}\right|}{|\Voca|-m}
      - \frac{w_\ell'-1}{|\Voca|-1} \right|
&\le
\frac{m-1}{|\Voca|-m}
+ \frac{(w_\ell'-1)(m-1)}{(|\Voca|-m)(|\Voca|-1)} = O\!\left(\frac{m}{|\Voca|}\right).
\end{align}
On the other hand, since $\sum_{j \in \Voca \setminus \{w_t\}} P_{t, j} = \Delta_t$, we also have
\begin{align}
   \left| \sum_{j \in \Voca \setminus \{w_1,\ldots,w_m\}} P_{t, j} - \sum_{j \in \Voca \setminus \{w_t\}}  P_{t, j} \right| \leq (m-1)P_{t,(2)}.
\end{align}
Putting the bounds together, we conclude that
\begin{align}\label{eq:expectation-b-pbi}
 \left| \EB_\pi\left[\sum_{j\in \Voca} b^{(t)}_{j,\pi(j)} \,\middle|\, C \right] - \frac{(w_\ell'-1)\Delta_t}{|\Voca|-1} \right| \leq 2m \left(P_{t,(2)} + \frac{1}{|\Voca|}\right), \quad \forall t \in \IM^Y_{\ell}.
\end{align}

Thus, Lemma~\ref{thm:perm-concentration} applies to $\sum_{j \in \Voca} b^{(t)}_{j,\pi(j)}$ for each $t \in \IM^Y_{\ell}$ and $\ell \in [m]$.
More specifically, combining \eqref{eq:a-pi-wl-1-b} and \eqref{eq:expectation-b-pbi}, for any $\lambda > 0$, we have that with probability at least $1 - \lambda$, 
\begin{align}
\left| a_{\pi,w_\ell'-1}^{(t)} - \frac{(w_\ell'-1)\Delta_t}{|\Voca|-1} \right| 
\leq O(m) \cdot \left( \sqrt{ \frac{w_\ell'}{|\Voca|} \sum_{j=2}^{|\Voca|} P_{t,(j)}^2 \log \frac{1}{\lambda} } + P_{t,(2)} \log \frac{1}{\lambda} + \frac{1}{|\Voca|} \right),
\end{align}
where a universal constant hidden in $O(1)$.
Finally, we apply a union bound over all choices of distinct tokens $w_1',\ldots,w_m'$ and all $t \in \IM^Y_{\ell}$, $\ell \in [m]$. Setting $\lambda = \frac{1}{m\,|\Voca|^{m+1}}$, we obtain
\begin{align}
  &\sup_{\substack{w_1',\ldots,w_m' \\ \mathrm{distinct}}} \EB_{\pi} \left[
  \max_{\substack{ \ell \in [m] \\ t \in \IM^Y_{\ell}}} \left|a_{\pi,w_\ell'-1}^{(t)} - \frac{(w_\ell'-1)\Delta_t}{|\Voca|-1}\right| \,\Bigg|\,  \pi(w_\ell') = w_\ell, \forall \ell \right] \label{eq:concentration-max-condition} \\
  &\leq O(m) \cdot \max_{\substack{  \ell \in [m] \\ t \in \IM^Y_{\ell}}}
 \left( \frac{1}{|\Voca|} + \sqrt{ \sum_{j=2}^{|\Voca|} P_{t,(j)}^2 \log|\Voca| } + P_{t,(2)} \log|\Voca| \right) \\
 &\le O(m) \cdot \max_{\substack{  \ell \in [m] \\ t \in \IM^Y_{\ell}}}
 \left( \frac{1}{|\Voca|} + \sqrt{ P_{t,(2)} \log|\Voca| } + P_{t,(2)} \log|\Voca| \right).
\end{align}
\end{proof}

\subsection{Proof of Lemma~\ref{lem:inverse-optimal-rule-lower-bound}}
\begin{proof}[Proof of Lemma~\ref{lem:inverse-optimal-rule-lower-bound}]
As we focus on a single block, the sub-block index $k$ is omitted, following the convention in the proof of Theorem~\ref{thm:inverse-asymptotic-distribution}.
For simplicity, we also write $\Delta$ instead of $\Delta_{\VM}$.
For a minimal unit $\VM = \IM^\zeta$ containing $m$ sub-blocks, we represent its associated pivotal statistics $Y_{\VM}$ as the vector $\bm{Y} = (Y_{1}, \ldots, Y_{m})$, where each component corresponds to a distinct sub-block. Since 

Under Assumption~\ref{asmp:heavy-tokensa}, the set of NTP distributions in $\VM$ can be rewritten using the notation $\QM_{\tau,\bm{\Delta}}$ from~\eqref{eq:PM-inv-eta-new} as  
\begin{equation}
\label{eq:relation-P-Q}
\{\bP_{\VM}: \bP_{\VM} \subseteq \SPM\} = \bigcup_{\Delta \le \bm{\Delta} \le 1-\delta} \QM_{\VM,\bm{\Delta}}.
\end{equation}
We aim to show that 
\begin{equation}
\label{eq:target-asymptotic-expression}
\limsup_{|\Voca| \to \infty} \left[ \EB_0[h(\bm{Y})] + \sup_{\bP_{\VM} \subseteq \SPM} \log \EB_{1,\bP_{\VM}}[\exp(-h(\bm{Y}))] \right]  
= \sup_{\Delta \le \bm{\bDelta} } L'(h, \bm{\bDelta}),
\end{equation}
where $\bm{\bDelta} = (\bDelta_{1}, \ldots, \bDelta_{m_k})$, with each $\bDelta_{\ell}$ defined in~\eqref{eq:vector-Delta-bar}, and $L'$ is 
defined by (given in~\eqref{eq:new-sub-optimization})
\begin{equation}  
\tag{\ref{eq:new-sub-optimization}}
L'(h, \bm{\bDelta}) = \EB_{f_0}[h(\bm{Y})] + \log \EB_{f_{\bm{\bDelta}}}[\exp(-h(\bm{Y}))],
\end{equation}
where $f_0$ and $f_{\bm{\bDelta}}$ denote the asymptotic PDFs of $\bm{Y}$ under the null and alternative, given in Theorem \ref{thm:inverse-asymptotic-distribution} respectively.

To prove \eqref{eq:target-asymptotic-expression}, it follows that
\begin{align}
&\limsup_{|\Voca| \to \infty} \left[\EB_0[h(\bm{Y})] + \sup_{\bP_{\VM} \subseteq \SPM} \log \EB_{1,\bP_{\VM}}[\exp(-h(\bm{Y}))] \right]\nonumber \\
  &\overset{(a)}{=}\limsup_{|\Voca| \to \infty} \sup_{\Delta \le  \bm{\Delta} } \sup_{\bP_{\VM} \in \QM_{\tau,\bm{\Delta}}} \left[ \EB_0[h(\bm{Y})] + \log \EB_{1,\bP_{\VM}}[\exp(-h(\bm{Y}))]  \right] \nonumber \\
  &\overset{(b)}{=} \sup_{\Delta \le \bm{\Delta} \le 1-\delta} \sup_{\bP_{\VM} \in \QM_{\tau,\bm{\Delta}}} \limsup_{|\Voca| \to \infty}   \left[ \EB_0[h(\bm{Y})] + \log  \EB_{1,\bP_{\VM}}[\exp(-h(\bm{Y}))]  \right] \nonumber \\ 
  &\overset{(c)}{=}  \sup_{\Delta \le \bm{\Delta} \le 1-\delta} \sup_{\bP_{\VM} \in \QM_{\tau,\bm{\Delta}}} \left[ \EB_{f_0}[h(\bm{Y})] + \log  \EB_{f_{\bm{\bDelta}}}[\exp(-h(\bm{Y}))] \right] 
  \label{eq:expectation-differencewith-zero3} \\
    &\overset{(d)}{=}  \sup_{\Delta \le \bm{\bDelta} \le 1-\delta } \left[ \EB_{f_0}[h(\bm{Y})] + \log  \EB_{f_{\bm{\bDelta}}}[\exp(-h(\bm{Y}))] \right] \nonumber \\
    &= \sup_{\Delta \le \bm{\bDelta} \le 1-\delta }  L'(h, \bm{\bDelta}),
\end{align}
where (a) uses the equivalence in~\eqref{eq:relation-P-Q}, (b) follows by exchanging the order of the $\limsup$ and the suprema, which we will justify later, (c) follows from the weak convergence in Theorem~\ref{thm:inverse-perfect-general-technical-main}, and (d) simplifies the expression by eliminating the dependence on a single $\bP_{\VM}$ and replacing $\bm{\Delta}=(\Delta_t)_{t \in \VM}$ with  $\bm{\bDelta}=(\bDelta_{\ell})_{\ell \in [m_k]}$, where $\bDelta_{\ell} := \max_{t \in \IM^Y_{\ell}} \Delta_t$ for each $\ell$.
At this point, the proof is complete.

In the remainder, we establish the validity of the order exchange in step~(b) above.
To this end, let us introduce a test function $J: [0,1]^{m+1} \to \mathbb{R}$ defined by
\begin{align*}
  J(u, x_1,\ldots,x_m) = \exp\big(-h(|u - x_1|,\ldots,|u - x_m|)\big).
\end{align*}
Since $h$ is Lipschitz continuous and both the exponential and absolute value functions are locally Lipschitz on a bounded domain, their composition $J$ is also Lipschitz continuous.
Theorem~\ref{thm:inverse-perfect-general-technical-main} ensures that the convergence of the expectation of such a function is uniform. Specifically, it guarantees that
\begin{align}\label{eq:expectation-differencewith-zero2}
  \lim_{|\Voca| \to \infty} \sup_{\Delta \le \bm{\Delta} \le 1-\delta} \sup_{\bP_{\VM} \in \QM_{\tau,\bm{\Delta}}} 
  \left| \EB_{1, \bP_{\VM}}\big[J(U, \eta(\pi(w_1)),\ldots, \eta(\pi(w_m)))\big] - \mathcal{L}_{\bar{\bm{\Delta}}}(J) \right| = 0,
\end{align}
where $\mathcal{L}_{\bar{\bm{\Delta}}}(J)$ is the asymptotic integral form
\[
\mathcal{L}_{\bar{\bm{\Delta}}}(J) 
:= 
\int_{[0,1]^{m}} 
\int_{\max_\ell\{\Delta_{\ell} x_\ell\}}^{\min_\ell\{1-\Delta_{\ell} + \Delta_{\ell} x_\ell\}} 
\frac{J(u, x_1,\dots,x_{m})}{I_{m}(\bar{\bm{\Delta}})} \,\rd u \,
\mathbf{1}_{\{\min_i\{1-\Delta_{i} + \Delta_{i} x_i\} \ge \max_i\{\Delta_{i} x_i\}\}} 
\, \rd x_1 \cdots \rd x_{m},
\]
and $\bar{\bm{\Delta}}$ denotes the sub-block-level vector derived from $\bm{\Delta}$.
By the definition of $J$, the uniform convergence in~\eqref{eq:expectation-differencewith-zero2} is equivalent to the uniform convergence of the moment-generating function term:
\begin{align*}
  \lim_{|\Voca| \to \infty} \sup_{\Delta \le \bm{\Delta} \le 1-\delta} \sup_{\bP_{\VM} \in \QM_{\tau,\bm{\Delta}}} 
  \left| \EB_{1, \bP_{\VM}}\big[\exp(-h(\bm{Y}))\big] - \mathcal{L}_{\bar{\bm{\Delta}}}(\exp(-h)) \right| = 0,
\end{align*}
which is precisely~\eqref{eq:expectation-differencewith-zero2}.
Since $\mathcal{L}_{\bar{\bm{\Delta}}}(\exp(-h)) = \EB_{f_{\bm{\bDelta}}}[\exp(-h(\bm{Y}))]$ (by the equivalence in Corollary~\ref{cor:inverse-perfect-general-technical-main}), and the supremum is a nonexpansive operator (see Lemma~\ref{lem:Lipschitz-max}), we prove step (b).
\end{proof}

\subsection{Proof of Lemma~\ref{lem:inverse-optimal-rule-upper-bound-true}}
\label{sec:proof-lem-inverse-optimal-rule-upper-bound-truegood}

\begin{proof}[Proof of Lemma~\ref{lem:inverse-optimal-rule-upper-bound-true}]

Since we focus on a single block $\VM$, we write $\Delta$ instead of $\Delta_{\VM}$ for simplicity. In this lemma, the alternative PDF is evaluated at the homogeneous regularity-level vector $(\Delta, \ldots, \Delta)$ and therefore depends only on the single parameter $\Delta$.
For simplicity, we write
\[
f_{1,\Delta} := f_{(\Delta, \ldots, \Delta)},
\]
to emphasize its dependence on the single parameter $\Delta$.  
For brevity, we define the truncated log-likelihood ratio as
\begin{align}
  h_{\mathrm{opt},M}(\bm{y}) 
  := \Biggl[\log \frac{f_{1,\Delta}(\bm{y})}{f_0(\bm{y})}\Biggr]_{[-M,M]},
\end{align}
where $[\,\cdot\,]_{[-M,M]}$ denotes truncation to the interval $[-M,M]$.

\begin{lem}[Properties of the alternative PDF]
\label{lem:joint-density-function-propertiesa}
Let $f_{\bm{\bDelta}}$ denote the joint alternative PDF of the pivotal statistic 
vector $\bm{Y} = (Y_1, \ldots, Y_m)$ under $H_1$, where the regularity levels are 
$\bm{\bDelta} = (\bDelta_1, \ldots, \bDelta_m)$. 
Its explicit form is given in Theorem~\ref{thm:inverse-asymptotic-distribution}. 
Then, under Assumption~\ref{asmp:heavy-tokensa}, 
\begin{itemize}
\item $f_{\bm{\bDelta}}(\bm{y})$ is strictly positive for $\bm{y} \in \prod_{\ell=1}^m [0, 1-\bDelta_{\ell})$, and equals $0$ on the complement set $[0,1]^m \setminus \prod_{\ell=1}^m [0, 1-\bDelta_{\ell})$.
\item The mapping $(\bm{y}, \bm{\bDelta}) \mapsto f_{\bm{\bDelta}}(\bm{y})$ 
is Lipschitz continuous on its domain 
$[0,1]^m \times [0,1-\delta]^m$. 
That is, there exists a universal constant $C > 0$ such that for any 
$(\bm{y}, \bm{\bDelta}), (\bm{y}', \bm{\bDelta}') \in [0,1]^m \times [0,1-\delta]^m$,
\[
\big| f_{\bm{\bDelta}}(\bm{y}) - f_{\bm{\bDelta}'}(\bm{y}') \big|
\;\leq\; C  \cdot \Big( \|\bm{y} - \bm{y}'\|_\infty + \|\bm{\bDelta} - \bm{\bDelta}'\|_\infty \Big).
\]
\end{itemize}
\end{lem}

Before proceeding with the proof, we first establish some properties of the PDF $f_{1,\Delta}$. The proof of this lemma is provided in Section~\ref{sec:proof-lem-joint-density-function-propertiesal5}.

With $L'$ defined in~\eqref{eq:new-sub-optimization}, we can decompose 
\begin{align}
\label{eq:decomposed-L'}
  L'(h_{\text{opt},M}, \bm{\bDelta}') 
  &= \underset{\text{(I)}}{\underbrace{\EB_{f_0}[h_{\text{opt},M}(\bm{Y})]}} 
  + \underset{\text{(II)}}{\underbrace{\log \EB_{f_{\bm{\bDelta}'}}[\exp(-h_{\text{opt},M}(\bm{Y}))]}}.
\end{align}
We then bound terms (I) and (II) separately.

\paragraph{Analysis of Term (I).}
By Lemma~\ref{lem:joint-density-function-propertiesa}, $f_{1,\Delta}$ is supported on $[0, 1-\Delta)^m$ and is Lipschitz continuous in $(\bm{y}, \Delta)$. Since $f_0$ coincides with $f_{1,\Delta}$ when $\Delta=0$ (see Corollary~\ref{cor:inverse-perfect-general-technical-main} and Lemma~\ref{lem:asymp-H0}), the null density $f_0(\bm{y})$ is continuous and strictly positive on $[0,1)^m$. This ensures the existence of a finite upper bound for the following likelihood ratio:
\[
    M' := \sup_{\Delta' \in [0,1]} \frac{\sup_{\bm{y} \in [0,1]^m} f_{1,\Delta'}(\bm{y})}{\inf_{\bm{y} \in [0,1-\Delta)^m} f_0(\bm{y})} < \infty.
\]
Importantly, $M'$ is independent of the clipping threshold $M$. When $M \ge |\log M'|$, the log-likelihood ratio $\log \tfrac{f_{1,\Delta}(\bm{Y})}{f_0(\bm{Y})}$ never exceeds $M$, so the score function $h_{\text{opt},M}$ can only be clipped at $-M$, which occurs when $\bm{Y}$ lies outside the support of $f_{\bm{\bDelta}}$, that is, outside $[0,1-\Delta)^m$.
\begin{align}
\EB_{f_0}[h_{\text{opt},M}(\bm{Y})] 
&= \EB_{f_0}\left[\log \frac{f_{1,\Delta}(\bm{Y})}{f_0(\bm{Y})}\right]_{[-M,M]} \nonumber \\
&= \int_{[0,1-\Delta)^m} \left[\log \frac{f_{1,\Delta}(\bm{y})}{f_0(\bm{y})}\right]_{[-M,M]} f_0(\bm{y}) \,\rd\bm{y} 
+ \int_{[0,1]^m \setminus [0,1-\Delta)^m} (-M) f_0(\bm{y}) \,\rd \bm{y} \nonumber \\
&\leq \int_{[0,1-\Delta)^m} (\log M') \cdot f_0(\bm{y}) \,\rd \bm{y} 
- M \cdot \PB_{f_0}(\bm{Y} \notin [0,1-\Delta)^m) \nonumber \\
&\leq (\log M') \cdot \PB_{f_0}(\bm{Y} \in [0,1-\Delta)^m) 
- M \cdot \bigl(1 - \PB_{f_0}(\bm{Y} \in [0,1-\Delta)^m)\bigr). 
\label{eq:bound-term-I-intermediate}
\end{align}
Although the probability $\PB_{f_0}(\bm{Y} \in [0,1-\Delta)^m)$ depends on the shape of $f_0$, it is bounded away from both $0$ and $1$ when $0 < \Delta < 1$. Hence, there exist positive constants $c$ and $C$ such that
\begin{align}
\EB_{f_0}[h_{\text{opt},M}(\bm{Y})] \leq -M c + C,
\label{eq:ell-sigma-delta-positiveargue2com5}
\end{align}
where $c$ and $C$ depend only on the fixed $\Delta$ used to define the score function $h_{\text{opt},M}$, but are independent of $M$ and $\bm{\Delta}'$, the latter introduced in (II).

\paragraph{Analysis of Term (II).}
We now analyze term (II), which takes the form of an integral over $[0,1]^m$
\begin{align*}
  &\EB_{f_{\bar{\bm{\Delta}}'}}[\exp(-h_{\text{opt},M}(\bm{Y}))] = \int_{[0,1]^m} \exp\left(-\left[\log \frac{f_{\bm{\bDelta}}(\bm{y})}{f_0(\bm{y})}\right]_{[-M,M]}\right) f_{\bar{\bm{\Delta}}'}(\bm{y})\,\rd\bm{y},
\end{align*}
where $\bar{\bm{\Delta}} = (\Delta, \ldots, \Delta)$ and $\bar{\bm{\Delta}}' = (\Delta_1', \ldots, \Delta_m') \in \RB^m$. 
By Lemma~\ref{lem:joint-density-function-propertiesa}, the PDF $f_{\bar{\bm{\Delta}}'}$ vanishes outside the set $\prod_{\ell=1}^m [0, 1-\Delta_\ell')$. Hence, the domain of integration can be restricted to this support without altering the integral:
\begin{align*}
  \EB_{f_{\bar{\bm{\Delta}}'}}[\exp(-h_{\text{opt},M}(\bm{Y}))] = \int_{\prod_{\ell=1}^m [0,1-\Delta'_\ell)} \exp\left(-\left[\log \frac{f_{\bm{\bDelta}}(\bm{y})}{f_0(\bm{y})}\right]_{[-M,M]}\right) f_{\bar{\bm{\Delta}}'}(\bm{y})\,\rd\bm{y}.
\end{align*}

From the analysis of term (I), we know that $\log \tfrac{f_{\bm{\bDelta}}(\bm{y})}{f_0(\bm{y})}$ never exceeds $M$ once $M \ge |\log M'|$. Therefore, the clipping interval $[-M, M]$ can safely be replaced by $[-M, \infty)$. This yields
\begin{align}
\EB_{f_{\bar{\bm{\Delta}}'}}[\exp(-h_{\text{opt},M}(\bm{Y}))] 
&= \int_{\prod_{\ell=1}^m [0,1-\Delta'_\ell)}\exp\left(-\left[\log \frac{f_{\bm{\bDelta}}(\bm{y})}{f_0(\bm{y})}\right]_{[-M,\infty)}\right) f_{\bar{\bm{\Delta}}'}(\bm{y}) \,\rd \bm{y} \nonumber \\
&\overset{(a)}{\le} \int_{\prod_{\ell=1}^m [0,1-\Delta'_\ell)} \exp\left(-\log \frac{f_{\bm{\bDelta}}(\bm{y})}{f_0(\bm{y})}\right) f_{\bar{\bm{\Delta}}'}(\bm{y}) \,\rd\bm{y} \nonumber \\
&= \int_{\prod_{\ell=1}^m [0,1-\Delta'_\ell)} \frac{f_{\bar{\bm{\Delta}}'}(\bm{y})}{f_{\bm{\bDelta}}(\bm{y})} f_0(\bm{y})  \,\rd\bm{y}, \nonumber \\
&\overset{(b)}{\le} \int_{\prod_{\ell=1}^m [0,1-\Delta'_\ell)}f_0(\bm{y})\,\rd\bm{y} \cdot R \le R, 
\label{eq:ell-sigma-delta-positiveargue2com4}
\end{align}
where (a) holds because removing the lower clipping at $-M$ can only increase the integral, and (b) follows from the uniform boundedness of $\tfrac{f_{\bar{\bm{\Delta}}'}(\bm{y})}{f_{\bm{\bDelta}}(\bm{y})}$ established in Lemma~\ref{lem:bounded-ratio}.

\begin{lem}[Uniformly bounded density ratio]
\label{lem:bounded-ratio}
There exists a constant $R > 0$, independent of $M$ and $\bm{\bDelta}'$, such that
\[
\sup_{\Delta \le \bm{\bDelta}' \le 1-\delta}
\;\sup_{\bm{y} \in \prod_{\ell=1}^m [0,1-\Delta'_\ell)}
\left| \frac{f_{\bar{\bm{\Delta}}'}(\bm{y})}{f_{\bm{\bDelta}}(\bm{y})} \right|
\;\le\; R.
\]
\end{lem}

Combining \eqref{eq:decomposed-L'}, \eqref{eq:ell-sigma-delta-positiveargue2com5}, and \eqref{eq:ell-sigma-delta-positiveargue2com4}, we conclude that there exist positive constants $c, C, R > 0$, independent of $\bm{\bDelta}'$ and $M$, such that
\[
L'(h_{\text{opt},M}, \bm{\bDelta}')  \le - M c + (C+R).
\]
Taking the supremum over $\bm{\bDelta}'$ and then letting $M \to \infty$ gives
\[
 \liminf_{M \to \infty} \;\sup_{\Delta \leq \bm{\bDelta}' \le 1-\delta} 
 L'(h_{\text{opt},M}, \bm{\bDelta}') = -\infty.
\]

\end{proof}

Finally, we provide the proof of Lemma~\ref{lem:bounded-ratio} below.
\begin{proof}[Proof of Lemma \ref{lem:bounded-ratio}]
Fix any $\bm{y} = (y_1, \ldots, y_m) \in \prod_{\ell=1}^m [0, 1- \Delta_{\ell}')$.
Without loss of generality, assume $y_1$ is the largest coordinate of $\bm{y}$.  
By definition, we have $y_1 < 1-\Delta_1' \le 1-\Delta$.  
We then construct the auxiliary vector $\bm{y}' = (1-\Delta_1', y_2, \ldots, y_m)$, which differs from $\bm{y}$ only in its largest entry.  
By the Lipschitz continuity in Lemma~\ref{lem:joint-density-function-propertiesa}, we know $f_{\bar{\bm{\Delta}}'}(\bm{y}') = 0$, and 
\begin{equation}
\label{eq:ratio-numeriator}
|f_{\bar{\bm{\Delta}}'}(\bm{y})|
\le |f_{\bar{\bm{\Delta}}'}(\bm{y}) -  f_{\bar{\bm{\Delta}}'}(\bm{y}')|
\le C \cdot \|\bm{y}-\bm{y}'\| 
\le C \cdot \left(1-\Delta - \max_{\ell \in [m]} y_\ell\right).
\end{equation}
where the constant $C$, given in Lemma~\ref{lem:joint-density-function-propertiesa}, is independent of $M$ and of $\bm{\bDelta}'$.  

On the other hand, Theorem~\ref{thm:inverse-asymptotic-distribution} gives
\begin{align}
f_{\bm{\bDelta}}(\bm{y}) 
&= I_m(\bm{\bDelta})^{-1} \sum_{\bm{\sigma} \in \{-1,1\}^m} \left(B_{\bm{\sigma}}^{\bm{\bDelta}}(\bm{y})-A_{\bm{\sigma}}^{\bm{\bDelta}}(\bm{y})\right)_+\nonumber \\
&\overset{(a)}{\ge} \left(B_{\bm{\sigma}'}^{\bm{\bDelta}'}(\bm{y})-A_{\bm{\sigma}'}^{\bm{\bDelta}}(\bm{y})\right)_+
\overset{(b)}{=} 1 - \frac{\max_{\ell \in [m]} y_\ell}{1-\Delta},
\label{eq:ratio-denumeriator}
\end{align}
where (a) uses the fact that $I_m(\bm{\bDelta}) \le 1$ (since it is a probability) and keeps only the term with $\bm{\sigma}' = (-1, \ldots, -1)$, while (b) follows directly from the definitions of $B_{\bm{\sigma}'}^{\bm{\bDelta}}(\bm{y})$ and $A_{\bm{\sigma}'}^{\bm{\bDelta}}(\bm{y})$.  

Combining \eqref{eq:ratio-numeriator} and \eqref{eq:ratio-denumeriator} completes the proof with $R = C$.  
\end{proof}

\subsection{Proof of Lemma~\ref{lem:joint-density-function-propertiesa}}
\label{sec:proof-lem-joint-density-function-propertiesal5}

\begin{proof}[Proof of Lemma~\ref{lem:joint-density-function-propertiesa}]
From Theorem \ref{thm:inverse-asymptotic-distribution}, we know that
\begin{align}
  f_{\bm{\bDelta}}
  (\bm{y}) 
  = \frac{1}{I_{m}(\bm{\bDelta})} \sum_{\bm{\sigma} \in \{-1,1\}^{m}} \left( B_{\bm{\sigma}}^{\bm{\bDelta}}(\bm{y}) - A_{\bm{\sigma}}^{\bm{\bDelta}}(\bm{y}) \right)_+,
\end{align}
where for each sign vector $\bm{\sigma} = (\sigma_1, \ldots, \sigma_{m_k}) \in \{-1,1\}^{m_k}$ and input $\bm{y} = (y_1, \ldots, y_{m})$,
\begin{align*}
  L_{\bm{\sigma}}(\bm{y}) &:= \max_{\ell \in [m]} (-\sigma_\ell y_\ell), & 
  U_{\bm{\sigma}}(\bm{y}) &:= \min_{\ell \in [m]} (1 - \sigma_\ell y_\ell), \\
  Y^+_{\bm{\sigma}}(\bm{y}) &:= \left( \max_{\ell: \sigma_{\ell}=1}\frac{\bDelta_{\ell} }{1 - \bDelta_{\ell}} \cdot y_\ell \right)_+, & 
  Y^-_{\bm{\sigma}}(\bm{y}) &:= \left( \max_{\ell: \sigma_{\ell}=-1}\frac{\bDelta_{\ell} }{1 - \bDelta_{\ell}} \cdot y_\ell \right)_+, \\
  A_{\bm{\sigma}}^{\bm{\bDelta}}(\bm{y}) &:= \max\left\{ L_{\bm{\sigma}}(\bm{y}),\ Y^+_{\bm{\sigma}}(\bm{y}) \right\}, & 
  B_{\bm{\sigma}}^{\bm{\bDelta}}(\bm{y}) &:= \min\left\{ U_{\bm{\sigma}}(\bm{y}),\ 1 - Y^-_{\bm{\sigma}}(\bm{y}) \right\},
\end{align*}
with $(x)_+ := \max(x, 0)$, and the normalization constant $I_{m}(\bm{\bDelta}_k)$ is given by
\[
I_{m}(\bm{\bDelta}) := \int_{[0,1]^{m}} \left( \min_{\ell \in [m]} \{1 - \bDelta_{k,\ell} + \bDelta_{k,\ell} x_\ell\} - \max_{\ell \in [m]} \{\bDelta_{k,\ell} x_\ell\} \right)_+ \rd x_1 \cdots \rd x_{m}.
\]

\paragraph{Part 1: Support of $f_{\bm{\bDelta}}$.}
We first show that $f_{\bm{\bDelta}}(\bm{y}) = 0$ whenever $\bm{y}\notin \prod_{\ell=1}^m [0,1-\bDelta_\ell)$.
Take any $\bm{y}\in [0,1]^m$ with $y_\ell \ge 1-\bDelta_\ell$ for some index $\ell$. Consider two cases depending on $\sigma_\ell$:
\begin{enumerate}
  \item If $\sigma_\ell = 1$, then
  \[
  Y^+_{\bm{\sigma}}(\bm{y}) \;\ge\; \tfrac{\bDelta_\ell}{1-\bDelta_\ell}(1-\bDelta_\ell) = \bDelta_\ell,
  \quad
  U_{\bm{\sigma}}(\bm{y}) \;\le\; 1 - (1-\bDelta_\ell) = \bDelta_\ell.
  \]
  Hence, $A_{\bm{\sigma}}^{\bm{\bDelta}}(\bm{y}) \ge \bDelta_\ell$ and 
  $B_{\bm{\sigma}}^{\bm{\bDelta}}(\bm{y}) \le \bDelta_\ell$, so that 
  $(B_{\bm{\sigma}}^{\bm{\bDelta}}-A_{\bm{\sigma}}^{\bm{\bDelta}})_+ = 0$.
  \item If $\sigma_\ell = -1$, then
  \[
  Y^-_{\bm{\sigma}}(\bm{y}) \;\ge\; \tfrac{\bDelta_\ell}{1-\bDelta_\ell}(1-\bDelta_\ell)=\bDelta_\ell,
  \quad
  L_{\bm{\sigma}}(\bm{y}) \;\ge\; 1-\bDelta_\ell.
  \]
  Consequently, $A_{\bm{\sigma}}^{\bm{\bDelta}}(\bm{y}) \ge 1-\bDelta_\ell$ while 
  $B_{\bm{\sigma}}^{\bm{\bDelta}}(\bm{y}) \le 1-\bDelta_\ell$, so again 
  $(B_{\bm{\sigma}}^{\bm{\bDelta}}-A_{\bm{\sigma}}^{\bm{\bDelta}})_+ = 0$.
\end{enumerate}
Since this holds for every $\bm{\sigma}$, the entire sum vanishes and $f_{\bm{\bDelta}}(\bm{y})=0$.  
Thus the support is contained in $\prod_{\ell=1}^m [0,1-\bDelta_\ell)$.

Conversely, let $\bm{y}\in \prod_{\ell=1}^m [0,1-\bDelta_\ell-\eps]$ for any small $\eps>0$.  
Take $\bm{\sigma} = (1,\ldots,1)$. Then
\[
L_{\bm{\sigma}}(\bm{y}) \le 0, 
\quad 
U_{\bm{\sigma}}(\bm{y}) \ge \bDelta_\ell + \eps \ \text{for all } \ell,
\quad 
Y^-_{\bm{\sigma}}(\bm{y}) = 0,
\]
and
\[
Y^+_{\bm{\sigma}}(\bm{y}) 
\le \max_\ell \tfrac{\bDelta_\ell}{1-\bDelta_\ell}(1-\bDelta_\ell-\eps)
\le \max_\ell\bDelta_\ell - \min_{\ell}\tfrac{\bDelta_\ell}{1-\bDelta_\ell} \eps.
\]
Hence,
\[
A_{\bm{\sigma}}^{\bm{\bDelta}}(\bm{y}) = Y^+_{\bm{\sigma}}(\bm{y}) 
  \le \max_\ell\bDelta_\ell - \min_{\ell}\tfrac{\bDelta_\ell}{1-\bDelta_\ell} \eps,
\quad
B_{\bm{\sigma}}^{\bm{\bDelta}}(\bm{y}) = U_{\bm{\sigma}}(\bm{y}) \ge \max_\ell \bDelta_\ell+\eps.
\]
By Assumption~\ref{asmp:heavy-tokensa}, $0 < \bDelta_\ell<1$, so the gap
\[
B_{\bm{\sigma}}^{\bm{\bDelta}}(\bm{y}) - A_{\bm{\sigma}}^{\bm{\bDelta}}(\bm{y})
  \;\ge\; \frac{\eps}{1-\max_\ell \bDelta_\ell} > 0.
\]
Thus, at least one term in the summation is strictly positive, and 
$f_{\bm{\bDelta}}(\bm{y}) > 0$.  
This proves that $f_{\bm{\bDelta}}$ is strictly positive on 
$\prod_{\ell=1}^m [0,1-\bDelta_\ell)$ and zero elsewhere.

\paragraph{Part 2: Lipschitz continuity.}

The joint density $f_{\bm{\bDelta}}(\bm{y})$ is uniformly continuous in 
$(\bm{y},\bm{\bDelta})$ on $[0,1]^m\times[0,1-\delta]^m$, since it is given by a finite sum of continuous functions on a compact domain. To strengthen this to Lipschitz continuity, recall from Theorem~\ref{thm:inverse-asymptotic-distribution} that each summand in the representation of $f_{\bm{\bDelta}}$ is constructed from a finite combination of linear functions, maxima, minima, and positive-part operators. Each of these building blocks is Lipschitz continuous with constants depending only on $(m, \delta)$, and finite maxima/minima of Lipschitz functions remain Lipschitz with constants given by the maximum of the individual constants. 

Therefore, every summand is Lipschitz continuous with respect to $(\bm{y},\bm{\bDelta})$, uniformly on the domain $[0,1]^m\times[0,1-\delta]^m$. Since $f_{\bm{\bDelta}}$ is a finite sum of such summands, it follows that $f_{\bm{\bDelta}}$ itself is Lipschitz continuous, with a constant depending only on $(m,\delta)$, but independent of $(\bm{y},\bm{\bDelta})$. 
 
   \end{proof}

\subsection{Auxiliary Lemmas}

\begin{lem}\label{lem:I_k_Delta}
Let
\[
I_m(\Delta)
=\int_{[0,1]^m}\bigl(1-\Delta-\Delta\,D(x_1,\ldots,x_m)\bigr)_+\,\rd x_1\cdots \rd x_m,
\quad
D(x_1,\ldots,x_m)=\max_{1\le i\le m}x_i-\min_{1\le i\le m}x_i.
\]
Then for $0\le\Delta < 1$, we have the closed form
\[
I_m(\Delta)
=\begin{cases}
1 - \dfrac{2m}{m+1}\,\Delta, & 0 \le \Delta \le \tfrac{1}{2}, \\[1em]
\displaystyle
\left(\frac{1-\Delta}{\Delta}\right)^m
\left[\,1 - \frac{2m(1-\Delta)}{m+1}\right], & \tfrac{1}{2} < \Delta < 1.
\end{cases}
\]
Moreover, $I_m(\Delta)$ is uniformly continuous on $[0, 1-\delta]$ for any given $\delta \in (0, 1)$.
\end{lem}

\begin{proof}[Proof of Lemma~\ref{lem:I_k_Delta}]
Let $X_1, \ldots, X_m$ be i.i.d. $\mathrm{Unif}(0,1)$. Then the target quantity can be expressed as $I_m(\Delta) = \EB\left[\left(1 - \Delta - \Delta\,D\right)_+\right]$, where $D = \max_i X_i - \min_i X_i \in [0,1]$. The density of $D$ is given by $f_D(r) = m(m-1)\,r^{m-2}(1 - r)$ for $r \in [0,1]$, as stated in formula (2.5.15) of \cite{nason1994first}.
Hence,
\[
I_m(\Delta)
= \int_0^1 \left(1 - \Delta - \Delta r\right)_+ f_D(r)\,\rd r
= \int_0^{r_0} (1 - \Delta - \Delta r)\,m(m-1)\,r^{m-2}(1 - r)\,\rd r,
\]
where $r_0 = (1 - \Delta)/\Delta$, since the integrand becomes zero for $r > r_0$.

\medskip\noindent
\textbf{Case 1: $\Delta \le \tfrac{1}{2}$.}  Then $r_0 \ge 1$, so the $(\cdot)_+$ operator has no effect over $r \in [0,1]$. Therefore,
\[
I_m(\Delta)
= \int_0^1 \left((1 - \Delta) - \Delta r\right)\,m(m-1)\,r^{m-2}(1 - r)\,dr.
\]
This evaluates to
\[
I_m(\Delta) = (1 - \Delta) - \Delta \cdot \frac{m - 1}{m + 1}
= 1 - \frac{2m}{m + 1}\,\Delta.
\]

\medskip\noindent
\textbf{Case 2: $\Delta > \tfrac{1}{2}$.} Then $r_0 < 1$, and the integral becomes
\[
I_m(\Delta)
= m(m - 1) \int_0^{r_0} (1 - \Delta - \Delta r)\,r^{m-2}(1 - r)\,dr.
\]
Let
\[
A = \int_0^{r_0} r^{m-2}(1 - r)\,dr = \frac{r_0^{m - 1}}{m - 1} - \frac{r_0^m}{m}, \quad
B = \int_0^{r_0} r^{m - 1}(1 - r)\,dr = \frac{r_0^m}{m} - \frac{r_0^{m + 1}}{m + 1}.
\]
Then we have
\[
I_m(\Delta)
= m(m - 1)\left[(1 - \Delta)A - \Delta B\right].
\]
Substituting $r_0 = (1 - \Delta)/\Delta$ and simplifying gives
\[
I_m(\Delta)
= \left(\frac{1 - \Delta}{\Delta}\right)^m
\left[1 - \frac{2m(1 - \Delta)}{m + 1}\right].
\]

\smallskip
The uniform continuity of $I_m(\Delta)$ over $[0, 1-\delta]$ follows directly from the smoothness of the integrand and the compactness of the domain. This concludes the proof.
\end{proof}

\begin{lem}[\cite{albertConcentrationInequalitiesRandomly2019}, Theorem 2.1]
  \label{thm:perm-concentration}
Let $\{a_{i,j}\}_{1\le i,j\le |\Voca|}$ be a collection of real numbers, and let $\pi$ be a uniformly random permutation  on $\Voca$.  Define
\[
Z \;=\; \sum_{j=1}^{|\Voca|} a_{j,\pi(j)}.
\]
Then for all $x>0$,
\begin{equation}\label{eq:perm-concentration}
\PB\left(\bigl|Z-\EB[Z]\bigr|
\;\ge\;2\sqrt{\,2\Bigl(\tfrac{1}{|\Voca|}\sum_{i,j=1}^{|\Voca|} a_{i,j}^2\Bigr)\,x
\;+\;2\max_{1\le i,j\le |\Voca|}|a_{i,j}|\;x}\right)
\;\le\;16\,e^{1/16}\,\exp\Bigl(-\tfrac{x}{16}\Bigr).
\tag{15}
\end{equation}
\end{lem}

\begin{lem}
\label{lem:Lipschitz-max}
The maximum function $\max: [0,1]^m \to [0,1]$ is Lipschitz continuous with constant $1$ with respect to the $L^\infty$ norm. That is, for any $x, y \in [0,1]^m$, we have
\begin{equation}
\label{eq:Lipschitz-max}
|\max(x_1,\ldots,x_m) - \max(y_1,\ldots,y_m)| \leq \max_{i=1,\ldots,m} |x_i - y_i|.
\end{equation}
Similarly, the minimum function $\min: [0,1]^m \to [0,1]$ is also Lipschitz continuous with constant $1$ under the $L^\infty$ norm:
\begin{equation}
\label{eq:Lipschitz-min}
|\min(x_1,\ldots,x_m) - \min(y_1,\ldots,y_m)| \leq \max_{i=1,\ldots,m} |x_i - y_i|.
\end{equation}
\end{lem}

\begin{proof}[Proof of Lemma~\ref{lem:Lipschitz-max}]
Without loss of generality, assume $\max(x_1,\ldots,x_m) = x_{i_0} \ge \max(y_1,\ldots,y_m) = y_{j_0}$ for some indices $i_0, j_0 \in [m]$. Since $y_{j_0} \ge y_{i_0}$ (as $y_{j_0}$ is the maximum of $y$), we have
\begin{align}
|\max(x_1,\ldots,x_m) - \max(y_1,\ldots,y_m)| &= x_{i_0} - y_{j_0} \\
&\le x_{i_0} - y_{i_0} \le \max_{i=1,\ldots,m} |x_i - y_i|.
\end{align}
This proves the Lipschitz continuity of the maximum function. The result for the minimum follows by noting that
\[
\min(x_1,\ldots,x_m) = -\max(-x_1,\ldots,-x_m),
\]
and applying the same argument to $-x$ and $-y$.
\end{proof}

\begin{lem}\label{lem:joint-law-UkX}
Let $X_{1}, \dots, X_{m}$ be independent $\mathrm{Unif}(0,1)$ random variables, and conditionally on $(X_{1}, \dots, X_{m}) = \bm{x}:=(x_1, \ldots, x_m)$, let
\[
U \sim \mathrm{Unif}\bigl[a(\bm{x}),\, b(\bm{x})\bigr],
\]
where $a(\bm{x}) = \max_{1 \le \ell \le m} \{\Delta_\ell x_\ell\}$ and $b(\bm{x}) = \min_{1 \le \ell \le m} \{\Delta_\ell x_\ell + 1 - \Delta_\ell\}$ with the convention that $U$ has no mass if $a(x) \ge b(x)$. 
With $\bar{\bm{\Delta}}=(\bDelta_1, \ldots, \bDelta_m)$, define
\[
I_m(\bar{\bm{\Delta}})
= \int_{[0,1]^m} \bigl(b(x) - a(x)\bigr)_+\, \rd x.
\]
Then for any measurable function $J \colon [0,1]^{m+1} \to [0,\infty)$, we have
\begin{align}
  \label{eq:joint-law-expectatioan}
\EB\bigl[J(U, X_1, \dots, X_m)\bigr]
=
\frac{1}{I_m(\bar{\bm{\Delta}})}
\int_{[0,1]^m}
\int_{a(x)}^{b(x)}
J\bigl(u, x_1, \dots, x_m\bigr) \rd u\,
\cdot \mathbf{1}_{a(\bm{x}) < b(\bm{x})}\,  \rd \bm{x},
\end{align}
and hence the right-hand side defines the joint law of $(U, X_1, \dots, X_m)$.
\end{lem}

\begin{proof}[Proof of Lemma \ref{lem:joint-law-UkX}]
This follows directly from the law of total expectation and the conditional definition of $U$, so we omit the details.
\end{proof}

\begin{lem}[Integral approximation error]\label{lem:integral-approximation-error}
Let $J : [0, 1]^{m+1} \to \mathbb{R}$ be a 1-Lipschitz function.  
Let $a_\pi$ and $b_\pi$ be random variables depending on a random variable $\pi$, and let $\bar{a}, \bar{b} \in [0,1]$ be deterministic values.  
Then, for any fixed fractions $\{x_\ell\}_{\ell=1}^m$, the following bound holds:
\begin{align}
&\left| \EB_{\pi} \left[ \int_{a_\pi}^{b_\pi} J(u, x_1, \ldots, x_m) \1_{\{b_\pi \ge a_\pi\}} \,\rd u \right] 
- \int_{\bar{a}}^{\bar{b}} J(u,x_1, \ldots, x_m) \1_{\{\bar{b} \ge \bar{a}\}} \,\rd u \right| \notag \\
&\le 2\|J\|_\infty \cdot \EB_\pi \left[ |a_\pi - \bar{a}| + |b_\pi - \bar{b}| \right]. \label{eq:abstract-integral-bound}
\end{align}
where $\|J\|_\infty := \max\limits_{(u, x_1, \ldots, x_m) \in [0,1]^{m+1}} |J(u, x_1, \ldots, x_m)|$ denotes the bound on $J$.
\end{lem}

\begin{proof}[Proof of Lemma \ref{lem:integral-approximation-error}]
Since the values $\{x_\ell\}_{\ell=1}^m$ are fixed, define the simplified function $J(u) := J(u, x_1, \ldots, x_m)$.  
We decompose $J$ into its positive and negative parts:
\[
J = J_+ - J_-, \quad \text{where} \quad J_+(u) := \max\{J(u), 0\}, \quad J_-(u) := \max\{-J(u), 0\}.
\]
Both $J_+$ and $J_-$ are non-negative and 1-Lipschitz, with $\|J_+\|_\infty, \|J_-\|_\infty \le \|J\|_\infty$.

Note that for any non-negative function $f$, we have:
\[
\left( \int_a^b f(u)\, \rd u \right) \1_{\{b \ge a\}} = \left( \int_a^b f(u)\, \rd u \right)_+.
\]
Applying this to $J_+$ and $J_-$, we write:
\begin{align*}
&\int_{a_\pi}^{b_\pi} J(u)\, \rd u \cdot \1_{\{b_\pi \ge a_\pi\}} - \int_{\bar{a}}^{\bar{b}} J(u)\, \rd u \cdot \1_{\{\bar{b} \ge \bar{a}\}} \\
&= \left( \int_{a_\pi}^{b_\pi} J_+(u)\, \rd u \right)_+ - \left( \int_{\bar{a}}^{\bar{b}} J_+(u)\, \rd u \right)_+  - \left[ \left( \int_{a_\pi}^{b_\pi} J_-(u)\, \rd u \right)_+ - \left( \int_{\bar{a}}^{\bar{b}} J_-(u)\, \rd u \right)_+ \right].
\end{align*}

Applying the triangle inequality and the fact that $(\cdot)_+$ is 1-Lipschitz, we obtain:
\begin{align*}
&\left| \int_{a_\pi}^{b_\pi} J(u)\, \rd u \cdot \1_{\{b_\pi \ge a_\pi\}} - \int_{\bar{a}}^{\bar{b}} J(u)\, \rd u \cdot \1_{\{\bar{b} \ge \bar{a}\}} \right| \\
&\le \left| \left( \int_{a_\pi}^{b_\pi} J_+(u)\, \rd u \right) - \left( \int_{\bar{a}}^{\bar{b}} J_+(u)\, \rd u \right) \right| 
+ \left| \left( \int_{a_\pi}^{b_\pi} J_-(u)\, \rd u \right) - \left( \int_{\bar{a}}^{\bar{b}} J_-(u)\, \rd u \right) \right| \\
&\le 2\|J\|_\infty \cdot (|a_\pi - \bar{a}| + |b_\pi - \bar{b}|).
\end{align*}
Taking expectation over $\pi$ completes the proof.
\end{proof}

\section{Details of Simulation Study}
\label{append:simulation}

\paragraph*{Choice of pseudorandom variable.}
We use a context window of size $m=7$, so the pseudorandom variable $\xi_t = \AM(s_{(t-m):(t-1)}, \Key)$ depends on the preceding $m$ tokens. With such a relatively large $m$, nearly all pseudorandom collisions stem from our generation mechanism. In practice, at each step $t$, the hash function $\AM$ takes as input the $m$ most recent tokens concatenated with the key $\Key$, producing a hash value that serves as a random seed. This seed is then passed to a pseudorandom number generator, for which we use the PCG-64 generator \citep{o2014pcg}, the default in Python’s \textsf{NumPy} package \citep{harris2020array}.

\paragraph*{Computation of critical values.}
For the Gumbel-max watermark under score functions $h_{\mathrm{ars}}$ and $h_{\mathrm{log}}$, the sum of score values follows a gamma distribution. In this case, the critical values can be obtained directly from the gamma $(1-\alpha)$ quantile. For other score functions, the exact distribution of the score sum is generally unavailable, so we rely on Monte Carlo simulation.
Concretely, for each $n$, we generate $n$ i.i.d. samples of the corresponding pivotal statistic $Y$ from $\mu_0$ and compute $\sum_{\VM \in \Pi} h(Y_{\VM})$ for the score function $h$. This procedure is repeated 4000 times, and the empirical $(1-\alpha)$ quantile of these values serves as an estimate. 
 
\paragraph{Additional results.}
Figure~\ref{fig:simulation-appendix} reports Type~II errors for other values of $\Delta_{\max}$, showing patterns consistent with those in Section~\ref{sec:simulation}.

\begin{figure}[!t]
\centering
\vspace{-0.1in}
\includegraphics[width=\linewidth]{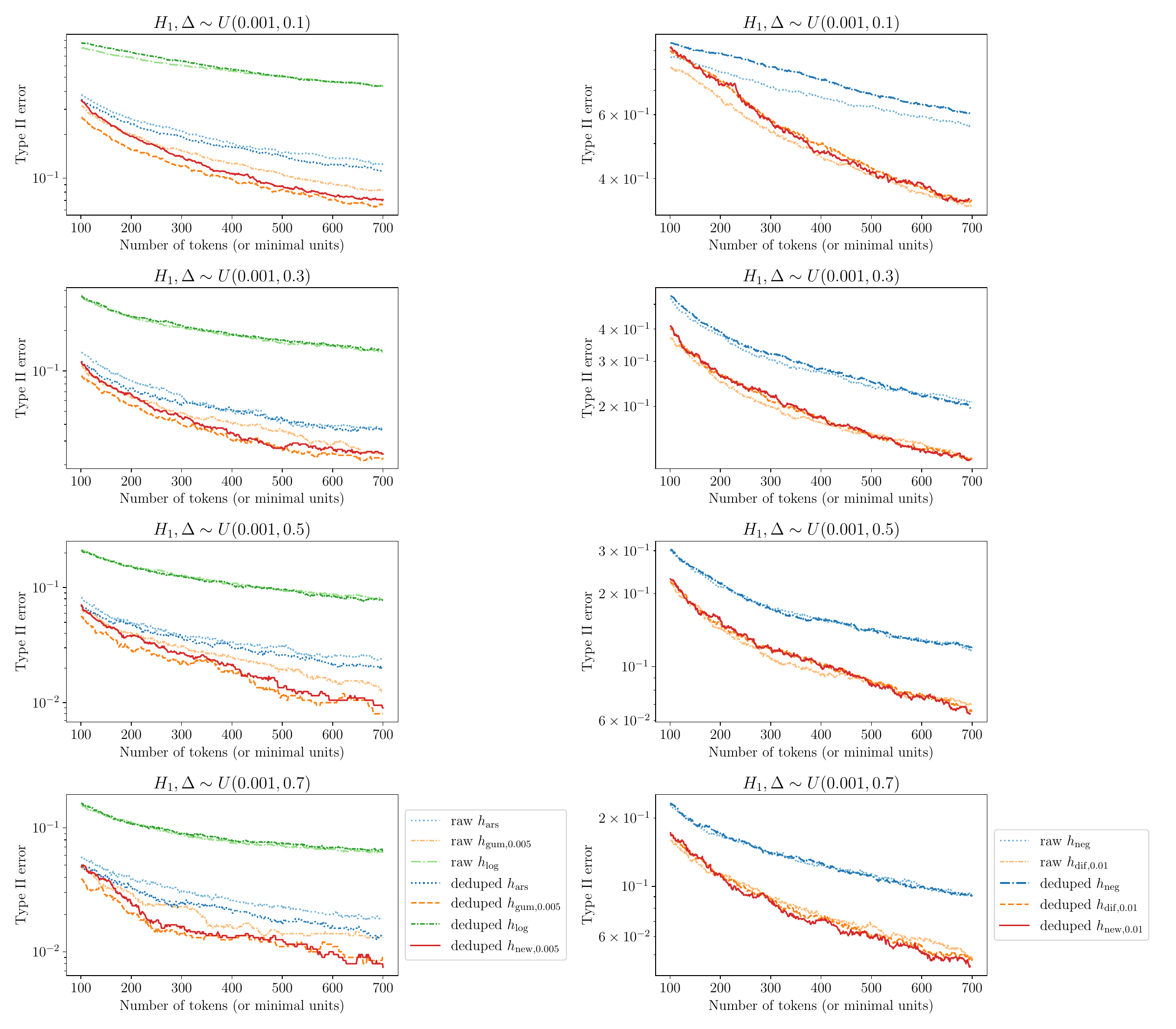}
\caption{Type~II errors on synthetic datasets for the Gumbel-max watermark (left) and the inverse-transform watermark (right), with results shown for $\Delta_{\max} \in \{0.1, 0.3, 0.5, 0.7\}$ from top to bottom.
}
\vspace{-0.1in}
\label{fig:simulation-appendix}
\end{figure}

\section{Details for Real-World Examples}
\label{append:LLM}

\paragraph{Detailed experimental setup.}
In our empirical analysis of the detection performance of different watermark detection methods, we focus on the OPT-1.3B model \citep{zhang2022opt}.
We evaluate Type I errors using 2000 human-written samples from the C4 news-like dataset \citep{raffel2020exploring}.
Specifically, for each human-written document in the dataset, we select it if and only if it contains at least 520 tokens, and we take the last 500 tokens as the initial prompt.
For each selected sample, we apply a hash function $\AM$ to compute the corresponding sequence of pseudorandom variables.
This procedure is repeated until we collect 2000 sequences, each containing 500 pivotal statistics.

To assess Type II errors, we randomly sample prompts from the same dataset.
We enforce a minimum prompt length of 50 tokens in all experiments and skip any document shorter than this threshold.
Each 50-token prompt is then fed into the model, which generates an additional 800 tokens.
Since 800 tokens are sufficiently long, we retain the generated text only if it contains at least 300 unique minimal units; otherwise, the generation is discarded.
Following this procedure, we collect 200 generated sequences, each consisting of at least 300 minimal units.

The temperature parameter controls the randomness of LLM generation.
Let $\bm{L} = (L_1, \ldots, L_{|\Voca|})$ denote the model’s logit vector over the vocabulary $\Voca$.
The temperature $\beta$ rescales this vector to obtain the token distribution $\bP$,
\[
P_{\token} = \frac{\exp(L_{\token} / \beta)}{\sum_{\token' \in \Voca} \exp(L_{\token'} / \beta)}.
\]
A smaller $\beta$ yields a more deterministic distribution.
To ensure a clear comparison across methods, we adopt relatively low temperatures: $\beta = 0.3$ for the Gumbel-max watermark and $\beta = 0.5$ for the inverse-transform watermark.
At higher temperatures (that is, more random generations), all detection methods tend to achieve nearly indistinguishable power within short text lengths.

\paragraph{Details of Figure \ref{fig:repetition}.}
We describe the experimental setup used to produce Figure \ref{fig:repetition}.
The left panel quantifies the proportion of token repetitions in both human-written and watermarked texts.
\begin{itemize}
    \item For the human-written case, we extract sentences from the C4 news-like dataset \citep{raffel2020exploring}. Each sentence is tokenized using the OPT-1.3B decoder, and we retain those with at least 200 tokens, collecting 1000 sequences in total. For a given text window size $m \in \{2, 3, \ldots, 10\}$, we compute the proportion of repeated tokens within each sequence and report the average repetition rate across all sequences.
    \item For the watermarked case, we sample 1000 prompts, each containing at least 50 tokens, and generate the following 200 tokens using the Gumbel-max watermark with temperature 1. When generating each text, we specify a window size $m$, which is used to compute the pseudorandom variables. We then measure the repetition rate for each generation and report the final value as the average across all 1000 samples.
\end{itemize}

The right panel of Figure \ref{fig:repetition} demonstrates that classic detection methods fail to control Type I error. To verify this, we generate 1000 sequences of 1000 tokens each using the OPT-1.3B model with temperature 0.1, without applying any watermarking. We then apply existing detection methods to these sequences and evaluate their empirical Type I error rates.

\end{appendix}